\documentclass{amsart}
\usepackage{amssymb}
\usepackage{mathrsfs}
\usepackage[all]{xy}

\RequirePackage{color}
\definecolor{myred}{rgb}{0.75,0,0}
\definecolor{mygreen}{rgb}{0,0.5,0}
\definecolor{myblue}{rgb}{0,0,0.65}

\RequirePackage{ifpdf}
\ifpdf
  \IfFileExists{pdfsync.sty}{\RequirePackage{pdfsync}}{}
  \RequirePackage[pdftex,
   colorlinks = true,
   urlcolor = myblue, % \href{...}{...} external (URL)
   citecolor = mygreen, % \cite{...}
   linkcolor = myblue, % \ref{...} and \pageref{...}
   pdfusetitle,
   pdfauthor = {{Pramod Achar, Anthony Henderson, Daniel Juteau and Simon Riche}},
   pdfpagemode=None,
   bookmarksopen=true]{hyperref}
\else
  \RequirePackage[hypertex]{hyperref}
\fi

\usepackage{fullpage}

\usepackage{graphicx}
\usepackage{caption}
\usepackage{subcaption}
\usepackage{longtable}
\usepackage{supertabular}

\newcommand{\C}{\mathbb{C}}

\newcommand{\Zl}{\mathbb{Z}_\ell}

\newcommand{\Qlb}{\overline{\mathbb{Q}}_\ell}
\newcommand{\Z}{\mathbb{Z}}

\newcommand{\K}{\mathbb{K}}
\newcommand{\F}{\mathbb{F}}
\renewcommand{\O}{\mathbb{O}}
\newcommand{\E}{\mathbb{E}}
\newcommand{\bk}{\Bbbk}

\newcommand{\fg}{\mathfrak{g}}

\newcommand{\fu}{\mathfrak{u}}

\newcommand{\cN}{\mathscr{N}}
\newcommand{\cO}{\mathscr{O}}
\newcommand{\cC}{\mathscr{C}}
\newcommand{\fL}{\mathfrak{L}}
\newcommand{\fN}{\mathfrak{N}}
\newcommand{\fM}{\mathfrak{M}}

\newcommand{\Db}{D^{\mathrm{b}}}
\newcommand{\Perv}{\mathsf{Perv}}

\newcommand{\bT}{\mathbb{T}}

\newcommand{\Res}{\mathbf{R}}
\newcommand{\Ind}{\mathbf{I}}

\newcommand{\cL}{\mathcal{L}}
\newcommand{\cM}{\mathcal{M}}
\newcommand{\IC}{\mathcal{IC}}
\newcommand{\cF}{\mathcal{F}}

\newcommand{\ubk}{\underline{\bk}}

\newcommand{\cE}{\mathcal{E}}

\newcommand{\cD}{\mathcal{D}}

\newcommand{\simto}{\xrightarrow{\sim}}
\DeclareMathOperator{\End}{End}
\DeclareMathOperator{\Hom}{Hom}

\DeclareMathOperator{\Irr}{Irr}
\DeclareMathOperator{\Ad}{Ad}
\DeclareMathOperator{\Lie}{Lie}

\newcommand{\cusp}{\mathrm{cusp}}

\newcommand{\reg}{\mathrm{reg}}
\newcommand{\mini}{\mathrm{min}}

\newcommand{\diag}{{\mathrm{diag}}}

\makeatletter
\def\lotimes{\@ifnextchar_{\@lotimessub}{\@lotimesnosub}}
\def\@lotimessub_#1{\mathchoice{\mathbin{\mathop{\otimes}^L}_{#1}}%
  {\otimes^L_{#1}}{\otimes^L_{#1}}{\otimes^L_{#1}}}
\def\@lotimesnosub{\mathbin{\mathop{\otimes}^L}}
\makeatother

\newcommand{\GL}{\mathrm{GL}}
\newcommand{\PGL}{\mathrm{PGL}}

\newcommand{\SL}{\mathrm{SL}}
\newcommand{\Sp}{\mathrm{Sp}}
\newcommand{\PSp}{\mathrm{PSp}}
\newcommand{\Spin}{\mathrm{Spin}}
\newcommand{\SO}{\mathrm{SO}}
\newcommand{\PSO}{\mathrm{PSO}}

\newcommand{\fS}{\mathfrak{S}}

\newtheorem*{thm*}{Theorem}
\numberwithin{equation}{section}
\numberwithin{table}{section}
\newtheorem{thm}{Theorem}[section]
\newtheorem{lem}[thm]{Lemma}
\newtheorem{prop}[thm]{Proposition}

\newtheorem{conj}[thm]{Conjecture}

\theoremstyle{definition}
\newtheorem{defn}[thm]{Definition}

\theoremstyle{remark}
\newtheorem{rmk}[thm]{Remark}

\DeclareFontFamily{U}{mathx}{\hyphenchar\font45}
\DeclareFontShape{U}{mathx}{m}{n}{
      <5> <6> <7> <8> <9> <10>
      <10.95> <12> <14.4> <17.28> <20.74> <24.88>
      mathx10
      }{}
\DeclareSymbolFont{mathx}{U}{mathx}{m}{n}
\DeclareFontSubstitution{U}{mathx}{m}{n}
\DeclareMathAccent{\widebar}{0}{mathx}{"73}

\title[Modular generalized Springer correspondence III]{Modular generalized Springer correspondence III: \\ exceptional groups}

\author{Pramod N. Achar}
\address{Department of Mathematics\\
  Louisiana State University\\
  Baton Rouge, LA 70803\\
  U.S.A.}
\email{pramod@math.lsu.edu}

\author{Anthony Henderson}
\address{School of Mathematics and Statistics\\
  University of Sydney, NSW 2006\\
  Australia}
\email{anthony.henderson@sydney.edu.au}

\author{Daniel Juteau}
\address{Laboratoire de Math\'ematiques Nicolas Oresme\\
  Universit\'e de Caen, BP 5186\\
  14032 Caen Cedex\\ 
  France}
\email{daniel.juteau@unicaen.fr}

\author{Simon Riche}
\address{Universit{\'e} Blaise Pascal - Clermont-Ferrand II, Laboratoire de Math{\'e}matiques, CNRS, UMR 6620, Campus universitaire des C{\'e}zeaux, F-63177 Aubi{\`e}re Cedex, France
}
\email{simon.riche@math.univ-bpclermont.fr}


\thanks{P.A. was supported by NSF Grant No.~DMS-1001594 and NSA Grant Nos.~H98230-14-1-0117 and~H98230-15-1-0175.  A.H. was supported by ARC Future Fellowship Grant No.~FT110100504. D.J. and S.R. were supported by ANR Grant 
No.~ANR-13-BS01-0001-01. 
}

\begin{document}

\begin{abstract}
We complete the construction of the modular generalized Springer correspondence for an arbitrary connected reductive group, with a uniform proof of the disjointness of induction series that avoids the case-by-case arguments for classical groups used in previous papers in the series. We show that the induction series containing the trivial local system on the regular nilpotent orbit is determined by the Sylow subgroups of the Weyl group. Under some assumptions, we give an algorithm for determining the induction series associated to the minimal cuspidal datum with a given central character. We also provide tables and other information on the modular generalized Springer correspondence for quasi-simple groups of exceptional type, including a complete classification of cuspidal pairs in the case of good characteristic, and a full determination of the correspondence in type $G_2$.
\end{abstract}

\maketitle

%%%%%%%%%%%%%%%%%%%%%%%%%%%%%%%%%%%%%%%%%%%%%%%%%%%%%%%%%
\section{Introduction}
\label{sec:intro}
%%%%%%%%%%%%%%%%%%%%%%%%%%%%%%%%%%%%%%%%%%%%%%%%%%%%%%%%%

%----------------------------------------------------------------------
\subsection{Summary}
%----------------------------------------------------------------------

This paper is the culmination of a series~\cite{genspring1,genspring2} in which our aim has been to construct and describe a modular generalized Springer correspondence for connected reductive groups. This requires us to prove analogues, for sheaves with modular coefficients, of the fundamental results of Lusztig (especially those in~\cite{lusztig}) on the generalized Springer correspondence for $\Qlb$-sheaves. We regard this as a first step towards a theory of modular character sheaves, which may offer new insights into the modular representation theory of finite groups of Lie type.

In~\cite{genspring1} we considered the group $\GL(n)$, and in~\cite{genspring2} we considered classical groups in general. The subtitle of this third part is `exceptional groups', to emphasize the cases that were not previously covered; however, many of the results in this part are case-independent, and some provide new proofs of results in the previous parts.

%----------------------------------------------------------------------
\subsection{Formulation of the modular generalized Springer correspondence}
%----------------------------------------------------------------------

Recall the set-up from the previous parts: $G$ denotes a connected reductive algebraic group over $\C$, and we consider the abelian category $\Perv_G(\cN_G,\bk)$ of $G$-equivariant perverse sheaves on the nilpotent cone $\cN_G$ of $G$ (in the strong topology) with coefficients in a field $\bk$ of characteristic $\ell$.  (If $\bk$ is an algebraic extension of $\F_\ell$, one could also work in the \'etale topology, with $G$ and $\cN_G$ defined over any algebraically closed field whose characteristic is good and different from $\ell$.)  The isomorphism classes of simple objects in $\Perv_G(\cN_G,\bk)$ are in bijection with the finite set $\fN_{G,\bk}$ of pairs $(\cO,\cE)$ where $\cO\subset\cN_G$ is a nilpotent orbit and $\cE$ runs over the irreducible $G$-equivariant $\bk$-local systems on $\cO$ (taken up to isomorphism). For $(\cO,\cE)\in\fN_{G,\bk}$, the corresponding simple perverse sheaf is the intersection cohomology complex $\IC(\cO,\cE)$. 

For a parabolic subgroup $P\subset G$ with Levi factor $L$, we have an induction functor
\[
\Ind_{L \subset P}^G : \Perv_{L}(\cN_{L},\bk) \to \Perv_G(\cN_G,\bk),
\]
see~\cite[\S 2.1]{genspring1}.
A simple object in $\Perv_G(\cN_G,\bk)$ is said to be \emph{cuspidal} if it does not occur as a quotient of any induced object $\Ind_{L \subset P}^G(\cF)$ where $L\neq G$. We write $\fN_{G,\bk}^\cusp\subset\fN_{G,\bk}$ for the corresponding set of \emph{cuspidal pairs}. (See~\cite[\S 2.2]{genspring1} for a comparison with Lusztig's definition of cuspidal pairs.)

A new piece of terminology will be convenient: a \emph{cuspidal datum} for $G$ is a triple $(L,\cO_L,\cE_L)$ where $L\subset G$ is a Levi subgroup as above and $(\cO_L,\cE_L)\in\fN_{L,\bk}^\cusp$. There is a $G$-action on the set of cuspidal data defined by the rule $g\cdot(L,\cO_L,\cE_L)=(gLg^{-1},g\cdot\cO_L,\Ad(g^{-1})^*\cE_L)$, and we choose a set $\fM_{G,\bk}$ of representatives for the $G$-orbits of cuspidal data. 

For any cuspidal datum $(L,\cO_L,\cE_L)$, the isomorphism class of the induced perverse sheaf $\Ind_{L \subset P}^G(\IC(\cO_L,\cE_L))$ is independent of the parabolic subgroup $P$ with Levi factor $L$, and depends only on the $G$-orbit of $(L,\cO_L,\cE_L)$, as explained in~\cite[\S 2.2]{genspring2}. We call the set of isomorphism classes of simple quotients of $\Ind_{L \subset P}^G(\IC(\cO_L,\cE_L))$, or the corresponding set of pairs $\fN_{G,\bk}^{(L,\cO_L,\cE_L)}\subset\fN_{G,\bk}$, the \emph{induction series} associated to $(L,\cO_L,\cE_L)$. It is a formal consequence of the definitions that the whole set $\fN_{G,\bk}$ is the union of the various induction series $\fN_{G,\bk}^{(L,\cO_L,\cE_L)}$ as $(L,\cO_L,\cE_L)$ runs over $\fM_{G,\bk}$, see~\cite[Corollary 2.7]{genspring1}.     

The main result of the first part of this paper is the following generalization of~\cite[Theorem 1.1]{genspring2} (which assumed $G$ to be a classical group). Lusztig proved the analogous result for $\bk=\Qlb$ in~\cite{lusztig}.

\begin{thm}
\label{thm:mgsc-intro}
Assume $\bk$ is big enough for $G$ in the sense of~\eqref{eqn:definitely-big-enough} below. 
Then we have a disjoint union
\begin{equation}
\label{eqn:disjointness}
\fN_{G,\bk} = \bigsqcup_{(L,\cO_L,\cE_L)\in\fM_{G,\bk}}\fN_{G,\bk}^{(L,\cO_L,\cE_L)},
\end{equation}
and for any $(L,\cO_L,\cE_L)\in\fM_{G,\bk}$ we have a canonical bijection
\begin{equation}
\label{eqn:bijection}
\fN_{G,\bk}^{(L,\cO_L,\cE_L)}\longleftrightarrow\Irr(\bk[N_G(L)/L]),
\end{equation}
where $\Irr(\bk[N_G(L)/L])$ denotes the set of isomorphism classes of irreducible $\bk$-representations of $N_G(L)/L$. 
Hence we obtain a bijection
\begin{equation}
\label{eqn:mgsc}
\fN_{G,\bk} \longleftrightarrow
\bigsqcup_{(L,\cO_L,\cE_L)\in\fM_{G,\bk}}\Irr(\bk[N_G(L)/L]), 
\end{equation}
which we call the \emph{modular generalized Springer correspondence} for $G$. 
\end{thm}

In this statement we say that $\bk$ is \emph{big enough} for $G$ if it satisfies:
\begin{equation} \label{eqn:definitely-big-enough}
\begin{array}{c}
\text{for every Levi subgroup $L$ of $G$ and pair $(\cO_L,\cE_L)\in\fN_{L,\bk}$,}\\
\text{the irreducible $L$-equivariant local system $\cE_L$ is absolutely irreducible.}
\end{array}
\end{equation}
(For our proof of Theorem~\ref{thm:mgsc-intro} it would be enough to know this for cuspidal data $(L,\cO_L,\cE_L)$, but in practice we cannot classify cuspidal pairs until after we have proved Theorem~\ref{thm:mgsc-intro}.)
Note that~\eqref{eqn:definitely-big-enough} is a weaker condition, in general, than the condition imposed, in the case of classical groups, in~\cite{genspring2}; we hope this slight conflict of terminology will not cause any confusion.
The condition~\eqref{eqn:definitely-big-enough} is equivalent to requiring that $\bk$ be a splitting field for each of the finite groups $A_L(x):= L_x/L_x^\circ$ where $L$ is a Levi subgroup of $G$ and $x\in\cN_L$. (Here, $L_x$ denotes the stabilizer of $x$ in $L$, and $L_x^\circ \subset L_x$ is its identity component.) In Proposition~\ref{prop:big-enough}, we will use the known description of these groups (see e.g.~\cite{cm}) to make~\eqref{eqn:definitely-big-enough} explicit in important cases. In particular, if $G$ has connected centre then every field $\bk$ is big enough for $G$.

\begin{rmk}
Assuming that $\bk$ is big enough for $G$, the modular generalized Springer correspondence~\eqref{eqn:mgsc} depends only on the characteristic $\ell$ of $\bk$, in a sense to be made precise in Lemma~\ref{lem:splitting}\eqref{it:indep-of-k}.
\end{rmk}

%----------------------------------------------------------------------
\subsection{Overview of the proof of Theorem~\ref{thm:mgsc-intro}}
%----------------------------------------------------------------------

A uniform construction of the bijections~\eqref{eqn:bijection} was provided in~\cite[Lemma 2.1 and Theorem 3.1]{genspring2}, but to apply the latter theorem we need the two assumptions~\eqref{eqn:definitely-big-enough} and
\begin{equation}
\label{eqn:precision}
\text{the action of }N_G(L)\text{ on }\fN_{L,\bk}\text{ fixes every cuspidal pair.}
\end{equation} 
Recall that in a cuspidal pair $(\cO_L,\cE_L)$, the $L$-orbit $\cO_L$ is distinguished~\cite[Proposition 2.6]{genspring2}, and the action of $N_G(L)$ preserves each distinguished nilpotent orbit for $L$~\cite[Lemma 2.9]{genspring2}. So the group $N_G(L,\cO_L)$ occurring in~\cite[Theorem 3.1]{genspring2} does equal $N_G(L)$. (In fact,~\cite[Theorem 3.1]{genspring2} has three assumptions, not two: the first is automatic for distinguished orbits by~\cite[Lemma 3.11]{genspring2}.)

Therefore, two things remain to be proved in order to obtain Theorem~\ref{thm:mgsc-intro}: the disjointness of the union in~\eqref{eqn:disjointness}, and the statement~\eqref{eqn:precision}. In the case of classical groups, treated in~\cite{genspring2}, we deduced both of these from explicit knowledge of the cuspidal pairs, obtained by induction on the rank within each type. In Theorem~\ref{thm:disjointness} below, we give an alternative uniform proof of disjointness, which relies on Theorem~\ref{thm:mackey}, a Mackey formula for our induction and restriction functors. In Proposition~\ref{prop:horrible}, we prove~\eqref{eqn:precision} by showing the stronger statement that $N_G(L)$ fixes every pair $(\cO_L, \cE_L) \in \fN_{L,\bk}$ where $\cO_L$ is distinguished. This proof
uses some case-by-case checking, but no explicit knowledge of cuspidal pairs.

\begin{rmk}
The analogue of~\eqref{eqn:precision} in Lusztig's setting is~\cite[Theorem 9.2(b)]{lusztig}; his proof does not generalize to our setting. Indeed, the analogue of~\cite[Theorem 9.2(a)]{lusztig} can fail, i.e.\ a Levi subgroup $L$ supporting a cuspidal pair need not be self-opposed in $G$, as already seen in~\cite{genspring1,genspring2}. In fact, the group $N_G(L)/L$ need not even be a reflection group; see Remark~\ref{rmk:reflection-group} for further discussion.    
\end{rmk}

\begin{rmk} \label{rmk:representatives}
As in~\cite{genspring2}, let $\fL$ denote a set of representatives for the $G$-conjugacy classes of Levi subgroups of $G$. If we assume (as we may) that the first component of any triple $(L,\cO_L,\cE_L) \in \fM_{G,\bk}$ belongs to $\fL$, we have an obvious surjective map
\[
\bigsqcup_{L\in\fL}\fN_{L,\bk}^\cusp \to \fM_{G,\bk},
\] 
and~\eqref{eqn:precision} is equivalent to the statement that this map is bijective. Hence we can re-state~\eqref{eqn:mgsc} in the form
\begin{equation} \label{eqn:mgsc-variant}
\fN_{G,\bk} \longleftrightarrow \bigsqcup_{L\in\fL} \bigsqcup_{(\cO_L,\cE_L)\in\fN_{L,\bk}^\cusp} \Irr(\bk[N_G(L)/L]),
\end{equation}
which is how it appeared in~\cite[Theorem 1.1]{genspring2}.
\end{rmk} 

%----------------------------------------------------------------------
\subsection{Further general results}
%----------------------------------------------------------------------

After Sections~\ref{sec:mackey} and~\ref{sec:case-by-case} which complete the proof of Theorem~\ref{thm:mgsc-intro}, we prove further results about a general connected reductive group $G$ in Sections~\ref{sec:regular} and~\ref{sec:basic-sets}. 

In Theorem~\ref{thm:regular-series} we determine the cuspidal datum $(L,\cO_L, \cE_L)$ such that the corresponding induction series $\fN_{G,\bk}^{(L,\cO_L, \cE_L)}$ contains the pair $(\cO_\reg, \underline{\bk})$, where $\cO_\reg \subset \cN_G$ is the regular orbit. It turns out that the Levi subgroup $L$ (determined up to $G$-conjugacy) is the one which is minimal such that its Weyl group $W_L$ contains an $\ell$-Sylow subgroup of the Weyl group $W$ of $G$. In particular, the pair $(\cO_\reg, \underline{\bk})$ is cuspidal if and only if no proper parabolic subgroup of $W$ contains an $\ell$-Sylow subgroup; this general criterion provides a new proof of the classification of cuspidal pairs for the general linear group~\cite{genspring1}, see Remark~\ref{rk:cuspidal-typeA}.

Section~\ref{sec:basic-sets} generalizes a construction of the third author. In~\cite[Section~5]{juteau} it is shown that the (non-generalized) modular Springer correspondence allows one to construct, for any $\ell$, an $\ell$-modular `basic set datum' (a variation on the classical notion of `basic set') for $W$. This construction is then used in~\cite[Section~9]{juteau} to provide an algorithm to explicitly determine the modular Springer correspondence, i.e.~the bijection~\eqref{eqn:bijection} in the case of the `principal' cuspidal datum $(T, \{0\}, \underline{\bk})$ (where $T \subset G$ is a maximal torus), provided that the decomposition matrix for $W$ is known (which is the case for exceptional groups). Theorem~\ref{thm:basic-set-chi} generalizes this algorithm to some non-principal cuspidal data $(L,\cO_L, \cE_L)$, namely those which are minimal with a given central character. (See~\S\ref{ss:normalizers-distinguished} for a discussion of central characters.  For technical reasons, Theorem~\ref{thm:basic-set-chi} excludes the Spin groups.)

%----------------------------------------------------------------------
\subsection{Classification of cuspidal pairs and determination of the modular generalized Springer correspondence for exceptional groups}
%----------------------------------------------------------------------

The remainder of the paper focuses on the exceptional groups. 

We first consider the problem of classifying cuspidal pairs, which we solved for groups of classical type in~\cite{genspring2}. In Section~\ref{sec:reduction} we explain how to determine the number of cuspidal pairs for an exceptional group in any characteristic; the result is summarized in Table~\ref{tab:exc-cuspidal-pairs}. (In this table, $\chi$ is the central character, and a symbol `$-$' means a case that does not occur.) 

\begin{table}
\[
\begin{array}{|c||c|c|c|c|}
\hline
& \ell=2 & \ell=3 & \ell=5 & \ell\geq 7\\
\hline\hline
G_2 & 2 & 2 & 1 & 1\\
\hline
F_4 & 4 & 3 & 1 & 1\\
\hline
E_6, \chi = 1 & 0 & 3 & 0 & 0\\
\phantom{E_6, }\chi \neq 1 & 2 & - & 1 & 1\\
\hline
E_7, \chi = 1 & 6 & 0 & 0 & 0\\
\phantom{E_7, }\chi \neq 1 & - & 3 & 1 & 1\\
\hline
E_8 & 10 & 8 & 5 & 1\\
\hline
\end{array}
\]
\caption{Count of cuspidal pairs for simply connected quasi-simple groups of exceptional type}\label{tab:exc-cuspidal-pairs}
\end{table} 

Our analysis of the exceptional groups completes the proof of the following general result, which says roughly that when $\ell$ is a good prime, the classification of cuspidal pairs behaves in the same way as in Lusztig's setting (see~\cite[Introduction]{lusztig}). See \S\ref{ss:decomposition-numbers} for the concept of `modular reduction' involved here.
\begin{thm}
\label{thm:cuspidal-facts}
Suppose that $\ell$ is a good prime for $G$ and that $\bk$ is big enough for $G$. 
\begin{enumerate}
\item If $G$ is semisimple and simply connected, the cuspidal pairs for $G$ over $\bk$ are exactly the modular reductions of the cuspidal pairs for $G$ over $\Qlb$.
\item $G$ has at most one cuspidal pair over $\bk$ of each central character.
\end{enumerate}
\end{thm}

\begin{proof}
By the principles of~\cite[\S 5.3]{genspring2}, we may assume that $G$ is simply connected and quasi-simple for both parts. For $G$ of type $A$ (i.e.\ $G=\SL(n)$), when the condition that $\ell$ is good is vacuous, the result was shown in~\cite[Theorem 6.3]{genspring2}. (Note that part (1) fails for $\GL(n)$ and $\PGL(n)$; see~\cite[Theorem 3.1]{genspring1}.) For $G$ of types $B,C,D$, when $\ell$ being good rules out $\ell=2$, the result was shown in~\cite[Theorems~7.2, 8.3 and~8.4]{genspring2}. For $G$ of exceptional type, the result is Proposition~\ref{prop:0-cusp-not-A}.
\end{proof}

When $\ell$ is a bad prime, we can determine all the cuspidal pairs in type $G_2$ but sometimes not in the other exceptional types; see \S\ref{ss:determining-cuspidal} for further discussion. Tables of cuspidal data for the exceptional groups, including those cuspidal pairs that we know, are given in Appendix~\ref{sec:tables}.
 
Finally, in Section~\ref{sec:series}, we present partial results on the explicit determination of the modular generalized Springer correspondence for exceptional groups. We show in \S\ref{ss:easy-mgsc} that when $\ell$ does not divide the order of the Weyl group $W$, the modular generalized Springer correspondence coincides with Lusztig's generalized Springer correspondence for $\Qlb$-sheaves. In \S\ref{ss:good-mgsc} we consider the case of good characteristic, where we can determine all the induction series but not all the bijections~\eqref{eqn:bijection}. In \S\ref{ss:G2-mgsc} we completely describe the modular generalized Springer correspondence for $G_2$ in all characteristics, and in \S\ref{ss:E6-mgsc} we give an almost complete description of the modular generalized Springer correspondence for $E_6$ in characteristic $3$.
 
%----------------------------------------------------------------------
\subsection{Acknowledgements}
%----------------------------------------------------------------------

We are grateful to Jean Michel for implementing many functions on unipotent classes, including
the generalized Springer correspondence (with characteristic zero coefficients), in the development version of the GAP3 Chevie package~\cite{chevie}.  We would also like to thank the anonymous referee for a careful reading and detailed comments on the paper.

%%%%%%%%%%%%%%%%%%%%%%%%%%%
%\tableofcontents
%%%%%%%%%%%%%%%%%%%%%%%%%%%

%%%%%%%%%%%%%%%%%%%%%%%%%%%%%%%%%%%%%%%%%%%%%%%%%%%%%%%%%%%%%%%%%%%%%%%
\section{A Mackey formula}
\label{sec:mackey}
%%%%%%%%%%%%%%%%%%%%%%%%%%%%%%%%%%%%%%%%%%%%%%%%%%%%%%%%%%%%%%%%%%%%%%%

We continue with the notation of the introduction and of~\cite{genspring1,genspring2}, with $G$ being an arbitrary connected reductive group over $\C$. Recall from~\cite[\S 2.1]{genspring1} that the induction functor $\Ind_{L\subset P}^G$ has left and right adjoints denoted ${}'\Res_{L\subset P}^G$ and $\Res_{L\subset P}^G$ respectively.
The aim of this section is to prove Theorem~\ref{thm:mackey} below, a Mackey formula for these induction and restriction functors, and to deduce the disjointness of induction series asserted in~\eqref{eqn:disjointness}.

%------------------------------------------------------------------------------
\subsection{Statement of the Mackey formula}
%------------------------------------------------------------------------------

We first need to recall a result about double cosets. Here and subsequently we use the notation ${}^g L$ as an abbreviation for $gLg^{-1}$.

\begin{lem} \cite[Lemma 5.6(i)]{dignemichel}
\label{lem:double-cosets}
Let $P,Q$ be parabolic subgroups of $G$ with Levi factors $L,M$ respectively. Define 
\[ \Sigma(M,L)=\{g\in G\,|\, M\cap {}^g L\text{ contains a maximal torus of }G\}. \]
Then $\Sigma(M,L)$ is a union of finitely many $M$--$L$ double cosets, and the inclusion of $\Sigma(M,L)$ in $G$ induces a bijection
\[
M\backslash\Sigma(M,L)/L \longleftrightarrow Q\backslash G/P.
\]
\end{lem}

Given $L\subset P$, $M\subset Q$ as in Lemma~\ref{lem:double-cosets}, we let $g_1,g_2,\cdots,g_s$ be a set of representatives for the $M$--$L$ double cosets in $\Sigma(M,L)$, ordered in such a way that if $Qg_iP\subseteq\overline{Qg_jP}$, then $i\leq j$. Since $g_i\in\Sigma(M,L)$, the group $M\cap{}^{g_i} L$ is simultaneously a Levi subgroup of $M$ and of ${}^{g_i} L$; more precisely, it is a Levi factor of the parabolic subgroup $M\cap{}^{g_i}P$ of $M$ and a Levi factor of the parabolic subgroup $Q\cap{}^{g_i}L$ of ${}^{g_i}L$.  

\begin{thm}
\label{thm:mackey}
Let $L\subset P$, $M\subset Q$, and $g_1,g_2,\cdots,g_s$ be as above. Let $\cF\in\Perv_L(\cN_L,\bk)$. Then in $\Perv_M(\cN_M,\bk)$ we have a filtration
\[
{}'\Res_{M\subset Q}^G\bigl(\Ind_{L\subset P}^G(\cF)\bigr)=\cF_0\supset\cF_1\supset\cF_2\supset\cdots\supset\cF_s=0,
\]
in which the successive quotients have the form
\[
\cF_{i-1}/\cF_i\cong \Ind_{M\cap{}^{g_i}L \subset M\cap{}^{g_i} P}^M\bigl({}'\Res_{M\cap{}^{g_i} L\subset Q\cap{}^{g_i} L}^{{}^{g_i} L}(\Ad(g_i^{-1})^*\cF)\bigr),\quad\text{for }i=1,\cdots,s.
\] 
\end{thm}

\begin{rmk}
We shall not need it, but one can immediately deduce a corresponding statement involving the other restriction functor $\Res_{M\subset Q}^G$ by applying Verdier duality.
\end{rmk}

\begin{rmk}
The prototypical geometric Mackey formula is Lusztig's result for character sheaves~\cite[Proposition 15.2]{charsh3}, which, like the corresponding results for $\Qlb$-representations of finite groups of Lie type (see~\cite[Theorems 5.1 and 11.13]{dignemichel}), has a direct sum rather than a filtration. Lusztig's proof in~\cite[Section 15]{charsh3} uses Frobenius traces and is therefore unsuited to the setting of modular perverse sheaves. Instead, our proof is modelled on the alternative proof of the Mackey formula for character sheaves given by Mars and Springer in~\cite[\S 10.1]{marsspringer}, which in turn follows the pattern of~\cite[Section 3]{charsh1}. Note that we have no analogue of the final part of their proof, in which they use purity considerations to deduce a direct sum. 
\end{rmk}

%------------------------------------------------------------------------------
\subsection{Disjointness of induction series}
%------------------------------------------------------------------------------

Before starting the proof of Theorem~\ref{thm:mackey}, we highlight its most important consequence, which is the disjointness statement~\eqref{eqn:disjointness} from the introduction. The proof is parallel to the analogous result on Harish-Chandra series for finite groups of Lie type (see~\cite[Theorem~2.4]{ghm2}).

\begin{thm} \label{thm:disjointness}
Let $(L,\cO_L,\cE_L)$ and $(M,\cO_M,\cE_M)$ be cuspidal data for $G$ that are not in the same $G$-orbit. Then the induction series $\fN_{G,\bk}^{(L,\cO_L,\cE_L)}$ and $\fN_{G,\bk}^{(M,\cO_M,\cE_M)}$ are disjoint.
\end{thm}

\begin{proof}
We prove the result in the contrapositive form, assuming that
\begin{equation} \label{eqn:intersection}
\fN_{G,\bk}^{(L,\cO_L,\cE_L)}\cap\fN_{G,\bk}^{(M,\cO_M,\cE_M)}\neq\emptyset.
\end{equation} 
By~\cite[Lemma 2.3]{genspring2}, the induction series associated to $(M,\cO_M,\cE_M)$, defined initially as the set of isomorphism classes of simple quotients of $\Ind_{M\subset Q}^G(\IC(\cO_M,\cE_M))$, equals the set of isomorphism classes of simple subobjects of $\Ind_{M\subset Q}^G(\IC(\cO_M,\cE_M))$. Thus our assumption~\eqref{eqn:intersection} implies that
\begin{equation} \label{eqn:hom1}
\Hom_{\Perv_G(\cN_G,\bk)}\bigl(\Ind_{L\subset P}^G(\IC(\cO_L,\cE_L)),\Ind_{M\subset Q}^G(\IC(\cO_M,\cE_M))\bigr)\neq 0.
\end{equation}
This in turn is equivalent (by adjunction) to
\begin{equation} \label{eqn:hom2}
\Hom_{\Perv_M(\cN_M,\bk)}\bigl({}'\Res_{M\subset Q}^G\bigl(\Ind_{L\subset P}^G(\IC(\cO_L,\cE_L))\bigr),\IC(\cO_M,\cE_M)\bigr)\neq 0.
\end{equation}

Now we apply Theorem~\ref{thm:mackey} with $\cF=\IC(\cO_L,\cE_L)$. The cuspidality of $\cF$ implies that of $\Ad(g_i^{-1})^*\cF$, so the successive quotients $\cF_{i-1}/\cF_i$ can only be nonzero when $M\cap{}^{g_i} L={}^{g_i}L$, or in other words ${}^{g_i}L\subset M$. Thus~\eqref{eqn:hom2} implies that for some such $g_i$ we have
\begin{equation} \label{eqn:hom3}
\Hom\bigl(\Ind_{{}^{g_i}L \subset M\cap{}^{g_i} P}^M(\Ad(g_i^{-1})^*\IC(\cO_L,\cE_L)),\IC(\cO_M,\cE_M)\bigr)\neq 0.
\end{equation}
Now the cuspidality of $\IC(\cO_M,\cE_M)$ forces ${}^{g_i}L=M$ and $\Ad(g_i^{-1})^*\IC(\cO_L,\cE_L)\cong\IC(\cO_M,\cE_M)$, showing that $(L,\cO_L,\cE_L)$ and $(M,\cO_M,\cE_M)$ are in the same $G$-orbit as desired.
\end{proof}

\begin{rmk}
The corresponding result for Lusztig's generalized Springer correspondence~\cite[Proposition 6.3]{lusztig} was proved without using a Mackey formula, by an argument specific to the case of $\Qlb$-sheaves.
\end{rmk}

Another useful consequence of Theorem~\ref{thm:mackey} is a generalization of our earlier result~\cite[Proposition 2.6]{genspring2}, stating that cuspidal pairs must be supported on distinguished orbits.

\begin{prop} \label{prop:levi-rule}
If $(L,\cO_L,\cE_L)\in\fM_{G,\bk}$ and $(\cO,\cE)\in\fN_{G,\bk}^{(L,\cO_L,\cE_L)}$, then any Levi subgroup $M$ of $G$ whose Lie algebra intersects $\cO$ must contain a $G$-conjugate of $L$. In particular, this is true when $M$ is the Levi subgroup in the Bala--Carter label of $\cO$. 
\end{prop}

\begin{proof}
Suppose that $M$ is a Levi subgroup of $G$ whose Lie algebra intersects $\cO$, and let $Q$ be a parabolic subgroup of $G$ with Levi factor $M$. Then~\cite[Proposition 2.7]{genspring2} shows that ${}'\Res_{M\subset Q}^G(\IC(\cO,\cE))\neq 0$, which implies that ${}'\Res_{M\subset Q}^G(\Ind_{L\subset P}^G(\IC(\cO_L,\cE_L)))\neq 0$, since ${}'\Res_{M\subset Q}^G$ is exact. By Theorem~\ref{thm:mackey} and the cuspidality of $\IC(\cO_L,\cE_L)$, we conclude that there is some $g\in G$ such that $M\cap{}^g L={}^g L$, i.e.\ $M\supset {}^g L$.
\end{proof}

%------------------------------------------------------------------------------
\subsection{Proof of Theorem~\ref{thm:mackey}}
%------------------------------------------------------------------------------

In the proof of Theorem~\ref{thm:mackey} it will be convenient, especially for drawing parallels with~\cite[\S 10.1]{marsspringer}, to regard $\Perv_L(\cN_L,\bk)$ as a subcategory of the non-equivariant derived category $\Db(\cN_L,\bk)$ rather than as a subcategory of $\Db_L(\cN_L,\bk)$, and likewise for the other categories of perverse sheaves involved. (Recall that the forgetful functor $\Db_L(\cN_L,\bk)\to\Db(\cN_L,\bk)$ is fully faithful on $\Perv_L(\cN_L,\bk)$.) In this setting, the definition of $\Ind_{L\subset P}^G$ can be reformulated in the terms familiar from Lusztig's work (e.g.~\cite{charsh1}), using the diagram
\[
\xymatrix@1{
\cN_L&G\times\cN_P\ar_-{\alpha}[l]\ar^-{\beta''}[r] & G\times^{P}\cN_P\ar^-{\beta'}[r] & \cN_G,
}
\]
where $\alpha(g,x)=p_{L\subset P}(x)$ (with $p_{L \subset P}: \cN_P \to \cN_L$ the natural projection), $\beta''$ is the quotient projection for the action of $P$ (a principal $P$-bundle), and $\beta'$ is defined so that $\beta=\beta'\beta''$ is the action map $G\times\cN_P\to\cN_G \colon (g,x) \mapsto g\cdot x$. For any $\cF\in\Perv_L(\cN_L,\bk)$, the inverse image $\alpha^*\cF[\dim G +\dim U_P]$ belongs to $\Perv_P(G\times\cN_P,\bk)$, so there is a unique (up to isomorphism) object $(\beta'')_\flat\alpha^*\cF[2\dim U_P]\in\Perv(G\times^P\cN_P,\bk)$ such that $(\beta'')^*((\beta'')_\flat\alpha^*\cF[2\dim U_P])[\dim P]\cong \alpha^*\cF[\dim G +\dim U_P]$. (The notation $(\beta'')_\flat$ is meant to suggest a right inverse of the functor $(\beta'')^*$, but this right inverse is not defined on the whole of $\Db(G\times\cN_P,\bk)$, only on objects that are shifts of a $P$-equivariant perverse sheaf.) By~\cite[Lemma 2.14]{genspring1} we have
\begin{equation} \label{eqn:induction}
\Ind_{L\subset P}^G(\cF)\cong(\beta')_!(\beta'')_\flat\alpha^*\cF[2\dim U_P].
\end{equation}   
In the remainder of the argument we will encounter many other expressions of the form $(\cdot)_!(\cdot)_\flat(\cdot)^*[\cdots]$, and we will omit the straightforward equivariance checks that ensure that the $(\cdot)_\flat$ operation is defined.  

Form the variety
\[
Z:=\{(g,x)\in G\times\cN_P\,|\,g\cdot x \in\cN_Q\},
\]
and let $p_1:Z\to\cN_L$, $p_2:Z\to\cN_M$ be defined by
\[
p_1(g,x)=p_{L\subset P}(x),\quad p_2(g,x)=p_{M\subset Q}(g\cdot x).
\]
As with $\beta$ above, we factor $p_2$ as $p_2'p_2''$ where $p_2''$ is the quotient projection for the action of $P$ (a principal $P$-bundle).

\begin{lem} \label{lem:easy}
For any $\cF\in\Perv_L(\cN_L,\bk)$, we have an isomorphism in $\Db(\cN_M,\bk)$:
\[
{}'\Res_{M\subset Q}^G(\Ind_{L\subset P}^G(\cF))
\cong
(p_2')_!(p_2'')_\flat(p_1)^*\cF[2\dim U_P].
\]
\end{lem}

\begin{proof}
This is the analogue of~\cite[Equation (1) in \S 10.1]{marsspringer}, with groups replaced by nilpotent cones. It follows easily from~\eqref{eqn:induction}, the definition of ${}'\Res_{M\subset Q}^G$, and the base change isomorphism.
\end{proof}

Now recall our set of representatives $g_1,\cdots,g_s$ for the $Q$--$P$ double cosets in $G$, which we have ordered in such a way that 
\[ G_i:=\bigcup_{j=1}^i Qg_jP \] 
is a closed subset for all $i\in\{0,1,\cdots,s\}$. We obtain closed subsets $Z_i:=\{(g,x)\in Z\,|\,g\in G_i\}$ of $Z$. Note that $G_s=G$ and $Z_s=Z$, while $G_0=Z_0=\emptyset$. Let $(p_1)_i,(p_2)_i$ denote the restrictions of $p_1,p_2$ to $Z_i$, and factor $(p_2)_i$ as $(p_2)_i'(p_2)_i''$ where $(p_2)_i''$ is the quotient projection by the action of $P$. Similarly, let $(r_1)_i,(r_2)_i$ denote the restrictions of $p_1,p_2$ to $Z_i\smallsetminus Z_{i-1}$, and factor $(r_2)_i$ as $(r_2)_i'(r_2)_i''$ where $(r_2)_i''$ is the quotient projection by the action of $P$. Then for any $\cF\in\Perv_L(\cN_L,\bk)$ and any $i\in\{1,\cdots,s\}$, the canonical distinguished triangle associated to the closed embedding $Z_{i-1}\hookrightarrow Z_i$ induces a distinguished triangle in $\Db(\cN_M,\bk)$:
\begin{equation} \label{eqn:dist-tri}
\begin{split}
&((r_2)_i')_!((r_2)_i'')_\flat(r_1)_i^*\cF[2\dim U_P] \to\\
&\qquad\qquad((p_2)_i')_!((p_2)_i'')_\flat(p_1)_i^*\cF[2\dim U_P] \to\\
&\qquad\qquad\qquad\qquad((p_2)_{i-1}')_!((p_2)_{i-1}'')_\flat(p_1)_{i-1}^*\cF[2\dim U_P] \xrightarrow{+1},
\end{split}
\end{equation}
in which the third term is zero if $i=1$. 

The following lemma is the analogue of~\cite[Equation (3) in \S 10.1]{marsspringer}; its proof is postponed to~\S\ref{ss:proof-hard}.

\begin{lem} \label{lem:hard}
For any $\cF\in\Perv_L(\cN_L,\bk)$ and any $i\in\{1,\cdots,s\}$ we have an isomorphism in $\Db(\cN_M,\bk)$:
\[
((r_2)_i')_!((r_2)_i'')_\flat(r_1)_i^*\cF[2\dim U_P]
\cong
\Ind_{M\cap{}^{g_i}L \subset M\cap{}^{g_i} P}^M({}'\Res_{M\cap{}^{g_i} L\subset Q\cap{}^{g_i} L}^{{}^{g_i} L}(\Ad(g_i^{-1})^*\cF)).
\]
\end{lem}
\noindent

\begin{proof}[Proof of Theorem~{\rm \ref{thm:mackey}}]
Recall that our induction and restriction functors preserve the equivariant perverse categories (see~\cite[\S2.1 and references therein]{genspring1}). So for $\cF\in\Perv_L(\cN_L,\bk)$, Lemma~\ref{lem:hard} shows that the first term in the distinguished triangle~\eqref{eqn:dist-tri} belongs to $\Perv_M(\cN_M,\bk)$. Since $\Perv(\cN_M,\bk)$ is closed under extensions in $\Db(\cN_M,\bk)$, we can conclude by induction on $i$ that the second term in~\eqref{eqn:dist-tri} belongs to $\Perv(\cN_M,\bk)$ for all $i \in \{1, \cdots, s\}$. Thus the distinguished triangle~\eqref{eqn:dist-tri} becomes a short exact sequence in $\Perv(\cN_M,\bk)$. Using Lemma~\ref{lem:hard} and induction again, one can check that, for any $i \in \{1, \cdots, s\}$, the perverse sheaf $((p_2)_i')_!((p_2)_i'')_\flat(p_1)_i^*\cF[2\dim U_P]$ has a descending filtration of length $i$ with the same successive quotients as in Theorem~\ref{thm:mackey}. Taking $i=s$ and using Lemma~\ref{lem:easy}, we obtain Theorem~\ref{thm:mackey}.
\end{proof}

%------------------------------------------------------------------------------
\subsection{Proof of Lemma~\ref{lem:hard}}
\label{ss:proof-hard}
%------------------------------------------------------------------------------

In this subsection we prove Lemma~\ref{lem:hard}. We fix some $i \in \{1, \cdots, s\}$.

By definition, $Z_i\smallsetminus Z_{i-1}=\{(g,x)\in Qg_iP\times\cN_P\,|\,g\cdot x\in\cN_Q\}$. Set $Y_i=Q\times P\times \cN_{Q\cap{}^{g_i}P}$, and define 
\[ \sigma:Y_i\to Z_i\smallsetminus Z_{i-1}:(q,p,y)\mapsto(qg_ip^{-1},pg_i^{-1}\cdot y). \]
Then $\sigma$ is a principal bundle for the group $Q\cap{}^{g_i}P$, which acts on $Y_i$ by the rule
\[
g\cdot(q,p,y)=(qg^{-1},pg_i^{-1}g^{-1}g_i,g\cdot y),\text{ for }g\in Q\cap{}^{g_i}P,\,(q,p,y)\in Y_i.
\]
Define maps
\[
\begin{split}
s_1&=(r_1)_i\sigma:Y_i\to\cN_L:(q,p,y)\mapsto p_{L\subset P}(pg_i^{-1}\cdot y),\\
s_2&=(r_2)_i\sigma:Y_i\to\cN_M:(q,p,y)\mapsto p_{M\subset Q}(q\cdot y),\\
\end{split}
\]
and factor $s_2$ as $s_2's_2''$ where $s_2'':Y_i\to Q\times^{Q\cap{}^{g_i}P}\cN_{Q\cap{}^{g_i}P}$ is the obvious projection, a principal bundle for $P\times(Q\cap{}^{g_i}P)$. Then by elementary isomorphisms of sheaf functors, we have
\begin{equation} \label{eqn:s-version}
((r_2)_i')_!((r_2)_i'')_\flat (r_1)_i^*\cF
\cong
(s_2')_!(s_2'')_\flat s_1^*\cF.
\end{equation}

Denote the image of $p\in P$ under the canonical projection $P\to L$ by $\overline{p}$. The map $s_1$ factors as $ab$ where $b:Y_i\to L\times\cN_L:(q,p,y)\mapsto(\overline{p},p_{L\subset P}(g_i^{-1}\cdot y))$ and $a:L\times\cN_L\to\cN_L$ is the $L$-action. Since $\cF$ belongs to $\Perv_L(\cN_L,\bk)$, we have an isomorphism $a^*\cF\cong\mathrm{pr}_2^*\cF$ where $\mathrm{pr}_2:L\times\cN_L\to\cN_L$ is the second projection. Hence $s_1^*\cF\cong\tilde{s}_1^*\cF$ where 
\[ \tilde{s}_1=\mathrm{pr}_2 b:Y_i\to\cN_L:(q,p,y)\mapsto p_{L\subset P}(g_i^{-1}\cdot y). \]
But after modifying $s_1$ to $\tilde{s}_1$, we see that both $\tilde{s}_1$ and $s_2$ factor through the projection $Y_i\to Q\times\cN_{Q\cap{}^{g_i}P}$ (i.e.\ they are independent of the $P$ factor). So if we define
\[
\begin{split}
\hat{s}_1&:Q\times\cN_{Q\cap{}^{g_i}P}\to\cN_L:(q,y)\mapsto p_{L\subset P}(g_i^{-1}\cdot y),\\
\hat{s}_2&:Q\times\cN_{Q\cap{}^{g_i}P}\to\cN_M:(q,y)\mapsto p_{M\subset Q}(q\cdot y),
\end{split}
\]
and factor $\hat{s}_2$ as $\hat{s}_2'\hat{s}_2''$ where $\hat{s}_2'':Q\times\cN_{Q\cap{}^{g_i}P}\to Q\times^{Q\cap{}^{g_i}P}\cN_{Q\cap{}^{g_i}P}$ is the quotient projection, then we have
\begin{equation} \label{eqn:hats-version}
(s_2')_!(s_2'')_\flat s_1^*\cF
\cong
(\hat{s}_2')_!(\hat{s}_2'')_\flat \hat{s}_1^*\cF.
\end{equation}

Since $M\cap{}^{g_i}L$ contains a maximal torus, we have a direct sum decomposition
\begin{equation*}
\Lie(Q\cap{}^{g_i}P)=\Lie(M\cap{}^{g_i}L)\oplus\Lie(M\cap{}^{g_i}U_P)\oplus\Lie(U_Q\cap{}^{g_i}L)\oplus\Lie(U_Q\cap{}^{g_i}U_P),
\end{equation*}
and an element of $\Lie(Q\cap{}^{g_i}P)$ is nilpotent if and only if its component in $\Lie(M\cap{}^{g_i}L)$ is nilpotent. Define
\begin{equation}
\label{eqn:def-Ni}
\cN_i:=\cN_{M\cap{}^{g_i}L}\times\Lie(M\cap{}^{g_i}U_P)\times\Lie(U_Q\cap{}^{g_i}L).
\end{equation}
Then we have an obvious identification of $\cN_i$ with the fibre product of $\cN_{M\cap{}^{g_i}P}$ and $\cN_{Q\cap{}^{g_i}L}$ over $\cN_{M\cap{}^{g_i}L}$. Also we have a vector bundle projection
\[
\tau:\cN_{Q\cap{}^{g_i}P}\to\cN_i,
\]  
given by forgetting the component in $\Lie(U_Q\cap{}^{g_i}U_P)$. Since the latter space is $(Q\cap{}^{g_i}P)$-stable, we have a natural action of $(Q\cap{}^{g_i}P)$ on $\cN_i$ making $\tau$ a $(Q\cap{}^{g_i}P)$-equivariant map. If we denote the induced vector bundle projection $Q\times\cN_{Q\cap{}^{g_i}P}\to Q\times\cN_i$ by $\tau$ also, then both $\hat{s}_1$ and $\hat{s}_2$ factor through $\tau$, say $\hat{s}_1=t_1\tau$ and $\hat{s}_2=t_2\tau$. Factoring $t_2$ as $t_2't_2''$ where $t_2'':Q\times\cN_i\to Q\times^{Q\cap{}^{g_i}P}\cN_i$ is the quotient projection, we have
\begin{equation} \label{eqn:t-version}
(\hat{s}_2')_!(\hat{s}_2'')_\flat \hat{s}_1^*\cF
\cong
(t_2')_!(t_2'')_\flat t_1^*\cF[-2\dim(U_Q\cap{}^{g_i}U_P)].
\end{equation}
We can alternatively factorize $t_2$ as $\hat{t}_2'\hat{t}_2''$ where $\hat{t}_2''$ is the quotient projection by the subgroup $M\cap{}^{g_i}P$ of $Q\cap{}^{g_i}P$. Since $U_Q\cap{}^{g_i}P$ is an affine space, we have
\begin{equation} \label{eqn:hatt-version}
(t_2')_!(t_2'')_\flat t_1^*\cF\cong
(\hat{t}_2')_!(\hat{t}_2'')_\flat t_1^*\cF[2\dim(U_Q\cap{}^{g_i}P)].
\end{equation}

To make the definition of $t_1$ and $t_2$ more explicit, if $q\in Q$ and $(x,v,w)\in\cN_i$ (with $x,v,w$ respectively in the three spaces appearing in the right-hand side of~\eqref{eqn:def-Ni}), we have
\[
\begin{split}
t_1(q,x,v,w)&=g_i^{-1}\cdot(x+w),\\
t_2(q,x,v,w)&=\overline{q}\cdot(x+v),
\end{split}
\]
where $\overline{q}\in M$ is the image of $q$ under the canonical projection $Q\to M$. In particular, both maps $t_1$ and $t_2$ factor through the induced projection $\varphi:Q\times\cN_i\to M\times\cN_i$, which is an affine space bundle. Writing $t_1=u_1\varphi$ and $t_2=u_2\varphi$ and factoring $u_2$ as $u_2'u_2''$ where $u_2'':M\times\cN_i\to M\times^{M\cap{}^{g_i}P}\cN_i$ is the quotient projection, we have
\begin{equation} \label{eqn:u-version}
(\hat{t}_2')_!(\hat{t}_2'')_\flat t_1^*\cF\cong
(u_2')_!(u_2'')_\flat u_1^*\cF[-2\dim U_Q].
\end{equation}

On the other hand, the definitions of $\Ind_{M\cap{}^{g_i}L \subset M\cap{}^{g_i} P}^M$ and ${}'\Res_{M\cap{}^{g_i} L\subset Q\cap{}^{g_i} L}^{{}^{g_i} L}$, together with the base change isomorphism applied to the Cartesian square
\[
\xymatrix{
\cN_i \ar[r] \ar[d] & \cN_{M \cap {}^{g_i} P} \ar[d] \\
\cN_{Q \cap {}^{g_i} L} \ar[r] & \cN_{M \cap {}^{g_i} L},
}
\]
show that
\begin{equation} \label{eqn:last-step}
\Ind_{M\cap{}^{g_i}L \subset M\cap{}^{g_i} P}^M({}'\Res_{M\cap{}^{g_i} L\subset Q\cap{}^{g_i} L}^{{}^{g_i} L}(\Ad(g_i^{-1})^*\cF))
\cong
(u_2')_!(u_2'')_\flat u_1^*\cF[2\dim(M\cap{}^{g_i}P)].
\end{equation}
Putting together~\eqref{eqn:s-version}, \eqref{eqn:hats-version}, \eqref{eqn:t-version}, \eqref{eqn:hatt-version}, \eqref{eqn:u-version} and~\eqref{eqn:last-step}, all that remains is to check that the shifts match up, or in other words that
\begin{equation}
\dim U_P-\dim(U_Q\cap{}^{g_i}U_P)+\dim(U_Q\cap{}^{g_i}P)-\dim U_Q=\dim(M\cap{}^{g_i}U_P).
\end{equation}
This is the content of~\cite[Third equation on p.~176]{marsspringer}.

%%%%%%%%%%%%%%%%%%%%%%%%%%%%%%%%%%%%%%%%%%%%%%%%%%%%%%%%%%%%%%%%%%%%%%%
\section{Completion of the proof of Theorem~\ref{thm:mgsc-intro}}
\label{sec:case-by-case}
%%%%%%%%%%%%%%%%%%%%%%%%%%%%%%%%%%%%%%%%%%%%%%%%%%%%%%%%%%%%%%%%%%%%%%%

%------------------------------------------------------------------------------
\subsection{Normalizers of distinguished orbits and local systems}
\label{ss:normalizers-distinguished}
%------------------------------------------------------------------------------

Since Theorem~\ref{thm:disjointness} has completed the proof of~\eqref{eqn:disjointness}, we turn now to the proof of the remaining statement~\eqref{eqn:bijection} of Theorem~\ref{thm:mgsc-intro}. As was mentioned in the introduction, under the assumption that $\bk$ is big enough for $G$ in the sense of~\eqref{eqn:definitely-big-enough}, this statement will follow immediately from~\cite[Lemma 2.1 and Theorem 3.1]{genspring2} if we can check that every cuspidal pair for a Levi subgroup $L$ is fixed by the action of $N_G(L)$. As was also mentioned in the introduction, we checked in~\cite[Lemma 2.9]{genspring2} that $N_G(L)$ preserves every distinguished orbit of $L$. Since every cuspidal pair is supported on a distinguished orbit (by Proposition~\ref{prop:levi-rule}), the following supplement to that lemma completes the proof of Theorem~\ref{thm:mgsc-intro}. Note that for this result we do not have to assume~\eqref{eqn:definitely-big-enough}.

\begin{prop}
\label{prop:horrible}
If $L$ is a Levi subgroup of $G$ and $\cO_L$ is a distinguished orbit for $L$, then $N_G(L)$ preserves the isomorphism class of every irreducible $L$-equivariant local system on $\cO_L$. 
\end{prop} 

Before beginning the proof, we make an observation which will be used throughout the paper. If $H$ is a connected reductive group and $Z$ is a closed subgroup of $Z(H)$, then we can identify the nilpotent cones $\cN_H$ and $\cN_{H/Z}$. Moreover, for any $x\in\cN_H=\cN_{H/Z}$, the short exact sequence $1\to Z\to H_x\to(H/Z)_x\to 1$ induces an exact sequence of component groups (part of the long exact sequence of homotopy groups):
\begin{equation} \label{eqn:exact-seq}
Z/Z^\circ\to A_H(x)\to A_{H/Z}(x)\to 1.
\end{equation}
We will often use~\eqref{eqn:exact-seq} without comment: in particular, the case $Z=Z(H)^\circ$ gives $A_H(x)\cong A_{H/Z(H)^\circ}(x)$, and the case $Z=Z(H)$ shows that $Z(H)/Z(H)^\circ$ surjects onto the kernel of $A_H(x)\twoheadrightarrow A_{H/Z(H)}(x)$.

Let $V$ be a $\bk$-representation of $A_H(x)$.  We say that $V$ has \emph{central character} $\chi$, where $\chi: Z(H)/Z(H)^\circ \to \bk^\times$ is a group homomorphism, if the action of $Z(H)/Z(H)^\circ$ on $V$ via $Z(H)/Z(H)^\circ \to A_H(x)$ is given by $\chi$.  Schur's lemma implies that if $V$ is absolutely irreducible, then it has a central character.  Central characters of local systems are defined in the same way; see~\cite[\S5.1]{genspring2} for a discussion.

Another ingredient of the proof of Proposition~\ref{prop:horrible} is the classification of conjugacy classes of Levi subgroups of exceptional groups, explained further in \S\ref{ss:computing}.

\begin{proof}[Proof of Proposition~{\rm \ref{prop:horrible}}]
First we have some easy reductions. Since the statement is unchanged if one replaces $G$ by $G/Z(G)^\circ$, we can assume that $G$ is semisimple; the statement can only become stronger if one replaces $G$ by its simply connected cover, so it suffices to consider the case where $G$ is simply connected. Then $G$ is a product of simply connected quasi-simple groups, so it suffices to consider the case where $G$ is simply connected and quasi-simple. We can thus consider each Lie type in turn.

Before turning to the individual types, we reformulate the problem. Fix $x\in\cO_L$ and recall that the isomorphism classes of the irreducible $L$-equivariant local systems on $\cO_L$ are in bijection with the set $\Irr(\bk[A_L(x)])$ of isomorphism classes of irreducible $\bk$-representations of $A_L(x)$. Since we already know that $N_G(L)$ preserves $\cO_L$, our claim is equivalent to saying that $A_{N_G(L)}(x)$ (in which $A_L(x)$ is a normal subgroup) acts trivially on $\Irr(\bk[A_L(x)])$. We have some easy principles:
\begin{enumerate}
\item The extreme cases where $L=G$ or $L$ is a maximal torus hold trivially.
\item If the homomorphism $Z(L)/Z(L)^\circ\to A_L(x)$ is surjective, then the claim holds, because the action of $N_G(L)$ on $Z(L)/Z(L)^\circ$ is trivial (see~\cite[proof of Lemma 5.3]{genspring2}). In particular, this applies when $\cO_L$ is the regular nilpotent orbit for $L$, because then $Z(L)/Z(L)^\circ\to A_L(x)$ is an isomorphism.
\item If $|A_L(x)|\leq 2$ or $A_L(x)\cong\fS_3$, then the claim holds, because all automorphisms of $A_L(x)$ are inner. 
\end{enumerate}

If $G$ is of type $A$ then the only distinguished nilpotent orbit for $L$ is the regular one, and principle~(2) applies.

If $G$ is of type $C$ then $G=\Sp(V)$ where $V$ is a vector space with a non-degenerate skew-symmetric bilinear form. As in the proof of~\cite[Lemma 2.9]{genspring2}, we have $L=\Sp(U)\times H$ and $N_G(L)=\Sp(U)\times H'$ where $V=U\oplus U^\perp$ is an orthogonal decomposition and $H$ and $H'$ are subgroups of $\GL(U^\perp)$. In fact, $H$ is a product of general linear groups, so $A_H(x)$ is trivial. (Here $x\in\cN_L=\cN_{\Sp(U)}\times\cN_H$, and the action of $H$ on the $\cN_{\Sp(U)}$ factor is trivial.) Thus $A_L(x)=A_{\Sp(U)}(x)$ and $A_{N_G(L)}(x)=A_{\Sp(U)}(x)\times A_{H'}(x)$, showing that the action of $A_{N_G(L)}(x)$ on $\Irr(\bk[A_L(x)])$ is trivial.
 
If $G$ is of type $B$ or $D$ then $G=\Spin(V)$ where $V$ is a vector space with a non-degenerate symmetric bilinear form. Let $\overline{G}=\SO(V)$ and let $\overline{L}$ denote the image of $L$ in $G$, a Levi subgroup of $G$. We have $\overline{L}=\SO(U)\times H$ and $N_{\overline{G}}(\overline{L})=\SO(U)\times H'$ where $V=U\oplus U^\perp$ is an orthogonal decomposition and $H$ and $H'$ are subgroups of $\GL(U^\perp)$, $H$ being a product of general linear groups. Since $G\twoheadrightarrow\overline{G}$ is a central quotient, the action of $A_{N_G(L)}(x)$ preserves the subset of $\Irr(\bk[A_L(x)])$ consisting of representations that factor through $A_{\overline{L}}(x)=A_{\SO(U)}(x)$; for these representations, we may replace $G$ by $\overline{G}$ and the same argument as in the type-$C$ case applies. So we need only consider the irreducible representations of $A_L(x)$ that do not factor through $A_{\overline{L}}(x)$. An explicit description of $L$ was given in~\cite[proof of Theorem 8.4]{genspring2}: we have
\begin{equation} \label{eqn:spin-levi}
L=(\Spin(U)\times M)/\langle(\varepsilon,\delta)\rangle,
\end{equation}
where $\{1,\varepsilon\}$ is the kernel of $\Spin(U)\twoheadrightarrow\SO(U)$ and $M\twoheadrightarrow H$ is a certain double cover of $H$ with kernel $\{1,\delta\}$. Using~\eqref{eqn:exact-seq} we see that $A_M(x)$ is generated by the image of $\delta$, and hence is either trivial or has two elements. If $A_M(x)$ is trivial, then $A_L(x)\to A_{\SO(U)}(x)$ is an isomorphism, so we can neglect this case. If $A_M(x)$ has two elements, then $A_L(x) \cong A_{\Spin(U)}(x)$. As explained in~\cite[\S 8.4]{genspring2}, $A_{\Spin(U)}(x)$ has at most $2$ irreducible $\bk$-representations not factoring through $A_{\SO(U)}(x)$. If there are $2$ such representations (which can only occur when $\bk$ contains a primitive fourth root of unity) then they have central characters for $Z(L)/Z(L)^\circ$ which are distinct, so they cannot be interchanged by $A_{N_G(L)}(x)$. 

It remains to consider the case where $G$ is a simply connected quasi-simple group of exceptional type. By principles (1) and (2), we need only consider non-regular distinguished nilpotent orbits $\cO_L$ for proper non-toral Levi subgroups $L$ of $G$; in particular, Levi subgroups all of whose components are of type $A$ can be ignored. We will see that each such orbit is covered by principle (3). The required descriptions of distinguished nilpotent orbits and the groups $A_L(x)$ can be found in~\cite[Corollary 6.1.6, Theorem 8.2.4 and \S 8.4]{cm}. 

If $G$ is of type $G_2$, there are no such orbits.

If $G$ is of type $F_4$, the only such orbit is the subregular orbit for $L$ of type $C_3$, for which $|A_L(x)|=2$ (note that $L/Z(L)^\circ\cong\PSp(6)$).

If $G$ is of type $E_6$, the only such orbits are the subregular orbits for $L$ of type $D_4$ or $D_5$, for which $|A_L(x)|=1$ (note that $L/Z(L)^\circ\cong\PSO(8)$ or $\PSO(10)$ respectively).

If $G$ is of type $E_7$, the only such orbits are:
\begin{itemize}
\item the subregular orbits for $L$ of type $D_4$ or $D_5$, for which $|A_L(x)|=1$ (again, $L/Z(L)^\circ\cong\PSO(8)$ or $\PSO(10)$ respectively);
\item the orbits of the form $(\text{subregular}\times\text{regular})$ for $L$ of type $D_4\times A_1$ or $D_5\times A_1$, for which $|A_L(x)|=2$ (here $L/Z(L)^\circ$ is a double cover of $\PSO(8)\times \PGL(2)$ or $\PSO(10)\times\PGL(2)$ respectively);
\item the two non-regular distinguished orbits for $L$ of type $D_6$, labelled by the partitions $[9,3]$ and $[7,5]$, for which $|A_L(x)|=2$ (here $L/Z(L)^\circ$ is a double cover of $\PSO(12)$);
\item the two non-regular distinguished orbits for $L$ of type $E_6$, having Bala--Carter labels $E_6(a_1)$ and $E_6(a_3)$, for which $|A_L(x)|=1$ and $2$ respectively (here $L/Z(L)^\circ$ is the adjoint group of type $E_6$).
\end{itemize}

If $G$ is of type $E_8$, we may neglect Levi subgroups which (up to $G$-conjugacy) are proper subgroups of a Levi subgroup of type $E_7$, because the groups $A_L(x)$ in the $E_8$ context are either the same as or smaller than the corresponding groups listed above in the $E_7$ context, since now $L/Z(L)^\circ$ is always of adjoint type. The remaining orbits are: 
\begin{itemize}
\item the orbits of the form $(\text{subregular}\times\text{regular})$ for $L$ of type $D_4\times A_2$ or $D_5\times A_2$, for which $|A_L(x)|=1$;
\item the two orbits of the form $(\text{non-regular distinguished}\times\text{regular})$ for $L$ of type $E_6\times A_1$, for which $|A_L(x)|=1$ or $2$;
\item the two non-regular distinguished orbits for $L$ of type $D_7$, labelled by the partitions $[11,3]$ and $[9,5]$, for which $|A_L(x)|=1$;
\item the five non-regular distinguished orbits for $L$ of type $E_7$, having Bala--Carter labels $E_7(a_1)$, $E_7(a_2)$, $E_7(a_3)$, $E_7(a_4)$ and $E_7(a_5)$, for which $|A_L(x)|$ is respectively $1,1,2,2$ and $6$, with $A_L(x)=\fS_3$ in the last case.
\end{itemize}

This concludes the proof of Proposition~\ref{prop:horrible} and thus of Theorem~\ref{thm:mgsc-intro}.
\end{proof}

%------------------------------------------------------------------------------
\subsection{The `big enough' condition}
\label{ss:big-enough}
%------------------------------------------------------------------------------

In this subsection we make the condition~\eqref{eqn:definitely-big-enough} on the field $\bk$ more explicit in particular cases of interest. Recall that this condition is equivalent to requiring $\bk$ to be a splitting field for all the finite groups $A_L(x)$ where $L$ is a Levi subgroup of $G$ and $x\in\cN_L$. We will use the obvious fact that if $\bk$ is a splitting field for a group $\Gamma$, it is also a splitting field for any quotient group of $\Gamma$.

\begin{prop}
\label{prop:big-enough}
Let $G$ be as above.
\begin{enumerate}
\item If $G$ has connected centre, then condition~\eqref{eqn:definitely-big-enough} is automatically true.
\item If $G$ is quasi-simple, then according to the type and isogeny class of $G$, the condition~\eqref{eqn:definitely-big-enough} is equivalent to the following:
\begin{itemize}
\item $A_{n-1}$, $n\geq 3$: $\bk$ contains all $|Z(G)|$-th roots of unity of its algebraic closure;
\item $B_n$, $n=7$ or $n\geq 9$, and $G$ is simply connected: $\bk$ contains all fourth roots of unity of its algebraic closure;
\item $D_n$, $n=5$, $7$, $9$, $11$, or $n \ge 13$, and $G$ is simply connected: $\bk$ contains all fourth roots of unity of its algebraic closure;
\item $D_n$, $n$ even and${}\ge 14$, and $G = \frac{1}{2}\Spin(2n)$: $\bk$ contains all fourth roots of unity of its algebraic closure;
\item $E_6$, and $G$ is simply connected: $\bk$ contains all third roots of unity of its algebraic closure;
\item other cases: no condition. 
\end{itemize} 
\end{enumerate}
\end{prop}

\begin{proof}
To prove (1), recall that if $G$ has connected centre then so does every Levi subgroup $L$ of $G$. Hence $A_L(x)\cong A_{L/Z(L)}(x)$ is a product of groups of the form $A_H(y)$ where $H$ is simple (of adjoint type) and $y\in\cN_H$. It is well known that every such $A_H(y)$ is one of $(\Z/2\Z)^k$ (for some $k\geq 0$), $\fS_3$, $\fS_4$ or $\fS_5$ (see~\cite[Corollary 6.1.7 and \S 8.4]{cm}). Any field is a splitting field for these groups.

We now prove (2), assuming that $G$ is quasi-simple. Note that if $G$ is of adjoint type, it is covered by part~(1).  In particular, quasi-simple groups of types $E_8$, $F_4$, or $G_2$ are covered by~(1).

Suppose that $G$ is of type $A_{n-1}$ for $n\geq 2$. Recall that, for $m \geq 2$, $\SL(m)$ has the property that for any $x\in\cN_{\SL(m)}$, the natural homomorphism $Z(\SL(m))\to A_{\SL(m)}(x)$ is surjective (see e.g.~\cite[\S 6.1]{genspring2}). Hence the same property holds for any central quotient of a product of $\SL(m_i)$'s, i.e.\ for any semisimple group of type $A$. For a Levi subgroup $L$ of $G$, the semisimple quotient $L/Z(L)^\circ$ is of type $A$, so we can conclude that for any $x\in\cN_L$, the natural homomorphism $Z(L)/Z(L)^\circ\to A_L(x) \cong A_{L/Z(L)^\circ}(x)$ is surjective. Composing this with the surjective homomorphism $Z(G)\to Z(L)/Z(L)^\circ$, we deduce that the natural homomorphism $Z(G)\to A_L(x)$ is surjective. Moreover, this homomorphism is an isomorphism when $L=G$ and $x\in\cN_G$ is regular nilpotent. So~\eqref{eqn:definitely-big-enough} is equivalent to requiring $\bk$ to be a splitting field for the cyclic group $Z(G)$, a quotient of $\mu_{n}$. This, in turn, is equivalent to the stated condition if $n\geq 3$ and is automatic for $n=2$.

If $G$ is of type $C_n$ for $n\geq 3$ and not of adjoint type, then $G=\Sp(V)$. As seen in the proof of Proposition~\ref{prop:horrible}, every Levi subgroup $L$ is a product of general linear groups and symplectic groups (at most one of the latter). So by~\cite[Corollary 6.1.6]{cm} every $A_L(x)$ is a group of the form $(\Z/2\Z)^k$ for some $k\geq 0$, for which any field is a splitting field.

If $G$ is of type $B_n$ for $n\geq 2$ or $D_n$ for $n\geq 3$ (using the convention that $D_3=A_3$), let us first consider the case where $G$ is simply connected.  Thus, $G=\Spin(V)$ where $d=\dim V\geq 5$. We use the same description of a general Levi subgroup $L$ as in the proof of Proposition~\ref{prop:horrible} (see~\eqref{eqn:spin-levi}). If $A_L(x)\cong A_{\SO(U)}(x)$, then $A_L(x)$ is a group of the form $(\Z/2\Z)^k$, as in the symplectic case. So we need only consider the cases in which $A_M(x)$ has two elements, meaning that $A_L(x)\cong A_{\Spin(U)}(x)$; as in~\cite[proof of Theorem 8.4]{genspring2}, this happens if and only if the partitions labelling the $\cN_H$ factor of $x$ have no odd parts. The latter condition implies in particular that every general linear group factor of $H$ has even rank, which forces $\dim U\equiv \dim V$ (mod $4$). We can conclude that $\bk$ satisfies~\eqref{eqn:definitely-big-enough} if and only if it is a splitting field for all the groups $A_{\Spin(d')}(y)$ where $d'\leq d$, $d'\equiv d$ (mod $4$), and $y\in\cN_{\Spin(d')}$. The groups $A_{\Spin(d')}(y)$ are $2$-groups, possibly non-abelian; they are explicitly described in~\cite[\S 14.3]{lusztig} in terms of the partition of $d'$ that labels the orbit of $y$. As we observed in~\cite[\S 8.4]{genspring2}, if $\bk$ contains all fourth roots of unity of its algebraic closure (a vacuous condition when $\ell=2$), then $\bk$ is a splitting field for all $A_{\Spin(d')}(y)$. What remains is just to determine, within each congruence class modulo $4$, what the smallest value of $d'$ is for which there is a group $A_{\Spin(d')}(y)$ that actually requires the fourth roots of unity, assuming $\ell\neq 2$. We claim that the answers are $6$ ($\equiv 2$), $15$ ($\equiv 3$), $28$ ($\equiv 0$) and $21$ ($\equiv 1$), whence the rank conditions in the statement. Suitable partitions (in fact, the unique suitable partitions) in these four cases are $[5,1]$ (giving $A_{\Spin(6)}(y)\cong\Z/4\Z$), $[9,5,1]$ (giving $A_{\Spin(15)}(y)\cong Q$, the quaternion group), $[13,9,5,1]$ (giving $A_{\Spin(28)}(y)\cong Q\times\Z/2\Z$) and $[11,7,3]$ (giving $A_{\Spin(21)}(y)\cong Q$). We leave it to the reader to verify that when $d'$ is below these claimed bounds within each congruence class, every group $A_{\Spin(d')}(y)$ is a product of copies of $\Z/2\Z$ and the dihedral group of order $8$, for both of which any field is a splitting field.   

For type $B$, the argument is now complete, but in type $D$ there are one or two additional isomorphism classes that are neither of adjoint type nor simply connected.  Suppose first that $G = \SO(V)$ with $d = \dim V$ even. Then, retaining the notation of the previous paragraph, we have that for any Levi subgroup $L$, $A_L(x) \cong A_{\SO(U)}(x)$ is a group of the form $(\Z/2\Z)^k$, so any field is a splitting field.  Now suppose that $G = \frac{1}{2}\Spin(V)$ with $d \equiv 0 \pmod 4$.  As above, it is enough to determine the smallest $d'$ with $d' \le d$ and $d' \equiv d \pmod 4$ such that some $A_{\frac{1}{2}\Spin(d')}(y)$ requires the fourth roots of unity.  For $d = 28$ and $y$ in the orbit labelled by $[13,9,5,1]$, one can check that $A_{\frac{1}{2}\Spin(28)}(y) \cong Q$, giving the rank condition in the statement.

If $G$ is of type $E_6$ and not of adjoint type, then $|Z(G)|=3$. Since $Z(G)\cong A_G(x)$ for $x$ regular nilpotent, the condition~\eqref{eqn:definitely-big-enough} certainly requires $\bk$ to contain all third roots of unity of its algebraic closure. We must show the converse: i.e.\ we assume that $\bk$ contains these third roots of unity, and must show that $\bk$ is a splitting field for all the groups $A_L(x)$. When $L=G$, we see from~\cite[\S 8.4]{cm} that every group $A_G(x)$ is one of $(\Z/2\Z)^k$, $\Z/3\Z$ or $\Z/2\Z\times\Z/3\Z$, so $\bk$ is a splitting field for all of them. For most classes of proper Levi subgroups $L$, $L$ has connected centre and is thus covered by our previous argument. The exceptions are the Levi subgroups of types $2A_2$, $2A_2+A_1$, and $A_5$, but these are all of type $A$, and we have seen above that $Z(G)$ surjects onto $A_L(x)$ in all such cases, so $\bk$ is a splitting field for $A_L(x)$ as required.

The argument for $G$ of type $E_7$ and not of adjoint type is similar: here $|Z(G)|=2$. Every group $A_G(x)$ is one of $(\Z/2\Z)^k$, $\fS_3$ or $\Z/2\Z\times\fS_3$, for which every field is a splitting field. (The table in~\cite[\S 8.4]{cm} contains two misprints, not affecting this statement: for $x$ with Bala--Carter label $4A_1$ or $(A_5)''$, the group $A_G(x)$ should be $\Z/2\Z$.) Most classes of proper Levi subgroups either have connected centre or are of type $A$, in which case they are covered by previous arguments; the remaining ones are those of type $D_4+A_1$, $D_5+A_1$, and $D_6$. If $L$ is one of these Levis, then $|Z(L)/Z(L)^\circ|=2$, and hence each group $A_L(x)$ is either isomorphic to, or a double cover of, the corresponding group $A_{L/Z(L)}(x)=A_{\PSO(d)}(x)=(\Z/2\Z)^p$, where $d=8,10,12$ respectively. We claim that $A_L(x)$ is of the form $(\Z/2\Z)^{k}$. Otherwise, we must have simultaneously that $A_{\Spin(d)}(x)$ is not of the form $(\Z/2\Z)^{k}$ and that $A_{\PSO(d)}(x)$ is nontrivial. From~\cite[Corollary 6.1.6]{cm} we see that when $d$ is even, such a state of affairs requires $d\geq 7+5+3+1=16$. The proof is finished.
\end{proof}

%------------------------------------------------------------------------------
\subsection{Modular generalized Springer correspondence and field extensions}
\label{ss:indep-of-k}
%------------------------------------------------------------------------------

To conclude this section we show that, once $\bk$ is big enough for $G$, the modular generalized Springer correspondence is unchanged under further field extension. In that sense, it depends only on the characteristic $\ell$ of $\bk$.

\begin{lem} \label{lem:splitting}
Assume that $\bk$ is big enough for $G$, and let $\bk'$ be an extension field of $\bk$. For any Levi subgroup $L$ of $G$, we identify $\fN_{L,\bk'}$ with $\fN_{L,\bk}$ in the canonical way.
\begin{enumerate}
\item\label{it:cuspidality-under-field-extension} 
Under this identification, $\fN_{L,\bk'}^\cusp$ and $\fN_{L,\bk}^\cusp$ coincide, so we can choose $\fM_{G,\bk'}$ so that it is canonically identified with $\fM_{G,\bk}$.
\item\label{it:cuspidal-splitting-field} 
For any cuspidal datum $(L,\cO_L,\cE_L)\in\fM_{G,\bk}$, the field $\bk$ is a splitting field for $N_G(L)/L$. Hence $|\fN_{G,\bk}^{(L,\cO_L,\cE_L)}|=|\Irr(\bk[N_G(L)/L])|$ equals the number of $\ell$-regular conjugacy classes of $N_G(L)/L$, and $\Irr(\bk'[N_G(L)/L])$ is canonically identified with $\Irr(\bk[N_G(L)/L])$.
\item\label{it:indep-of-k} 
Under the above identifications, the bijection~\eqref{eqn:mgsc} is the same for $\bk'$ as for $\bk$.
\end{enumerate}
\end{lem}

\begin{proof}
Part~\eqref{it:cuspidality-under-field-extension} is immediate from the definition of cuspidality. Then, the fact that the bijection~\eqref{eqn:mgsc-variant} holds for both $\bk$ and $\bk'$ forces the natural inequalities $|\Irr(\bk[N_G(L)/L])|\leq|\Irr(\bk'[N_G(L)/L])|$ to be equalities, so every irreducible representation of $N_G(L)/L$ over $\bk$ must be absolutely irreducible, proving part~\eqref{it:cuspidal-splitting-field}. Part~\eqref{it:indep-of-k} is clear from the construction of~\eqref{eqn:mgsc} in~\cite{genspring2}, since every sheaf-theoretic functor involved in that construction commutes with extension of scalars (see~\cite[Remark 2.23]{genspring1}).
\end{proof}

See Remark~\ref{rmk:reflection-group} for comments on the structure on $N_G(L)/L$.

%%%%%%%%%%%%%%%%%%%%%%%%%%%%%%%%%%%%%%%%%%%%%%%%%%%%%%%%%%%%%%%%%%%%%%%
\section{The $\IC$ sheaf of the regular nilpotent orbit}
\label{sec:regular}
%%%%%%%%%%%%%%%%%%%%%%%%%%%%%%%%%%%%%%%%%%%%%%%%%%%%%%%%%%%%%%%%%%%%%%%

Continue to let $G$ be an arbitrary connected reductive group over $\C$. Our focus now turns to the explicit description of the modular generalized Springer correspondence~\eqref{eqn:mgsc}. Of particular interest are the ways in which it differs from Lusztig's generalized Springer correspondence (the analogous bijection for $\bk=\Qlb$). As in the case of classical groups~\cite{genspring1,genspring2}, we should expect to find more cuspidal data and hence more induction series than in Lusztig's setting.

Let $\cO_{\reg}$ be the regular nilpotent orbit in $\cN_G$. In Lusztig's setting, the simple perverse sheaf $\IC(\cO_{\reg},\underline{\Qlb})\cong(\underline{\Qlb})_{\cN_G}[\dim\cN_G]$ always belongs to the principal induction series associated to the cuspidal datum $(T,\{0\},\ubk)$ where $T$ is a maximal torus, i.e.\ the non-generalized Springer correspondence; in the convention aligned with ours in the modular case, it corresponds to the sign representation of the Weyl group $W=N_G(T)/T$. In Theorem~\ref{thm:regular-series} we will determine which induction series contains $\IC(\cO_{\reg},\ubk)$ in the modular case. In particular, this gives a necessary and sufficient condition for the pair $(\cO_{\reg}, \ubk)$ to be cuspidal.

%------------------------------------------------------------------------------
\subsection{The constant perverse sheaf on $\cN_G$}
\label{ss:constant-perv}
%------------------------------------------------------------------------------

Note that, since
$\cN_G$ is a complete intersection, the shifted constant sheaf $\ubk_{\cN_G}[\dim\cN_{G}]$ belongs to $\Perv_G(\cN_G,\bk)$ by~\cite[Lemma~III.6.5]{kw}.

\begin{lem}
\label{lem:projective-cover}
If $\ell\nmid |Z(G)/Z(G)^\circ|$, then $\ubk_{\cN_G}[\dim\cN_{G}]$ is 
a
projective 
cover of $\IC(\cO_\reg, \ubk)$ 
in $\Perv_G(\cN_G,\bk)$.
\end{lem}

\begin{proof}
It is shown in~\cite[Proposition 5.1]{am} that $\ubk_{\cN_G}[\dim\cN_{G}]$ is projective and has $\IC(\cO_\reg, \ubk)$ as a quotient. Since $\End(\ubk_{\cN_G}[\dim\cN_{G}])=\bk$, the claim follows.
\end{proof}

Now let $P\subset G$ be a parabolic subgroup and $L\subset P$ a Levi factor containing the maximal torus $T$. Let $W_L=N_L(T)/T$ be the Weyl group of $L$, a parabolic subgroup of $W=N_G(T)/T$.

\begin{prop}
\label{prop:constant-summand}
The shifted constant sheaf $\ubk_{\cN_G}[\dim\cN_{G}]$ is a direct summand of the induced perverse sheaf
$\Ind_{L \subset P}^G(\ubk_{\cN_{L}}[\dim\cN_{L}])$ if and only if 
$\ell\nmid |W/W_{L}|$, or equivalently if and only if $W_L$ contains an $\ell$-Sylow subgroup of $W$.
\end{prop}

\begin{rmk} \label{rmk:trivial-summand}
For comparison, recall that the trivial (respectively, the sign) representation of the group $W$ over $\bk$ occurs as a direct summand of the induction of the trivial (respectively, the sign) representation of $W_L$ if and only if $W_L$ contains an $\ell$-Sylow subgroup of $W$, see~\cite[\S 5, Corollary 1]{green}.
\end{rmk}

\begin{proof}
The geometric Ringel duality functor $\mathcal{R}$ of~\cite{am} is an autoequivalence of the derived category $\Db_G(\cN_G,\bk)$ which sends $\ubk_{\{0\}}$ to $\ubk_{\cN_G}[\dim\cN_{G}]$ and commutes with induction. Hence it suffices to prove that $\ubk_{\{0\}}$ is a direct summand of $\Ind_{L \subset P}^G(\ubk_{\{0\}})$ if and only if $\ell\nmid |W/W_{L}|$.

For this we can use the general results of~\cite[Section 3]{jmw}. Recall from~\cite[Lemma 2.14]{genspring1} that $\Ind_{L \subset P}^G(\ubk_{\{0\}})\cong\mu_{!}\ubk_{\widetilde{\cN}_P}[\dim \widetilde{\cN}_P]$, where 
\[
\mu:\widetilde{\cN}_P:=G\times^P\fu_P \to \fg
\] 
is the semismall morphism induced by the adjoint action. (Here $\fu_P$ is the Lie algebra of the unipotent radical of $P$.) By \cite[Proposition 3.2]{jmw}, the multiplicity of $\ubk_{\{0\}}$ as a direct summand of $\mu_{!}\ubk_{\widetilde{\cN}_P}[\dim \widetilde{\cN}_P]$ is given by the rank of the matrix of a certain intersection form. In this case, since the fibre $\mu^{-1}(0)\cong G/P$ is irreducible, the matrix is $1\times 1$ and its sole entry is the self-intersection number of $G/P$ inside $T^*(G/P)$, interpreted as an element of $\bk$. Up to sign, this self-intersection number equals the Euler characteristic of $G/P$, which is $|W/W_{L}|$ by Bruhat decomposition. The result follows.
\end{proof}

%------------------------------------------------------------------------------
\subsection{Application to $\IC(\cO_\reg,\ubk)$}
\label{ss:IC-reg}
%------------------------------------------------------------------------------

Continue with the notation $L\subset P\subset G$. In the following proposition, we denote by $\cO_\reg^L \subset \cN_L$ the regular $L$-orbit.

\begin{prop} \label{prop:oreg-quotient}
The following are equivalent:
\begin{enumerate}
\item $W_L$ contains an $\ell$-Sylow subgroup of $W$;
\item 
\label{it:oreg-quotient-cond2}
$\IC(\cO_\reg,\ubk)$ occurs as a quotient of $\Ind_{L \subset P}^G(\cF)$ for some $\cF\in\Perv_L(\cN_L,\bk)$;
\item $\IC(\cO_\reg,\ubk)$ occurs as a quotient of $\Ind_{L \subset P}^G(\IC(\cO_\reg^L,\ubk))$.
\end{enumerate}
\end{prop}

\begin{proof}
First, we notice that if $\IC(\cO_\reg,\ubk)$ occurs as a quotient of $\Ind_{L \subset P}^G(\cF)$ for some $\cF\in\Perv_L(\cN_L,\bk)$, then it occurs as such a quotient for $\cF$ simple; this can be shown by induction on the length of $\cF$, using the fact that $\Ind_{L \subset P}^G$ is exact.
In this case, since $\IC(\cO_\reg,\ubk)$ has trivial central character in the sense of~\cite[\S 5.1]{genspring2}, $\cF$ also has trivial central character. Therefore, condition~\eqref{it:oreg-quotient-cond2}
is unchanged if we replace $G$ by $G/Z(G)$. This is also clearly the case for the other conditions, so that we can assume that $Z(G)$ is trivial. This condition implies that $Z(L)$ is connected, allowing us to apply Lemma~\ref{lem:projective-cover} both to $G$ and to $L$.

The implication (1)$\Rightarrow$(2) follows from Proposition~\ref{prop:constant-summand}, since $\IC(\cO_\reg,\ubk)$ is a quotient of $\ubk_{\cN_G}[\dim\cN_{G}]$. For the implication (2)$\Rightarrow$(3), we argue as follows. As seen above, we can assume that $\IC(\cO_\reg,\ubk)$ is a quotient of $\Ind_{L \subset P}^G(\IC(\cO_L,\cE_L))$ for some 
$(\cO_L,\cE_L)\in\fN_{L,\bk}$. By~\cite[Corollary 2.15(1)]{genspring1}, the induced orbit $\mathrm{Ind}_L^G(\cO_L)$ must equal $\cO_\reg$, forcing $\cO_L=\cO_\reg^L$ and thus (since $Z(L)$ is connected) $\cE_L=\ubk$. Finally, we prove (3)$\Rightarrow$(1). Using the exactness of $\Ind_{L \subset P}^G$ again, we know that $\IC(\cO_\reg,\ubk)$ is a quotient of $\Ind_{L \subset P}^G(\ubk_{\cN_L}[\dim\cN_{L}])$. But $\Ind_{L \subset P}^G(\ubk_{\cN_L}[\dim\cN_{L}])$ is projective, since $\ubk_{\cN_L}[\dim\cN_{L}]$ is projective (see Lemma~\ref{lem:projective-cover}) and $\Ind_{L \subset P}^G$ has an exact right adjoint functor $\Res_{L \subset P}^G$.
Using Lemma~\ref{lem:projective-cover} again (this time for the group $G$), it follows that
$\ubk_{\cN_G}[\dim\cN_{G}]$ is a direct summand of $\Ind_{L \subset P}^G(\ubk_{\cN_L}[\dim\cN_{L}])$, and Proposition~\ref{prop:constant-summand} finishes the proof.
\end{proof}

%------------------------------------------------------------------------------
\subsection{$\ell$-Sylow classes and induction series of $\IC(\cO_\reg,\ubk)$}
\label{ss:l-vertex}
%------------------------------------------------------------------------------

The parabolic subgroups of $W$ that contain an $\ell$-Sylow subgroup of $W$, and are minimal with that property, form a single $W$-conjugacy class; this follows from the conjugacy of $\ell$-Sylow subgroups and the fact that the class of parabolic subgroups of $W$ is closed under conjugation and intersection. We call the corresponding $G$-conjugacy class of Levi subgroups the \emph{$\ell$-Sylow class} of $G$. Note that if $L$ is in the $\ell$-Sylow class of $G$, then the $\ell$-Sylow class of $L$ consists solely of $L$ itself.

In Table~\ref{tab:vertex} we list the $\ell$-Sylow classes of the various quasi-simple groups $G$, named by their Lie type. For a positive integer $n$, we define its base-$\ell$ digits $b_i(n)$ by $n=\sum_{i\geq 0}b_i(n)\ell^i$, $0\leq b_i(n)<\ell$. In the exceptional types, we list only the primes $\ell$ that divide $|W|$; for other $\ell$, the $\ell$-Sylow class is clearly the class of maximal tori.

\begin{table}
\[
\begin{array}{|c|c|c|c|}
\hline
G&|W|&\ell&\text{$\ell$-Sylow class}\\
\hline\hline
A_{n-1}, n\geq 2 & n! & \text{any} & \displaystyle\sum_{i>0} b_i(n) A_{\ell^i-1}\\
\hline
B_n, n\geq 2 & 2^n\cdot n!
& 2 & B_n\\
\cline{3-4}
&& >2 & \displaystyle\sum_{i>0} b_i(n) A_{\ell^i-1}\\
\hline
C_n, n\geq 3 & 2^n\cdot n!
& 2 & C_n\\
\cline{3-4}
&& >2 & \displaystyle\sum_{i>0} b_i(n) A_{\ell^i-1}\\
\hline
D_n, n\geq 4 & 2^{n-1}\cdot n!
& 2 & D_n\\
\cline{3-4}
&& >2 & \displaystyle\sum_{i>0} b_i(n) A_{\ell^i-1}\\
\hline
E_6 & 2^7\cdot 3^4\cdot 5
& 2 & D_5\\
\cline{3-4}
&& 3 & E_6\\
\cline{3-4}
&& 5 & A_4\\
\hline
E_7 & 2^{10}\cdot 3^4\cdot 5 \cdot 7
& 2 & E_7\\
\cline{3-4}
&& 3 & E_6\\
\cline{3-4}
&& 5 & A_4\\
\cline{3-4}
&& 7 & A_6\\
\hline
E_8 & 2^{14}\cdot 3^5\cdot 5^2\cdot 7
& 2,3,5 & E_8\\
\cline{3-4}
&& 7 & A_6\\
\hline
F_4 & 2^7\cdot 3^2 & 2,3 & F_4\\
\hline
G_2 & 2^2\cdot 3 & 2,3 & G_2\\
\hline
\end{array}
\]
\caption{$\ell$-Sylow classes}\label{tab:vertex}
\end{table} 

The following result is an immediate consequence of Proposition~\ref{prop:oreg-quotient}.

\begin{thm}
\label{thm:regular-series}
Let $L$ be a Levi subgroup in the $\ell$-Sylow class of $G$ and $P$ a parabolic subgroup of which $L$ is a Levi factor. Then $\IC(\cO_\reg^L,\ubk)$ is a cuspidal simple perverse sheaf in $\Perv_L(\cN_L,\bk)$, and $\IC(\cO_\reg,\ubk)$ is a quotient of $\Ind_{L\subset P}^G(\IC(\cO_\reg^L,\ubk))$. That is, $(L,\cO_\reg^L,\ubk)$ is a cuspidal datum and $(\cO_\reg,\ubk)\in\fN_{G,\bk}^{(L,\cO_\reg^L,\ubk)}$. In particular, $(\cO_\reg,\ubk)$ belongs to the principal series $\fN_{G,\bk}^{(T,\{0\},\ubk)}$ if and only if $\ell$ does not divide $|W|$, and $(\cO_\reg,\ubk)$ is cuspidal if and only if the $\ell$-Sylow class of $G$ consists of $G$ itself.\qed
\end{thm}

\begin{rmk}
Theorem~\ref{thm:regular-series} is reminiscent of the result of Geck--Hiss--Malle~\cite[Theorem 4.2]{ghm} concerning the semisimple vertex of the $\ell$-modular Steinberg character of a finite group of Lie type.\end{rmk}

\begin{rmk}
\label{rk:cuspidal-typeA}
In type $A$ (for any $\ell$) and types $B,C,D$ (for $\ell=2$), we have already determined the modular generalized Springer correspondence in~\cite{genspring1,genspring2}. Using Table~\ref{tab:vertex} one can easily check that Theorem~\ref{thm:regular-series} is consistent with those earlier results. In particular, when combined with the fact that all cuspidal pairs are supported on distinguished orbits~\cite[Proposition 2.6]{genspring2}, Theorem~\ref{thm:regular-series} gives a new proof of the classification of modular cuspidal pairs for $\GL(n)$ obtained in~\cite[Theorem 3.1]{genspring1}, independent of any counting argument.
\end{rmk}

\subsection{The case where the $\ell$-Sylow class is of type $A_{\ell-1}$}
\label{ss:linear-prime}

We conclude this section with an observation which will be useful later (see Proposition~\ref{prop:0-cusp-not-A}), concerning the special case where the $\ell$-Sylow class of $G$ consists of Levi subgroups of type $A_{\ell-1}$. Note that this assumption implies that $\ell$ divides $|W|$ exactly once, i.e.\ $\ell$ divides $|W|$ but $\ell^2$ does not. From Table~\ref{tab:vertex}, we see that the converse is almost true: if $G$ is quasi-simple and $\ell$ divides $|W|$ exactly once, then the $\ell$-Sylow class is of type $A_{\ell-1}$ except when $G$ is of type $G_2$ and $\ell=3$.

We recall a well-known result about the structure of the normalizer of a parabolic subgroup of $W$. Again, let $L$ be a Levi subgroup of $G$ containing the maximal torus $T$; also choose a Borel subgroup $B$ containing $T$ such that $B\cap L$ is a Borel subgroup of $L$. Let $\Phi$ be the root system of $(G,T)$, $\Pi\subset\Phi$ the set of simple roots specified by $B$, and $J\subset\Pi$ the set of simple roots for $L$. As usual, if $X(T)$ denotes the lattice of characters of $T$, and if we endow $\mathbb{Q}\otimes_{\mathbb{Z}} X(T)$ with an invariant scalar product, we can identify $W$ with the reflection group on $\mathbb{Q}\otimes_{\mathbb{Z}} X(T)$ generated by the reflections in the hyperplanes perpendicular to the roots in $\Phi$. The reflections corresponding to the roots in $\Pi$ form a Coxeter generating set of $W$, and $W_L$ is the parabolic subgroup generated by the reflections corresponding to the subset $J$. Thus $W_L=W_J$ in the notation of, for instance,~\cite{howlett}. We have an obvious isomorphism
\begin{equation} \label{eqn:weyl-isom}
N_G(L)/L\cong N_W(W_L)/W_L,
\end{equation}
since both sides are isomorphic to $(N_G(T)\cap N_G(L))/N_L(T)$. By~\cite[Corollary 3]{howlett}, we also have
\begin{equation} \label{eqn:semi-direct}
N_W(W_L)=W_L\rtimes W',\text{ where }W'=\{w\in W\,|\,w(J)=J\}.
\end{equation}
Note that the subgroup $W'$ depends on both $L$ and $B$.

If $L$ belongs to the $\ell$-Sylow class of $G$, then $\ell$ does not divide $|N_W(W_L)/W_L|=|N_G(L)/L|$. Hence, by Lemma~\ref{lem:splitting}\eqref{it:cuspidal-splitting-field}, the number of pairs in the induction series $\fN_{G,\bk}^{(L,\cO_\reg^L,\ubk)}$ mentioned in Theorem~\ref{thm:regular-series} equals the number of conjugacy classes of $N_G(L)/L$. This motivates the following result. 

\begin{lem} \label{lem:linear-prime}
Assume that $L$ belongs to the $\ell$-Sylow class of $G$ and is of type $A_{\ell-1}$. Then the number of conjugacy classes of $N_G(L)/L$ equals the number of $\ell$-singular conjugacy classes of $W$.
\end{lem}

\begin{proof}
Since the proof is purely Coxeter-theoretic (indeed, it still applies when $W$ is of type $H_3$ and $\ell=3$), it seems appropriate to use the notation $W_J$ rather than $W_L$. Our assumption implies that $W_J\cong\fS_\ell$, with the Coxeter generating set corresponding to the adjacent transpositions. Let $w_J$ and $c_J$ denote the longest element of $W_J$ and a Coxeter element of $W_J$ respectively. Then $c_J$ is an $\ell$-cycle. 

Any element of the subgroup $W'$ defined in~\eqref{eqn:semi-direct} commutes with $w_J$, and either fixes every simple root in $J$ or acts on $J$ by the unique diagram involution of the Dynkin diagram of type $A_{\ell-1}$. Recall that the conjugation action of $w_J$ on the Coxeter generating set of $W_J$ is by this diagram involution. Hence there is a group homomorphism $\varphi:W'\to\langle w_J\rangle$ such that for any $w\in W'$, $w\varphi(w)$ commutes with every element of $W_J$. We can thus write $N_W(W_J)$ as a direct product $W_J\times \widetilde{W'}$, where $\widetilde{W'}=\{w\varphi(w)\,|\,w\in W'\}$.

Since the unique $\ell$-singular conjugacy class of $W_J$ is the class of $c_J$, and $\ell\nmid|\widetilde{W'}|$, the $\ell$-singular conjugacy classes of $N_W(W_J)$ are in bijection with the conjugacy classes of $\widetilde{W'}$: specifically, as $w$ runs over a set of representatives for the conjugacy classes of $\widetilde{W'}$, $c_J w$ runs over a set of representatives for the $\ell$-singular conjugacy classes of $N_W(W_J)$. So to prove the claim it suffices to show that the inclusion of $N_W(W_J)$ in $W$ induces a bijection
\begin{equation} \label{eqn:non-ell-regular}
\{\textup{$\ell$-singular conjugacy classes of $N_W(W_J)$}\}\simto\{\textup{$\ell$-singular conjugacy classes of $W$}\}.
\end{equation}

As $\ell$ divides $|W|$ exactly once, the $\ell$-Sylow subgroups of $W$ are the cyclic subgroups of order $\ell$. So every element of $W$ of order $\ell$ is conjugate to $c_J$. If $y\in W$ is any $\ell$-singular element, then $y$ has order $\ell d$ for some $d$ coprime to $\ell$. So $y^d$ is conjugate  to $c_J$, and therefore $y$ is conjugate to an element $z\in W$ which commutes with $c_J$. But any such $z$ must belong to $N_W(W_J)$, because the Coxeter element $c_J$ cannot belong to a proper (conjugate-)parabolic subgroup $W_J\cap zW_Jz^{-1}$ of $W_J$. This shows that the map in~\eqref{eqn:non-ell-regular} is surjective.

To prove injectivity of~\eqref{eqn:non-ell-regular}, it is enough to show that if $w_1,w_2\in \widetilde{W'}$ and $z\in W$ satisfy $z(c_Jw_1)z^{-1}=c_Jw_2$, then $z\in N_W(W_J)$. If $w_i$ has order $d_i$ (necessarily coprime to $\ell$), then $c_Jw_i$ has order $\ell d_i$ and we see that $d_1=d_2$. Hence $zc_J^{d_1}z^{-1}=c_J^{d_1}$, implying $zc_Jz^{-1}=c_J$ and then $z\in N_W(W_J)$ as seen above. 
\end{proof}

%%%%%%%%%%%%%%%%%%%%%%%%%%%%%%%%%%%%%%%%%%%%%%%%%%%%%%%%%%%%%%%%%%%%%%%
\section{Decomposition numbers and generalized Springer basic sets}
\label{sec:basic-sets}
%%%%%%%%%%%%%%%%%%%%%%%%%%%%%%%%%%%%%%%%%%%%%%%%%%%%%%%%%%%%%%%%%%%%%%%

In~\cite{juteau}, the third author described an algorithm to determine the elements of the principal induction series $\fN_{G,\bk}^{(T,\{0\},\ubk)}$, and the (modular) Springer correspondence between this induction series and $\Irr(\bk[W])$, from the knowledge of the Springer correspondence in characteristic $0$. This algorithm relied on an equality~\cite[Theorem 5.2]{juteau} between decomposition numbers for representations of $W$ and certain decomposition numbers for perverse sheaves on $\cN_G$. In this section we will see that, under various hypotheses, a similar equality holds for non-principal induction series. This leads to an algorithm for determining the induction series associated to a cuspidal datum that is minimal for its central character (where for technical reasons we have to exclude the Spin groups).

We assume in this section that the field $\bk$ is big enough for $G$ in the sense of~\eqref{eqn:definitely-big-enough}. This implies in particular that, for a Levi subgroup $L$ of $G$, every irreducible $L$-equivariant local system $\cE_L$ on a nilpotent $L$-orbit $\cO_L$ has a central character, in the sense explained in~\cite[\S 5.1]{genspring2}.

%------------------------------------------------------------------------------
\subsection{The minimal cuspidal datum with a given central character}
\label{ss:min-cusp-datum}
%------------------------------------------------------------------------------

Recall that for any Levi subgroup $M$ of $G$, the natural homomorphism $h_M:Z(G)/Z(G)^\circ\to Z(M)/Z(M)^\circ$ is surjective (see~\cite[Corollaire 2.2]{bonnafe-elements}, for example). As in~\cite[\S 5.1]{genspring2}, we use this fact to identify central characters for $M$ with central characters for $G$ that factor through $h_M$. The following result is essentially due to Bonnaf\'e.

\begin{prop} \label{prop:bonnafe}
Let $\chi:Z(G)/Z(G)^\circ\to\bk^\times$ be any group homomorphism.
\begin{enumerate}
\item 
\label{it:minimal-Levi}
There is a Levi subgroup $L_\chi$ of $G$, unique up to $G$-conjugacy, with the property that for any Levi subgroup $M$ of $G$, $\chi$ factors through $h_M$ if and only if $M$ contains a $G$-conjugate of $L_\chi$.
\item 
\label{it:minimal-Levi-A}
The Levi subgroup $L_\chi$ is of type $A$ and is self-opposed in $G$ in the sense of~\cite[\S 1.E]{bonnafe1}.
\item There is a unique cuspidal pair $(\cO_\chi,\cE_\chi)\in\fN_{L_\chi,\bk}^\cusp$ with central character $\chi$, and the orbit $\cO_\chi$ is the regular nilpotent orbit for $L_\chi$. 
\end{enumerate}
\end{prop}

\begin{proof}
Part (1) follows from~\cite[Lemme 2.16(b)]{bonnafe-elements}, applied to the subgroup $K=\ker(\chi)$ of $Z(G)/Z(G)^\circ$. Part (2) follows from~\cite[Remarque 2.14, Lemme 2.16(a), Proposition 2.18]{bonnafe-elements}. To prove (3), by the classification of cuspidal pairs for $\SL(n)$ given in~\cite[Theorem 6.3]{genspring2}, it suffices to show that the homomorphism $Z(L_\chi)/Z(L_\chi)^\circ\to\bk^\times$ induced by $\chi$ is injective, i.e.\ that $\ker(h_{L_\chi})=\ker(\chi)$. More generally, a supplement to~\cite[Lemme 2.16]{bonnafe-elements} (in the notation of that result) is that for $L\in\cL_\mini(K)$, $\ker(h_L)=K$; this follows from the observation that it holds on each line of~\cite[Table 2.17]{bonnafe-elements}.  
\end{proof}

We will refer to $(L_\chi,\cO_\chi,\cE_\chi)$ as the \emph{minimal cuspidal datum with central character $\chi$}. When $\chi=1$, this is just the principal cuspidal datum $(T,\{0\},\ubk)$. When $\chi\neq 1$, an explicit case-by-case description of the Levi subgroup $L_\chi$ can be found in~\cite[Table 2.17]{bonnafe-elements}. 

The analogue for general $\chi$ of the Weyl group $W=N_G(T)/T$ is the group $W(\chi):=N_G(L_\chi)/L_\chi$, which is a subquotient of $W$ by~\eqref{eqn:weyl-isom}. Since $L_\chi$ is self-opposed in $G$, $W(\chi)$ is a Coxeter group in a natural way (see~\cite[Proposition 1.12]{bonnafe1}). When $M$ is a Levi subgroup of $G$ containing $L_\chi$, the group $N_M(L_\chi)/L_\chi$ is a (conjugate-)parabolic subgroup $W_M(\chi)$ of $W(\chi)$. Note that we have an isomorphism generalizing~\eqref{eqn:weyl-isom}:
\begin{equation} \label{eqn:generalized-weyl-isom}
N_G(M)/M\cong N_{W(\chi)}(W_M(\chi))/W_M(\chi),
\end{equation} 
since both sides are isomorphic to $(N_G(L_\chi)\cap N_G(M))/N_M(L_\chi)$. (Note that this argument uses the uniqueness of $L_\chi$ in Proposition~\ref{prop:bonnafe}\eqref{it:minimal-Levi}.) In particular, if $(M,\cO_M,\cE_M)$ is any cuspidal datum for $G$ where $\cE_M$ has central character $\chi$, the corresponding group $N_G(M)/M$ is a subquotient of $W(\chi)$.

%------------------------------------------------------------------------------
\subsection{Equality of decomposition numbers}
\label{ss:decomposition-numbers}
%------------------------------------------------------------------------------

In this subsection we fix a finite integral extension $\O$ of $\Zl$, denote by $\K$ its field of fractions (which is of characteristic $0$), and take $\bk$ to be its residue field (of characteristic $\ell$). We assume that $\K$ and $\bk$ are big enough for $G$. 

We use the notation $\fN_{G,\K}$, $\fM_{G,\K}$, etc.\ to denote the sets analogous to $\fN_{G,\bk}$, $\fM_{G,\bk}$, etc. The analogue of Theorem~\ref{thm:mgsc-intro} holds when $\bk$ is replaced by $\K$, by Lusztig's results in the $\Qlb$ setting~\cite{lusztig}. For any cuspidal datum $(L,\cO_L,\cE_L)\in\fM_{G,\bk}$, we denote by
\[
\Psi^{(L,\cO_L,\cE_L)}_\bk : \Irr(\bk[N_G(L)/L]) \simto \fN_{G,\bk}^{(L,\cO_L,\cE_L)} 
\]
the bijection~\eqref{eqn:bijection}. For $(L,\cO_L,\cE_L)\in\fM_{G,\K}$, we use the analogous notation
\[
\Psi^{(L,\cO_L,\cE_L)}_\K : \Irr(\K[N_G(L)/L]) \simto \fN_{G,\K}^{(L,\cO_L,\cE_L)} 
\]
for the bijection defined in the way analogous to~\eqref{eqn:bijection}. Since this definition uses Fourier transform, it differs from the bijection defined by Lusztig~\cite{lusztig} by twisting with the sign representation of $N_G(L)/L$, which is always a Coxeter group when $(L,\cO_L,\cE_L)\in\fM_{G,\K}$; see~\cite[\S3.7 and Theorem 3.8(c)]{em}.

If $H$ is a finite group, for $E$ in $\Irr(\K[H])$ and $F$ in $\Irr(\bk[H])$, we will consider the \emph{decomposition number}
\[
d^H_{E,F}:=[\bk \otimes_\O E_\O : F],
\]
where $E_\O$ is an $\O$-form of $E$. (It is well known that this number does not depend on the choice of $\O$-form.) One can also define analogous decomposition numbers for perverse sheaves on nilpotent cones; see~\cite[\S 2.7]{genspring1}. For $(\cO,\cE)$ in $\fN_{G,\K}$ and $(\cC,\cF)$ in $\fN_{G,\bk}$, we set
\[
d^{\cN_G}_{(\cO,\cE), (\cC,\cF)} := [\bk \otimes_\O^L \cM_\O : \IC(\cC,\cF)],
\]
where $\cM_\O$ is any $\O$-form of $\IC(\cO,\cE)$, i.e.~any torsion-free $G$-equivariant $\O$-perverse sheaf on $\cN_G$ such that $\K \otimes_\O \cM_\O \cong \IC(\cO,\cE)$.

For the remainder of this subsection we consider the following situation. Let $(L,\cO_L,\cE_L^\K) \in \fM_{G,\K}$, let $\cE_L^\O$ be an $\O$-form of $\cE_L^\K$, and set $\cE_L^\bk:=\bk \otimes_\O \cE_L^\O$, an $L$-equivariant $\bk$-local system on $\cO_L$. Also let $\tilde{\chi} : Z(G)/Z(G)^\circ \to \K^\times$ be the central character of $\cE_L^\K$. Then $\tilde{\chi}$ takes values in $\O^\times$, and hence induces a character $\chi : Z(G)/Z(G)^\circ \to \bk^\times$. We need to make the following assumptions:
\begin{equation}
\label{eqn:assumption-dec-numbers}
\begin{array}{c}
\text{$\cE_L^\bk$ is an irreducible $L$-equivariant $\bk$-local system on $\cO_L$,}\\
\text{and $(\cO_L,\cE_L^\bk)$ is the unique cuspidal pair in $\fN_{L,\bk}^\cusp$ with central character $\chi$.}
\end{array}
\end{equation}
The assumption that $\cE_L^\bk$ is irreducible implies that $(\cO_L,\cE_L^\bk)$ is indeed a cuspidal pair with central character $\chi$, by~\cite[Proposition 2.2]{genspring1}. We say that $(\cO_L,\cE_L^\bk)$ is the \emph{modular reduction} of the cuspidal pair $(\cO_L,\cE_L^\K)$; the assumption that $L$ has no other cuspidal pairs with central character $\chi$ implies that $\IC(\cO_L,\cE_L^\bk)$ is the modular reduction of $\IC(\cO_L,\cE_L^\K)$ in the sense of~\cite[\S 2.7]{genspring1}, so this terminology is not misleading.

\begin{rmk} \label{rmk:modular-reduction-of-cuspidal}
Lusztig's classification of cuspidal pairs in characteristic zero~\cite{lusztig} shows that the first assumption in~\eqref{eqn:assumption-dec-numbers} is almost always true. Among quasi-simple groups, it is only the Spin groups that have cuspidal pairs where the local system has rank bigger than one; and even for the Spin groups, we observed in~\cite[\S 8.4]{genspring2} that if $\ell\neq 2$, the modular reduction of the cuspidal local system remains irreducible. It follows from Theorem~\ref{thm:cuspidal-facts} (whose proof does not involve the results of this section) that the second assumption in~\eqref{eqn:assumption-dec-numbers} is satisfied whenever $\ell$ is good for $L$; for instance, it is always true when $L$ is of type $A$.
\end{rmk}

By~\cite[Lemma~5.2]{genspring2}, the central character assumption in~\eqref{eqn:assumption-dec-numbers} implies that the pair $(\cO_L', (\cE_L^\bk)')$ obtained from $(\cO_L, \cE_L^\bk)$ by Fourier transform as in~\cite[Equation~(2.1)]{genspring2} coincides with $(\cO_L, \cE_L^\bk)$. The same property over $\K$ holds for the same reason, see~\cite[Remark~2.13]{genspring1}.

Recall (see~\cite[\S 2.6]{genspring1}) that to each Levi subgroup $M \subset G$ and nilpotent orbit $\cO_M \subset \cN_M$, Lusztig associates a locally-closed subvariety $Y_{(M,\cO_M)} \subset \fg$.
Applying~\cite[Theorem~3.1]{genspring2} over both $\K$ and $\bk$, we obtain local systems $\overline{\cE_L^\K}$ and $\overline{\cE_L^\bk}$ on 
$Y_{(L,\cO_L)}$.

\begin{lem}
\label{lem:Ebar}
Assume that~\eqref{eqn:assumption-dec-numbers} holds. Then
there exists an $\O$-form $\overline{\cE_L^\O}$ of $\overline{\cE_L^\K}$ such that $\bk \otimes_\O \overline{\cE_L^\O} \cong \overline{\cE_L^\bk}$.
\end{lem}

\begin{proof}
Recall the setting of~\cite[Theorem~3.1]{genspring2}. We fix $x \in \cO_L$, and let $V^\K=(\cE_L^\K)_x$ be the representation of $A_L(x)$ associated to the local system $\cE_L^\K$ on $\cO_L$. As explained in~\cite[Remark~3.3]{genspring2}, the statement of~\cite[Theorem~3.1]{genspring2} can be interpreted as providing a canonical extension of the action of $A_L(x)$ on $V$ to the larger group $A_{N_G(L)}(x)$. Hence one can choose an $\O$-form $V^\O$ of $V^\K$, considered as a representation of $A_L(x)$, which also has the property that the action can be extended to $A_{N_G(L)}(x)$.

We claim that
\begin{equation}
\label{eqn:abs-irred-O}
\Hom_{A_L(x)}(V^\O, V^\O) = \O.
\end{equation}
In fact, we have a canonical isomorphism $\bk \lotimes_\O R\Hom_{A_L(x)}(V^\O, V^\O) \simto R\Hom_{A_L(x)}(\bk \otimes_\O V^\O, \bk \otimes_\O V^\O)$. From this (and the fact that any complex in the derived category of the category of $\O$-modules is isomorphic to the direct sum of its cohomology objects) we deduce that $\Hom_{A_L(x)}(V^\O, V^\O)$ is $\O$-free and that the natural morphism $\bk \otimes_\O \Hom_{A_L(x)}(V^\O, V^\O) \to \Hom_{A_L(x)}(\bk \otimes_\O V^\O, \bk \otimes_\O V^\O)$ is injective. Now by assumption $\bk \otimes_\O V^\O$ is irreducible, and hence absolutely irreducible, so $\Hom_{A_L(x)}(\bk \otimes_\O V^\O, \bk \otimes_\O V^\O) = \bk$. We deduce~\eqref{eqn:abs-irred-O}.

With our choice of $V^\O$, using in particular~\eqref{eqn:abs-irred-O}, one can see that the arguments in~\cite[\S 3.5]{genspring2} apply to the local system $\cE_L^\O$ associated to $V^\O$, and provide a canonical $N_G(L)/L$-equivariant structure on the associated local system $\widetilde{\cE_L^\O}$ which induces, upon extension of scalars, the corresponding equivariant structures on $\bk \otimes_\O \widetilde{\cE_L^\O} = \widetilde{\cE_L^\bk}$ and $\K \otimes_\O \widetilde{\cE_L^\O} = \widetilde{\cE_L^\K}$. Applying a variant of the equivalence in~\cite[Equation~(3.7)]{genspring2} over $\O$, we obtain a torsion-free $\O$-local system $\overline{\cE_L^\O}$ on $Y_{(L,\cO_L)}$ such that $\bk \otimes_\O \overline{\cE_L^\O} \cong \overline{\cE_L^\bk}$ and $\K \otimes_\O \overline{\cE_L^\O} \cong \overline{\cE_L^\K}$, which finishes the proof.
\end{proof}

The following proposition is a generalization of~\cite[Theorem~5.2]{juteau}, with an identical proof.

\begin{prop}
\label{prop:equality-decomp-number}
Assume that~\eqref{eqn:assumption-dec-numbers} holds. For $E \in \Irr(\K[N_G(L)/L])$ and $F \in\Irr(\bk[N_G(L)/L])$ we have
\[
d^{N_G(L)/L}_{E,F} = d^{\cN_G}_{\Psi^{(L,\cO_L,\cE_L^\K)}_\K(E), \Psi^{(L,\cO_L,\cE_L^\bk)}_\bk(F)}.
\]
\end{prop}

\begin{proof}
If $\E \in \{\K, \O, \bk\}$ and if $V$ is a representation of $N_G(L)/L$ over $\E$, we denote by $\cL_V$ the $\E$-local system on $Y_{(L,\cO_L)}$ associated to $V$ as in~\cite[\S\S 3.1--3.2]{genspring2}.

By definition, the simple perverse sheaf on $\cN_G$ associated to $\Psi^{(L,\cO_L,\cE_L^\K)}_\K(E)$, resp.~$\Psi^{(L,\cO_L,\cE_L^\bk)}_\bk(F)$, is 
\[
(\bT_{\fg}^\K)^{-1}(\IC(Y_{(L,\cO_L)}, \overline{\cE_L^\K} \otimes_\K \cL_E)), \qquad \text{resp.} \qquad (\bT_{\fg}^\bk)^{-1}(\IC(Y_{(L,\cO_L)}, \overline{\cE_L^\bk} \otimes_\bk \cL_F)),
\]
where $\bT_\fg^\K$, resp.~$\bT_\fg^\bk$, is the Fourier transform over $\K$, resp.~$\bk$, see~\cite[\S 2.4]{genspring1}.
Now we let $E_\O$ be an $\O$-form of $E_\K$. Using~\cite[Theorem~3.1(4)]{genspring2} we have
\[
[\bk \otimes_\O E_\O : F] = [\bk \otimes_\O \cL_{E_\O} : \cL_F] = [\overline{\cE_L^\bk} \otimes_\bk (\bk \otimes_\O \cL_{E_\O}) : \overline{\cE_L^\bk} \otimes_\bk \cL_F] = [\bk \otimes_\O (\overline{\cE_L^\O} \otimes_\O \cL_{E_\O}) : \overline{\cE_L^\bk} \otimes_\bk \cL_F],
\]
where $\overline{\cE^\O_L}$ is as in Lemma~\ref{lem:Ebar}. Using~\cite[Corollary~2.5]{juteau} we deduce that
\[
[\bk \otimes_\O E_\O : F]
= [\bk \lotimes_\O \IC(Y_{(L,\cO_L)}, \overline{\cE_L^\O} \otimes_\O \cL_{E_\O}) : \IC(Y_{(L,\cO_L)},\overline{\cE_L^\bk} \otimes_\bk \cL_F)].
\]
Finally, since Fourier transform is an equivalence which commutes with extension of scalars (see~\cite[Remark~2.23]{genspring1}), we obtain that
\[
[\bk \otimes_\O E_\O : F]
= [\bk \lotimes_\O (\bT_\fg^\O)^{-1}( \IC(Y_{(L,\cO_L)}, \overline{\cE_L^\O} \otimes_\O \cL_{E_\O})) : (\bT_\fg^\bk)^{-1}(\IC(Y_{(L,\cO_L)},(\overline{\cE_L^\bk} \otimes_\bk \cL_F))],
\]
where $\bT_\fg^\O$ is the Fourier transform over $\O$. This proves the claim.
\end{proof}

%------------------------------------------------------------------------------
\subsection{Basic set datum attached to the minimal cuspidal datum}
\label{ss:basic-set}
%------------------------------------------------------------------------------

Continue with the notation $\K$, $\O$, $\bk$ from~\S\ref{ss:decomposition-numbers}. Let $\chi:Z(G)/Z(G)^\circ\to\bk^\times$ be any homomorphism.

\begin{lem} \label{lem:lift}
There exists a lift $\tilde{\chi}:Z(G)/Z(G)^\circ\to\O^\times$ of $\chi$ such that $\ker(\tilde{\chi})=\ker(\chi)$. 
\end{lem}

\begin{proof}
Let $N$ be the exponent of the finite group $Z(G)/Z(G)^\circ$. By the assumption that $\K$ is big enough for $G$, the group of $N$-th roots of unity $\mu_N(\O)=\mu_N(\K)$ is cyclic of order $N$. The canonical homomorphism $\mu_N(\O)\to\mu_N(\bk)$ is the projection onto the second factor of the direct product decomposition $\mu_N(\O)=\mu_N(\O)_\ell\times\mu_N(\O)_{\ell'}$, so it has a section $s:\mu_N(\bk)\to\mu_N(\O)$, and we can set $\tilde{\chi}=s\circ\chi$.
\end{proof}

Henceforth we will fix some $\tilde{\chi}$ as in Lemma~\ref{lem:lift}. We denote by $\fN_{G,\bk}^\chi \subset \fN_{G,\bk}$, resp.~$\fN_{G,\K}^{\tilde{\chi}} \subset \fN_{G,\K}$, the subset consisting of the pairs $(\cO,\cE)$ where $\cE$ has central character $\chi$, resp.~$\tilde{\chi}$.

We consider the setting of \S\ref{ss:decomposition-numbers}, taking as our cuspidal datum $(L,\cO_L, \cE_L^\K)$ the minimal cuspidal datum with central character $\tilde\chi$ over $\K$, as defined in \S\ref{ss:min-cusp-datum}. Since $\ker(\tilde\chi)=\ker(\chi)$, we can assume that $L_{\tilde\chi}=L_\chi$ and $\cO_{\tilde\chi}=\cO_{L_{\tilde\chi}}^\reg=\cO_{L_\chi}^\reg=\cO_\chi$, so we can write this minimal cuspidal datum with central character $\tilde\chi$ over $\K$ as $(L_\chi,\cO_\chi,\cE_{\tilde\chi}^\K)$. In this case, by Remark~\ref{rmk:modular-reduction-of-cuspidal}, assumption~\eqref{eqn:assumption-dec-numbers} is satisfied since $L_\chi$ is of type $A$ (see Proposition~\ref{prop:bonnafe}\eqref{it:minimal-Levi-A}). We write the modular reduction of $(\cO_\chi,\cE_{\tilde\chi}^\K)$ as $(\cO_\chi,\cE_{\chi}^\bk)$. Then $(L_\chi, \cO_\chi, \cE_\chi^\bk)$ is the minimal cuspidal datum with central character $\chi$ over $\bk$.

The following result is a generalization of~\cite[Proposition~5.4]{juteau}, with an identical proof.

\begin{prop}
\label{prop:mod-redu-min-series}
Let $(\cO,\cE) \in \fN_{G,\K}^{\tilde\chi}$ and $(\cC, \cF) \in \fN_{G,\bk}^\chi$. If $(\cO, \cE) \notin \fN_{G,\K}^{(L_\chi, \cO_\chi, \cE_{\tilde\chi}^\K)}$ and $d^{\cN_G}_{(\cO,\cE), (\cC,\cF)} \neq 0$, then $(\cC, \cF) \notin \fN_{G,\bk}^{(L_\chi, \cO_\chi, \cE_\chi^\bk)}$.
\end{prop}

\begin{proof}
Let us set
\[
Y_\chi = \bigcup_{\substack{(L,\cO_L) \\ L \supset L_\chi}} Y_{(L,\cO_L)},
\]
where the union runs over Levi subgroups $L$ containing $L_\chi$, and nilpotent orbits $\cO_L \subset \cN_L$. Since $\cO_\chi$ is the regular orbit $\cO_{L_\chi}^\reg$ (see Proposition~\ref{prop:bonnafe}), it follows from~\cite[Proposition~6.5]{lusztig-cusp2} that $Y_\chi$ is the closure of $Y_{(L_\chi, \cO_\chi)}$ in $\fg$; in particular $Y_\chi$ is a closed subvariety in $\fg$, and $Y_{(L_\chi, \cO_\chi)}$ is an open subvariety in $Y_\chi$.

Suppose that $(M,\cO_M,\cE_M^\K)\in\fM_{G,\K}$ labels the induction series containing $(\cO, \cE)$. By~\cite[Lemma~2.1]{genspring2}, the perverse sheaf $\bT_\fg^\K(\IC(\cO,\cE))$ is supported on the closure of $Y_{(M,\cO_M')}$ for some nilpotent orbit $\cO_M' \subset \cN_M$. By~\cite[Lemma~5.1]{genspring2}, the central character of $\cE_M^\K$ is the same as that of $\cE$, so $\tilde\chi$ must factor through the surjection $h_M:Z(G)/Z(G)^\circ\to Z(M)/Z(M)^\circ$. By Proposition~\ref{prop:bonnafe}, this implies that $M$ contains a $G$-conjugate of $L_\chi$; hence $Y_{(M,\cO_M')}\subset Y_\chi$. The assumption that $(M,\cO_M,\cE_M^\K)$ is not conjugate to $(L_\chi, \cO_\chi, \cE_{\tilde\chi}^\K)$ implies that $Y_{(M,\cO_M')}\neq Y_{(L_\chi,\cO_\chi)}$, so $\bT_\fg^\K(\IC(\cO,\cE))$ is supported on $Y_\chi \smallsetminus Y_{(L_\chi, \cO_\chi)}$.

If $\cE^\O$ is an $\O$-form of $\cE$, then by~\cite[Proposition~2.8]{juteau}, $\bT_\fg^\O(\IC(\cO,\cE^\O))$ is also supported on $Y_\chi \smallsetminus Y_{(L_\chi, \cO_\chi)}$, so that none of the composition factors of $\bk \lotimes_\O \bT_\fg^\O(\IC(\cO,\cE^\O)) \cong \bT_\fg^\bk(\bk \lotimes_\O \IC(\cO,\cE^\O))$ can be an $\IC$-extension of a local system on $Y_{(L_\chi, \cO_\chi)}$. In particular, $\bT_\fg^\bk(\IC(\cC, \cF))$ is not such an $\IC$-extension. By~\cite[Lemma~2.1]{genspring2}, this implies that $(\cC, \cF) \notin \fN_{G,\bk}^{(L_\chi, \cO_\chi, \cE_\chi^\bk)}$.
\end{proof}

For the remainder of this subsection we assume that $G$ is quasi-simple and not a Spin group. Under this assumption, one can easily check by case-by-case calculation (using, for instance, the tables on~\cite[p.~92 and pp.~128--134]{cm}) that for any $x \in \cN_G$, if $Z(x)$ denotes the image of $Z(G)$ in $A_G(x)$ then the exact sequence of groups
\begin{equation}
\label{eqn:exact-sequence-refined}
1 \to Z(x) \to A_G(x) \to A_{G/Z(G)}(x) \to 1
\end{equation}
induced by~\eqref{eqn:exact-seq} splits.  We choose once and for all, for any nilpotent orbit $\cO \subset \cN_G$, an element $x_\cO \in \cO$ and a splitting of the corresponding exact sequence~\eqref{eqn:exact-sequence-refined}. If $(\cO,\cE^\bk)$ belongs to $\fN_{G,\bk}^\chi$, then $\chi$ factors through a character of $Z(x_{\cO})$, which we will also denote by $\chi$; similarly if $(\cO,\cE^\K)$ belongs to $\fN_{G,\bk}^{\tilde\chi}$.

Recall the following terminology introduced in~\cite{jls}. Here $H$ is any finite group.
\begin{defn}
\label{def:basic}
A \emph{basic set datum} for $H$ is a pair $(\leq, \beta)$ where $\leq$ is an order relation on $\Irr(\K[H])$ and
$\beta:\Irr(\bk[H]) \to \Irr(\K[H])$ is an injection, which satisfy
\begin{gather*}
 d^H_{\beta(V), V} = 1;\\
 d^H_{U, V} \neq 0 \ \Rightarrow \ U \leq \beta(V).
\end{gather*}
\end{defn}

We will apply this concept when $H$ is the Coxeter group $W(\chi)=N_G(L_\chi)/L_\chi$ discussed in \S\ref{ss:min-cusp-datum}.

For any $G$-orbit $\cO \subset \cN_G$, we consider the basic set datum for $A_{G/Z(G)}(x_\cO)$ defined in~\cite[\S 5.3]{juteau}, and denote it by $(\leq_\cO, \beta_\cO)$. Then we define an injection
\[
\beta^\chi_\cN : \fN_{G,\bk}^\chi \hookrightarrow \fN_{G,\K}^{\tilde\chi}
\]
by setting $\beta^\chi_\cN(\cO, \cL_{\rho \times \chi}) = (\cO, \cL_{\beta_\cO(\rho) \times \tilde\chi})$. (Here $\rho \in \Irr(\bk[A_G(x_\cO)])$, and we denote by $\cL_{\rho \times \chi}$ the local system on $\cO$ associated to the representation $\rho \times \chi$ of $A_G(x_\cO) = A_{G/Z(G)}(x_\cO) \times Z(x)$, and similarly for $\cL_{\beta_\cO(\rho) \times \tilde\chi}$.) We also define an order $\leq^\chi_S$ on $\Irr(\K[W(\chi)])$ as follows. Let $E_i$ for $i\in\{1,2\}$ be in $\Irr(\K[W(\chi)])$, and define nilpotent orbits $\cO_i$ and $\K$-representations $\rho_i$ of $A_{G/Z(G)}(x_{\cO_i})$ such that $\Psi_\K^{(L_\chi, \cO_\chi, \cE_{\tilde\chi}^\K)}(E_i)=(\cO_i,\cL_{\rho_i \times \tilde\chi})$ for $i \in \{1,2\}$. Then we set
\begin{equation} \label{eqn:springer-order}
E_1 \leq^\chi_S E_2 \qquad \text{iff} \qquad \cO_2 \subset \overline{\cO_1} \ \text{ and } \ \rho_1 \leq_\cO \rho_2 \ \text{ if } \ \cO_1=\cO_2.
\end{equation}
(In this notation, ``$S$'' stands for ``Springer''.)

The following result is a generalization of~\cite[Theorem~5.3]{juteau} (or equivalently of~\cite[Theorem 3.13]{jls}), with an identical proof.

\begin{thm}
\label{thm:basic-set-chi}
There exists an injection $\beta^\chi_S : \Irr(\bk[W(\chi)]) \hookrightarrow \Irr(\K[W(\chi)])$ which makes the following diagram commutative:
\begin{equation}
\label{eqn:diagram-basic-set}
\vcenter{
\xymatrix@C=3cm{
\Irr(\bk[W(\chi)]) \ar@{^{(}->}[r]^-{\Psi_\bk^{(L_\chi, \cO_\chi, \cE_\chi^\bk)}} \ar@{^{(}->}[d]_-{\beta^\chi_S} & \fN_{G,\bk}^\chi \ar@{^{(}->}[d]^-{\beta^\chi_\cN} \\
\Irr(\K[W(\chi)]) \ar@{^{(}->}[r]^-{\Psi_\K^{(L_\chi, \cO_\chi, \cE_{\tilde\chi}^\K)}} & \fN_{G,\K}^{\tilde\chi}.
}
}
\end{equation}
Moreover, this map has the following properties:
\begin{gather}
\label{eqn:cond-beta-1}
\forall F \in \Irr(\bk[W(\chi)]), \ d^{W(\chi)}_{\beta^\chi_S(F),F}=1; \\
\label{eqn:cond-beta-2}
\forall E \in \Irr(\K[W(\chi)]), \forall F \in \Irr(\bk[W(\chi)]), \ d^{W(\chi)}_{E,F} \neq 0 \Rightarrow E \leq^\chi_S \beta^\chi_S(F).
\end{gather}
In particular, the pair $(\leq^\chi_S, \beta^\chi_S)$ is a basic set datum for $W(\chi)$ in the sense of Definition~{\rm \ref{def:basic}}.
\end{thm}

\begin{proof}
Let $F \in \Irr(\bk[W(\chi)])$, and let $(\cO, \cF) := \Psi_\bk^{(L_\chi, \cO_\chi, \cE_\chi^\bk)}(F)$. Let also $\cD$ be the $\K$-local system on $\cO$ such that $\beta^\chi_\cN(\cO, \cF)=(\cO, \cD)$. Then $d^{\cN_G}_{(\cO, \cD), (\cO, \cF)}=1$, so that by Proposition~\ref{prop:mod-redu-min-series} we have $(\cO,\cD) = \Psi_\K^{(L_\chi, \cO_\chi, \cE_{\tilde\chi}^\K)}(D)$ for some (unique) $D \in \Irr(\K[W(\chi)])$; then we set $\beta^\chi_S(F):=D$. This defines the map $\beta^\chi_S$, and~\eqref{eqn:diagram-basic-set} commutes by definition.

What remains to be proved is conditions~\eqref{eqn:cond-beta-1} and~\eqref{eqn:cond-beta-2}. For $F \in \Irr(\bk[W(\chi)])$, using the notation $\cO, \cF, \cD$ as above, by Proposition~\ref{prop:equality-decomp-number} and the definition of $\cD$ we have
\[
d^{W(\chi)}_{\beta^\chi_S(F),F} = d^{\cN_G}_{\Psi_\K^{(L_\chi, \cO_\chi, \cE_{\tilde\chi}^\K)}(\beta_S^\chi(F)), \Psi_\bk^{(L_\chi, \cO_\chi, \cE_\chi^\bk)}(F)} = d^{\cN_G}_{(\cO,\cD), (\cO,\cF)} = 1,
\]
which proves~\eqref{eqn:cond-beta-1}. On the other hand, let $E \in \Irr(\K[W(\chi)])$ and $F \in \Irr(\bk[W(\chi)])$, and assume that $d^{W(\chi)}_{E,F} \neq 0$. Let $(\cC,\cE):=\Psi_\K^{(L_\chi, \cO_\chi, \cE_{\tilde\chi}^\K)}(E)$, $(\cO, \cF):=\Psi_\bk^{(L_\chi, \cO_\chi, \cE_\chi^\bk)}(F)$, and define the $\K$-local system $\cD$ on $\cO$ as above. Let $\rho_E$ be the $\K$-representation of $A_{G/Z(G)}(x_\cC)$ such that $\cE$ is the local system associated to the representation $\rho_E \times \tilde\chi$ of $A_G(x_\cC) = A_{G/Z(G)}(x_\cC) \times Z(x_\cC)$, and let $\rho_D$ be the $\K$-representation of $A_{G/Z(G)}(x_\cO)$ defined similarly using $\cD$. Then using Proposition~\ref{prop:equality-decomp-number} again we have
\[
d^{\cN_G}_{(\cC, \cE), (\cO, \cF)} 
= d^{W(\chi)}_{E,F} \neq 0,
\]
which implies that $\cO \subset \overline{\cC}$ and that, if $\cO=\cC$, then $\rho_E \leq_\cO \rho_D$. (Here we use the facts that $(\leq_\cO, \beta_\cO)$ is a basic set datum and that the $\IC$ functor preserves decomposition numbers, see~\cite[Corollary~2.4]{juteau}.) This proves~\eqref{eqn:cond-beta-2}, and finishes the proof of Theorem~\ref{thm:basic-set-chi}.
\end{proof}

%------------------------------------------------------------------------------
\subsection{Determination of the minimal induction series}
\label{ss:determination-minimal}
%------------------------------------------------------------------------------

Assume that $G$ is quasi-simple and not a Spin group, and that $\bk$ is big enough for $G$. Let $\chi:Z(G)\to\bk^\times$ be any character. As in the case $\chi=1$ considered in~\cite{juteau}, Theorem~\ref{thm:basic-set-chi} can be used as the basis of an algorithm for the determination of $\fN^{(L_\chi, \cO_\chi, \cE_\chi^\bk)}_{G,\bk}$ and $\Psi^{(L_\chi, \cO_\chi, \cE_\chi^\bk)}_{\bk}$, provided the decomposition matrix for $W(\chi)$ is known. (In fact, one can also perform a similar algorithm using only the character table of $W(\chi)$; see~\cite[Section~9]{juteau} for details.)

By Lemma~\ref{lem:splitting}\eqref{it:indep-of-k}, we can assume that $\bk$ is part of a triple $(\K,\O,\bk)$ as in \S\S\ref{ss:decomposition-numbers}--\ref{ss:basic-set}, and fix a lift $\tilde\chi:Z(G)\to\O^\times$ of $\chi$ as in Lemma~\ref{lem:lift}. 

Let $F \in \Irr(\bk[W(\chi)])$. If $E \in \Irr(\K[W(\chi)])$ is maximal (with respect to the order $\leq^\chi_S$) with the property that $d^{W(\chi)}_{E,F} \neq 0$, then conditions~\eqref{eqn:cond-beta-1}--\eqref{eqn:cond-beta-2} imply that $E=\beta_S^\chi(F)$. Moreover, the commutativity of~\eqref{eqn:diagram-basic-set} implies that, in this case, $\Psi_\bk^{(L_\chi, \cO_\chi, \cE_\chi^\bk)}(F)$ is the unique pair $(\cO,\cE) \in \fN_{G,\bk}^\chi$ such that $\beta^\chi_\cN(\cO,\cE)=\Psi_\K^{(L_\chi, \cO_\chi, \cE_{\tilde\chi}^\K)}(E)$. (Recall that the map $\Psi_\K^{(L_\chi, \cO_\chi, \cE_{\tilde\chi}^\K)}$ is known in all cases; see in particular~\cite{lus-spalt, spaltenstein}.) This determines the bijection $\Psi_\bk^{(L_\chi, \cO_\chi, \cE_\chi^\bk)}$.

We now apply this algorithm to the two cases of most interest in the exceptional groups, namely in type $E_6$ when $\ell=2$, and in type $E_7$ when $\ell=3$, with $\chi$ nontrivial in both cases. In these cases the lift $\tilde\chi$ is unique, so we denote it simply by $\chi$. (See Section~\ref{sec:reduction} below for details on our notational conventions.)

%---------------------------------------------------------------------
\subsubsection{Case of $E_6$ for $\ell = 2$}
\label{sss:E6ell2}
%---------------------------------------------------------------------

In this case, for $\chi \neq 1$ we have $(L_\chi, \cO_\chi)=(2A_2, [3]^2)$, and $W(\chi)$ is the Weyl group of type $G_2$. The corresponding decomposition matrix is shown on the right of Table~\ref{tab:dec-matric-E6ell2chinot1}. (See~\cite[Section 9]{juteau} for its computation.)
Here the first column displays the ordinary characters of $W(\chi)$, denoted as in~\cite{juteau} (i.e.~as in~\cite{carter}, except that we use the symbol `$\chi$' instead of `$\phi$'), and the second column their image under $\Psi_{\K}^{(L_\chi, \cO_\chi, \cE_\chi^\K)}$. Recall that the latter bijection differs from that computed in~\cite{spaltenstein} by a sign twist, and that an indeterminacy in~\cite{spaltenstein} was resolved in~\cite[\S24.10]{charsh5} and~\cite[Theorem 5.5]{lusztig-gsc}. We have ordered the rows so that the total order on $\Irr(\K[W(\chi)])$ obtained by reading from bottom to top refines the partial order $\leq^\chi_S$ defined in~\eqref{eqn:springer-order}.

\begin{table}
\[
\begin{array}{|c|c|cc|}
\hline
E & (\cO,\cE^\K)&& \\
\hline\hline
\chi_{1,0}   & (2A_2, \chi)       &1&.\\
\chi_{1,3}'' & (2A_2{+}A_1, \chi)  &1&.\\
\chi_{2,1}  & (A_5, \chi)         &.&1\\
\chi_{2,2}  & (E_6(a_3), 1\times\chi) &.&1\\
\chi_{1,3}' & (E_6(a_1), \chi)    &1&.\\
\chi_{1,6}  & (E_6, \chi)         &1&.\\
\hline
\end{array}
\]
\caption{Decomposition matrix for $E_6$ when $\ell=2$ and $\chi \neq 1$}
\label{tab:dec-matric-E6ell2chinot1}
\end{table}

From this table one obtains that
\[
\fN^{(L_\chi, \cO_\chi, \cE_\chi^\bk)}_{G,\bk} = \{(2A_2, \chi), \, (A_5, \chi)\},
\]
and that $(2A_2, \chi)$ corresponds to the trivial representation of $W(\chi)$, while $(A_5, \chi)$ corresponds to the unique nontrivial irreducible representation over $\bk$.

%---------------------------------------------------------------------
\subsubsection{Case of $E_7$ for $\ell = 3$}
\label{sss:E7ell3}
%---------------------------------------------------------------------

In this case, for $\chi \neq 1$ we have $(L_\chi, \cO_\chi)=((3A_1)'', [2]^3)$, and $W(\chi)$ is the Weyl group of type $F_4$. The decomposition matrix is shown in Table~\ref{tab:dec-matric-E7ell3chinot1}. (See~\cite{juteau} for references on this computation.) The table should be interpreted as in the preceding case.

\begin{table}
\[
\begin{array}{|c|c|cccccccccccccc|}
\hline
E & (\cO,\cE^\K)&&&&&&&&&&&&&& \\
\hline\hline
\chi_{1,0}&((3A_1)'',\chi)&1&.&.&.&.&.&.&.&.&.&.&.&.&.\\
\chi_{2,4}''&(4A_1,\chi)&1&1&.&.&.&.&.&.&.&.&.&.&.&.\\
\chi_{1,12}''&(A_2{+}3A_1,\chi)&.&1&.&.&.&.&.&.&.&.&.&.&.&.\\
\chi_{4,1}&((A_3{+}A_1)'',\chi)&.&.&1&.&.&.&.&.&.&.&.&.&.&.\\
\chi_{8,3}''&(A_3{+}2A_1,\chi)&.&.&1&1&.&.&.&.&.&.&.&.&.&.\\
\chi_{2,4}'&(D_4(a_1){+}A_1,1\times \chi)&1&.&.&.&1&.&.&.&.&.&.&.&.&.\\
\chi_{9,2}&(D_4(a_1){+}A_1,\varepsilon\times \chi)&.&.&.&.&.&1&.&.&.&.&.&.&.&.\\
\chi_{4,7}''&(A_3{+}A_2{+}A_1,\chi)&.&.&.&1&.&.&.&.&.&.&.&.&.&.\\
\chi_{9,6}''&(D_4{+}A_1,\chi)&.&.&.&.&.&.&1&.&.&.&.&.&.&.\\
\chi_{8,3}'&((A_5)'',\chi)&.&.&1&.&.&.&.&1&.&.&.&.&.&.\\
\chi_{6,6}'&(A_5{+}A_1,\chi)&.&.&.&.&.&.&.&.&1&.&.&.&.&.\\
\chi_{4,8}&(D_5(a_1){+}A_1,\chi)&1&1&.&.&1&.&.&.&.&1&.&.&.&.\\
\chi_{16,5}&(D_6(a_2),\chi)&.&.&1&1&.&.&.&1&.&.&1&.&.&.\\
\chi_{12,4}&(E_7(a_5),1\times \chi)&.&.&.&.&.&.&.&.&1&.&.&1&.&.\\
\chi_{6,6}''&(E_7(a_5),\chi_{21}\times \chi)&.&.&.&.&.&.&.&.&.&.&.&1&.&.\\
\chi_{2,16}''&(D_5{+}A_1,\chi)&.&1&.&.&.&.&.&.&.&1&.&.&.&.\\
\chi_{9,6}'&(D_6(a_1),\chi)&.&.&.&.&.&.&.&.&.&.&.&.&1&.\\
\chi_{8,9}''&(E_7(a_4),1\times \chi)&.&.&.&1&.&.&.&.&.&.&1&.&.&.\\
\chi_{4,7}'&(E_7(a_4),\varepsilon\times \chi)&.&.&.&.&.&.&.&1&.&.&.&.&.&.\\
\chi_{9,10}&(D_6,\chi)&.&.&.&.&.&.&.&.&.&.&.&.&.&1\\
\chi_{1,12}'&(E_7(a_3),1\times \chi)&.&.&.&.&1&.&.&.&.&.&.&.&.&.\\
\chi_{8,9}'&(E_7(a_3),\varepsilon\times \chi)&.&.&.&.&.&.&.&1&.&.&1&.&.&.\\
\chi_{4,13}&(E_7(a_2),\chi)&.&.&.&.&.&.&.&.&.&.&1&.&.&.\\
\chi_{2,16}'&(E_7(a_1),\chi)&.&.&.&.&1&.&.&.&.&1&.&.&.&.\\
\chi_{1,24}&(E_7,\chi)&.&.&.&.&.&.&.&.&.&1&.&.&.&.\\
\hline
\end{array}
\]

\caption{Decomposition matrix for $E_7$ when $\ell=3$ and $\chi \neq 1$}
\label{tab:dec-matric-E7ell3chinot1}
\end{table}

From this table one obtains that $\fN^{(L_\chi, \cO_\chi, \cE_\chi^\bk)}_{G,\bk}$ consists of the following pairs:
\begin{gather*}
((3A_1)'', \chi), \, (4A_1, \chi), \, ((A_3{+}A_1)'', \chi), \, (A_3{+}2A_1, \chi), \, (D_4(a_1){+}A_1, 1 \times \chi), \, (D_4(a_1){+}A_1, \varepsilon \times \chi), \\
(D_4{+}A_1, \chi), \,
((A_5)'', \chi), \, (A_5{+}A_1, \chi), \, (D_5(a_1){+}A_1, \chi), \, (D_6(a_2), \chi), \, (E_7(a_5), 1 \times \chi), \, (D_6(a_1), \chi), \, (D_6, \chi).
\end{gather*}
One can also obtain the corresponding representations of $W(\chi)$. For instance, $((3A_1)'', \chi)$ corresponds to the trivial representation, $(4A_1, \chi)$ corresponds to the unique nontrivial irreducible representation which appears in the modular reduction of $\chi''_{2,4}$, etc.

%%%%%%%%%%%%%%%%%%%%%%%%%%%%%%%%%%%%%%%%%%%%%%%%%%%%%%%%%%%%%%%%%%%%%%%
\section{Cuspidal pairs and cuspidal data}
\label{sec:reduction}
%%%%%%%%%%%%%%%%%%%%%%%%%%%%%%%%%%%%%%%%%%%%%%%%%%%%%%%%%%%%%%%%%%%%%%%

We have seen in~\cite[\S 5.3]{genspring2} that, to classify cuspidal pairs and determine the modular generalized Springer correspondence~\eqref{eqn:mgsc}, it suffices to consider the case where $G$ is simply connected and quasi-simple. Recall that the classical types were considered in~\cite{genspring2}, where we completed the classification of cuspidal pairs in all cases, and determined the modular generalized Springer correspondence when $G$ is of type $A$ (for all $\ell$) and when $G$ is of type $B$/$C$/$D$ (for $\ell=2$). In the remainder of this paper, we focus mainly on the five exceptional types. 

%------------------------------------------------------------------------------
\subsection{Conditions on the characteristic $\ell$}
%------------------------------------------------------------------------------

It will be useful to introduce terminology for those values of the characteristic $\ell$ for which we can expect our problems to be easier to solve. For the moment, we continue to allow $G$ to denote an arbitrary connected reductive group.

\begin{defn}
\label{def:rather-good}
If $\ell$ is a prime number, we say that $\ell$ is \emph{easy} for $G$ if it does not divide $|W|$. We say that $\ell$ is \emph{rather good} for $G$ if it does not divide $|A_G(x)|$ for any $x \in \cN_G$.
\end{defn}

The reason for the term `rather good' is the following relationship with the better known concepts of \emph{good} and \emph{very good} primes (see Table~\ref{tab:rather-good}):

\begin{lem}\label{lem:good}
For a prime number $\ell$, the following conditions are equivalent:
\begin{enumerate}
\item $\ell$ is rather good for $G$;\label{it:newgood-ax}
\item $\ell$ is good for $G$ and does not divide $|Z(G)/Z(G)^\circ|$.\label{it:newgood-center}
\end{enumerate}
In particular, if $G$ is semisimple of adjoint type, then $\ell$ is rather good for $G$ if and only if it is good for $G$; if $G$ is semisimple and simply connected, then $\ell$ is rather good for $G$ if and only if it is very good for $G$.
\end{lem}

\begin{proof}
If $G$ is a simple group, the claim that `rather good' is the same as `good' is an easy case-by-case verification using the description of the groups $A_G(x)$ in~\cite{cm}. This implies the claim in the case where $G$ is semisimple of adjoint type.

For a general connected reductive group $G$, let $\overline{G}:=G/Z(G)$ be the associated adjoint group. From~\eqref{eqn:exact-seq} and the fact that $Z(G)/Z(G)^\circ\to A_G(x)$ is an isomorphism when $x$ is regular nilpotent, it follows that~\eqref{it:newgood-ax} is equivalent to the condition that $\ell$ is rather good for $\overline{G}$ and does not divide $|Z(G)/Z(G)^\circ|$. But by the case already discussed, $\ell$ is rather good for $\overline{G}$ if and only if it is good for $\overline{G}$, which is by definition equivalent to being good for $G$. 
\end{proof}

The conditions `good', `very good', and `easy' all depend only on the root system of $G$, and have the feature that a prime satisfies the condition if and only if it satisfies it for all irreducible components of the root system. So it is enough to know what the conditions mean for quasi-simple groups $G$, and this information is given in Table~\ref{tab:rather-good}. 

\begin{lem} \label{lem:implications}
If $\ell$ is easy for $G$, then $\ell$ is rather good for $G$. If $G$ is quasi-simple and of type other than $A$, then $\ell$ is rather good for $G$ if and only if it is good for $G$.
\end{lem}

\begin{proof}
By Lemma~\ref{lem:good}, we may assume for both claims that $G$ is simply connected and quasi-simple; then `rather good' becomes the same as `very good', and both claims follow from Table~\ref{tab:rather-good}. 
\end{proof}

\begin{table}
\[
\begin{array}{|c|c|c|c|}
\hline
&\text{good}&\text{very good}&\text{easy}\\
\hline\hline
A_{n-1}, n\geq 2 & \text{all }\ell & \ell\nmid n & \ell>n\\
\hline
B_n, n\geq 2 & \ell>2 & \ell>2 & \ell>n\\
\hline
C_n, n\geq 3 & \ell>2 & \ell>2 & \ell>n\\
\hline
D_n, n\geq 4 & \ell>2 & \ell>2 & \ell>n\\
\hline
E_6 & \ell>3 & \ell>3 & \ell>5\\
\hline
E_7 & \ell>3 & \ell>3 & \ell>7\\
\hline
E_8 & \ell>5 & \ell>5 & \ell>7\\
\hline
F_4 & \ell>3 & \ell>3 & \ell>3\\
\hline
G_2 & \ell>3 & \ell>3 & \ell>3\\
\hline
\end{array}
\]
\caption{Conditions on $\ell$ for quasi-simple $G$}\label{tab:rather-good}
\end{table}

\begin{lem} \label{lem:good-inherited}
If $\ell$ is rather good for $G$, then $\ell$ is rather good for every Levi subgroup $L$ of $G$.
\end{lem}

\begin{proof}
This follows from the fact~\cite[Lemma 3.10]{genspring2} that for any $y\in\cN_L$, the order of $A_L(y)$ divides that of $A_G(x)$ where $x$ belongs to the induced nilpotent orbit $\mathrm{Ind}_L^G(L\cdot y)$. Alternatively, the claim follows from Lemma~\ref{lem:good} and the fact that the homomorphism $h_L : Z(G)/Z(G)^\circ\to Z(L)/Z(L)^\circ$ is surjective, see~\S\ref{ss:min-cusp-datum}.
\end{proof}

%------------------------------------------------------------------------------
\subsection{Counting cuspidal pairs}
\label{ss:strategy}
%------------------------------------------------------------------------------

Now assume that $G$ is a simply connected quasi-simple group of exceptional type, and that $\bk$ is big enough for $G$ in the sense of~\eqref{eqn:definitely-big-enough}. By Proposition~\ref{prop:big-enough}, the latter condition is in fact automatic for such $G$ except when $G$ is of type $E_6$, when it requires $\bk$ to contain the third roots of unity of its algebraic closure. 

Having proved Theorem~\ref{thm:mgsc-intro}, we can determine the number of cuspidal pairs $|\fN_{G,\bk}^\cusp|$ using the same recursive counting argument as in Lusztig's setting~\cite{lusztig} and our own previous papers~\cite{genspring1,genspring2}. That is, we use the following formula:
\begin{equation} \label{eqn:counting}
|\fN_{G,\bk}^\cusp| = |\fN_{G,\bk}| - \sum_{\substack{L\in\fL\\L\neq G}} |\fN_{L,\bk}^\cusp|\times|\Irr(\bk[N_G(L)/L])|,
\end{equation} 
which follows immediately from~\eqref{eqn:mgsc-variant}. The resulting values of $|\fN_{G,\bk}^\cusp|$ for simply-connected quasi-simple groups $G$ of exceptional type are displayed in Table~\ref{tab:exc-cuspidal-pairs}, and the calculations are explained below. 

If $G$ is of type $E_6$ or $E_7$, the centre $Z(G)$ is nontrivial, being isomorphic to $\mu_3$ or $\mu_2$ respectively (where $\mu_m$ denotes the cyclic group of $m$th roots of unity in $\C$). For these groups, Table~\ref{tab:exc-cuspidal-pairs} shows more refined information, namely how many cuspidal pairs have each possible central character $\chi$ (see~\cite[\S 5.1]{genspring2}). In type $E_6$ when $\ell\neq 3$, there are three characters $\chi:Z(G)\to\bk^\times$ by our assumption on $\bk$, but the two nontrivial characters are inverse to each other and therefore interchangeable (via Verdier duality). In type $E_7$ when $\ell\neq 2$, there is a unique nontrivial character $\chi:Z(G)\to\bk^\times$. In this setting we use the refined formula
\begin{equation} \label{eqn:counting-chi}
|\fN_{G,\bk}^{\chi,\cusp}| = |\fN_{G,\bk}^\chi| - \sum_{\substack{L\in\fL\\L\supset L_\chi\\L\neq G}} |\fN_{L,\bk}^{\chi,\cusp}|\times|\Irr(\bk[N_G(L)/L])|,
\end{equation} 
where the superscript $\chi$ throughout indicates a restriction to central character $\chi$, and where we assume (as we may) that $\fL$ is chosen in such a way that if $L \in \fL$ contains a $G$-conjugate of $L_\chi$, then it actually contains $L_\chi$. This equality follows by combining~\eqref{eqn:mgsc-variant} with~\cite[Lemma 5.1]{genspring2} and Proposition~\ref{prop:bonnafe}\eqref{it:minimal-Levi}.

We have recorded the information required for the calculations in the tables of Appendix~\ref{sec:tables}. For each $G$, we first give the Bala--Carter labels of the distinguished nilpotent orbits for $G$ and the corresponding groups $A_G(x)$, as found in~\cite[\S 8.4]{cm}; these are for reference in connection to the classification of cuspidal pairs, discussed in \S\ref{ss:determining-cuspidal}. Then we have various tables for the different values of the characteristic $\ell$ and (in types $E_6$ and $E_7$) the central character $\chi$. Each table has a row for every cuspidal datum $(L,\cO_L,\cE_L)$ in the set $\fM_{G,\bk}$ of representatives of the $G$-orbits of cuspidal data (or the more refined set $\fM_{G,\bk}^\chi$ where the central character is required to equal $\chi$); we have ordered the rows by the semisimple rank of $L$. Each row displays:
\begin{itemize}
\item the cuspidal datum $(L,\cO_L,\cE_L)$ itself, where $L$ is denoted by its Bala--Carter name, see~\S\ref{ss:computing} below (or simply $T$ in the case of a maximal torus), $\cO_L$ by its partition label (for $L$ of classical type) or its Bala--Carter label (for $L$ of exceptional type), and $\cE_L$ by either $\ubk$ if it is the trivial local system or by some \textit{ad hoc} notation such as $\cE_\chi$ otherwise;
\item the group $N_G(L)/L$, written in standard notation where $\fS_m$ denotes the symmetric group on $m$ letters and $W(X_n)$ the Weyl group of type $X_n$ (in cases of multiple cuspidal data with the same $L$, we write this group only once, to make the table more readable);
\item the size $|\fN_{G,\bk}^{(L,\cO_L,\cE_L)}|$ of the induction series associated to the cuspidal datum, or equivalently the number $|\Irr(\bk[N_G(L)/L])|$ of irreducible representations over $\bk$ of the group $N_G(L)/L$. 
\end{itemize}
Before applying the recursive count, we know only the \emph{proper} cuspidal data $(L,\cO_L,\cE_L)$, i.e.\ those where $L\neq G$. (Strictly speaking, in some cases where $L$ itself is of exceptional type, we know only the number of cuspidal pairs for $L$, not what those cuspidal pairs are, but this is enough information to apply~\eqref{eqn:counting}.) The remaining cuspidal data are the triples $(G,\cO,\cE)$ where $(\cO,\cE)$ is a cuspidal pair for $G$; the number of these cuspidal pairs, and those cuspidal pairs we are able to determine by the methods explained in \S\ref{ss:determining-cuspidal}, are displayed beneath the proper cuspidal data, so that the total in the final column adds up to the known value of $|\fN_{G,\bk}|$, as dictated by~\eqref{eqn:counting} (or $|\fN_{G,\bk}^\chi|$, as dictated by~\eqref{eqn:counting-chi}).  

%------------------------------------------------------------------------------
\subsection{Computing the tables}
\label{ss:computing}
%------------------------------------------------------------------------------

We now explain how the various entries in the tables were computed (other than the information about cuspidal pairs). 

To find $|\fN_{G,\bk}|$ (respectively, $|\fN_{G,\bk}^\chi|$), we simply add, over all nilpotent orbits, the number of irreducible representations of the corresponding group $A_G(x)$ over $\bk$ (respectively, the number of irreducible representations on which $Z(G)$ acts via the character $\chi$). Recall that, since~\eqref{eqn:definitely-big-enough} holds, all these irreducible representations are absolutely irreducible, so there are as many of them as there are $\ell$-regular conjugacy classes of $A_G(x)$. The classification of nilpotent orbits for $G$, and the description of the groups $A_G(x)$, can be found in the tables of~\cite[\S 8.4]{cm}. (One must correct two misprints in their table for type $E_7$: when the Bala--Carter label of the orbit is $4A_1$ or $(A_5)''$, the entry under $\pi_1(\cO)$ should be $\Z/2\Z$ rather than $1$, i.e.\ $A_G(x)=Z(G)$ in these cases.)

To classify $G$-orbits of proper cuspidal data, we first need the classification of $G$-conjugacy classes of Levi subgroups of $G$ (which is the same as the classification of $W$-conjugacy classes of parabolic subgroups of $W$). This is well known: in fact, it is embedded in the classification of nilpotent orbits of $G$, because the Bala--Carter label of an orbit records the unique $G$-conjugacy class of Levi subalgebras intersecting that orbit in a distinguished orbit. So taking the list of Bala--Carter labels for nilpotent orbits of $G$, and deleting those with parenthetical decorations indicating non-regular distinguished orbits (e.g.\ $E_7(a_1)$), produces a list of names for our set $\fL$ of representatives of the $G$-conjugacy classes of Levi subgroups of $G$. These Bala--Carter names differ slightly from the names of the parabolic subgroups of $W$ as found in the tables of~\cite{howlett}, but the translation is easy.

We now need to consider each $L\in\fL$ (other than $G$ itself) in turn, and classify (or at least count) the cuspidal pairs for $L$. In most cases, the classification of cuspidal pairs for $L$ reduces to the classification of cuspidal pairs for the simple components of $L$, for the following reasons. Recall that $\fN_{L,\bk}^\cusp$ is unchanged if we replace $L$ by the semisimple group $L/Z(L)^\circ$ (see~\cite[\S 5.3]{genspring2}). Much of the time, $L/Z(L)^\circ$ is of adjoint type and hence a direct product of simple groups; for instance, this is automatic if $G$ is of type $E_8$, $F_4$ or $G_2$. Moreover, if $G$ is of type $E_6$ or $E_7$ and we are considering cuspidal pairs of trivial central character, then we can replace $L/Z(L)^\circ$ by its adjoint quotient $L/Z(L)$ anyway (see~\cite[\S 5.3]{genspring2}). The cuspidal pairs for a direct product of groups are obtained by taking products of the cuspidal pairs for the individual groups. 

The simple components of $L$ are almost always of classical type, and the cuspidal pairs for simple groups of classical type were determined in~\cite{genspring1,genspring2}. Recall from~\cite[Theorem~3.1]{genspring1} that a simple group of type $A_{n-1}$ has a cuspidal pair if and only if $n$ is a power of $\ell$, in which case the cuspidal pair is $([n],\ubk)$. Recall from~\cite[\S 7.2, \S\S8.3--8.4]{genspring2} that when $\ell>2$, simple groups of types $B_2$, $B_3$, $C_3$, $D_4$, $D_5$, $D_6$, $D_7$ (that is, all the type-$B$/$C$/$D$ connected subdiagrams of a Dynkin diagram of exceptional type) have no cuspidal pairs. So these type-$B$/$C$/$D$ factors come into play only when $\ell=2$, in which case their cuspidal pairs are exactly the distinguished orbits with trivial local systems (see~\cite[\S 7.1, \S\S8.1--8.2]{genspring2}).

It remains to discuss the case where $G$ is of type $E_6$ or $E_7$ and we are considering a nontrivial central character $\chi$. Then, as noted in~\eqref{eqn:counting-chi}, we need only consider proper Levi subgroups $L$ containing the Levi subgroup $L_\chi$, which is of type $2A_2$ or $(3A_1)''$ respectively. For such $L$, one must examine the root datum of $L$ to determine the isomorphism class of $L/Z(L)^\circ$ and hence its cuspidal pairs of central character $\chi$. When $G$ is of type $E_6$, the relevant $L$ are those of type $2A_2$, $2A_2+A_1$, $A_5$, and the groups $L/Z(L)^\circ$ are, respectively:
\[
\frac{\SL(3)^2}{\mu_3^{\diag}},\ \frac{\SL(3)^2}{\mu_3^{\diag}}\times \PGL(2),\ \frac{\SL(6)}{\mu_2}.
\]
When $G$ is of type $E_7$, the relevant $L$ are those of types
\[
(3A_1)'',\ 4A_1,\ (A_3+A_1)'',\ A_3+2A_1,\ A_2+3A_1,\ (A_5)'',\  D_4+A_1,\  A_3+A_2+A_1,\ A_5+A_1,\ D_5+A_1,\ D_6.
\]
It turns out that only those $L$ of types $(3A_1)''$, $A_2+3A_1$, and $(A_5)''$ support cuspidal data of nontrivial central character over some $\bk$; for these, the groups $L/Z(L)^\circ$ are, respectively:
\[
\frac{\SL(2)^3}{\ker(\mu_2^3\overset{\times}{\to}\mu_2)},\
\PGL(3)\times \frac{\SL(2)^3}{\ker(\mu_2^3\overset{\times}{\to}\mu_2)},\
\frac{\SL(6)}{\mu_3}.
\]

For all Levi subgroups $L$ of $G$, the groups $N_G(L)/L$, or rather the isomorphic groups $N_W(W_L)/W_L$ (see~\eqref{eqn:weyl-isom}), are described in the tables of~\cite{howlett}. We have copied the relevant information into the middle columns of the tables in Appendix~\ref{sec:tables}. (The $(3A_1)''$ class of Levi subgroups of $E_7$ is omitted from the table in~\cite{howlett}; a Levi $L$ in this class is self-opposed, so it is easy to calculate that the relevant group $N_G(L)/L$ is $W(F_4)$, as stated in~\cite[\S 15.2]{lusztig}.) 

To complete the final columns of the tables in Appendix~\ref{sec:tables}, and hence complete the count of cuspidal pairs, we need only compute the number of irreducible representations of each group $N_G(L)/L$ over $\bk$, which equals the number of $\ell$-regular conjugacy classes in $N_G(L)/L$ by Lemma~\ref{lem:splitting}\eqref{it:cuspidal-splitting-field}. 

\begin{rmk} \label{rmk:reflection-group}
The groups $N_G(L)/L$ in Appendix~\ref{sec:tables} show a strong tendency to be reflection groups, a phenomenon which holds generally in the classical types (as seen in~\cite{genspring1,genspring2}) but is more noteworthy in exceptional types (see~\cite{howlett}). In fact, among all Levi subgroups $L$ of a simply connected quasi-simple group $G$ such that $L$ supports a cuspidal pair over some $\bk$, there are only two cases where $N_G(L)/L$ is not a finite crystallographic reflection group. Namely, when $G$ is of type $E_6$ (respectively, $E_8$) and $L$ is of type $A_2$ (respectively, $2A_2$), the group $N_G(L)/L$ is the wreath product $\fS_3^2\rtimes\fS_2$ (respectively, $W(G_2)^2\rtimes\fS_2$), as shown in Table~\ref{tab:E6ell3} (respectively, Table~\ref{tab:E8ell3}). In the notation of~\cite[Corollary 7]{howlett}, the reflection subgroup $W''$ is $\fS_3^2$ (respectively, $W(G_2)^2$), and the subgroup $V$ is $\fS_2$, acting on $W''$ by interchanging the factors. These observations provide another proof of Lemma~\ref{lem:splitting}\eqref{it:cuspidal-splitting-field}, since it is well known that $\mathbb{Q}$ (and hence any field) is a splitting field for a finite crystallographic reflection group $\Gamma$, and the same holds for $\Gamma^2\rtimes\fS_2$. 
\end{rmk} 

%------------------------------------------------------------------------------
\subsection{Determining cuspidal pairs}
\label{ss:determining-cuspidal}
%------------------------------------------------------------------------------

Knowing the number of cuspidal pairs for $G$ over $\bk$, we next try to determine what they are, using the general results we have proved in this series of papers. Recall that cuspidal pairs must be supported on distinguished orbits~\cite[Proposition 2.6]{genspring2}, which is why we have listed the distinguished nilpotent orbits and the corresponding groups $A_G(x)$ for each $G$ in Appendix~\ref{sec:tables}. In naming cuspidal pairs, we use a slightly nonuniform notation for the local systems, in order to avoid ambiguity: when the local system is trivial we continue to denote it as $\ubk$, and when it is nontrivial we denote it in the same way as the corresponding irreducible representation of the relevant group $A_G(x)$ (where, for instance, $\varepsilon$ always denotes the sign representation of a symmetric group $\fS_m$).

Recall that Lusztig's classification of cuspidal pairs is particularly nice in the exceptional types:

\begin{prop}[\cite{lusztig}]
\label{prop:cuspidal-exc-char0}
Let $G$ be a simply connected quasi-simple group of exceptional type. If $G$ is not of type $E_6$, it has a unique cuspidal pair $(\cO,\cE)$ over $\Qlb$. If $G$ is of type $E_6$, it has two cuspidal pairs $(\cO,\cE)$, one of each nontrivial central character. The orbit $\cO$ is the minimal distinguished nilpotent orbit for $G$ in each case, and all the local systems $\cE$ have rank one.
\end{prop}
\noindent
(Here, the `minimal' distinguished nilpotent orbit is the unique distinguished nilpotent orbit that sits below the other distinguished orbits in the closure order; such an orbit exists in every exceptional type, see~\cite{carter}.)

It follows from Proposition~\ref{prop:cuspidal-exc-char0} that every cuspidal pair $(\cO,\cE)$ of $G$ over $\Qlb$ has a modular reduction $(\cO,\cE^\bk)$ in the sense defined after~\eqref{eqn:assumption-dec-numbers}, which is a cuspidal pair over $\bk$ by~\cite[Proposition 2.22]{genspring1}. It is not always the unique cuspidal pair with its central character as in~\eqref{eqn:assumption-dec-numbers}, but the fact that there are no distinguished nilpotent orbits in $\overline{\cO}\smallsetminus\cO$ ensures that $\IC(\cO,\cE^\bk)$ is the modular reduction of $\IC(\cO,\cE)$ in the sense of~\cite[\S 2.7]{genspring1}.

\begin{rmk}
When $G$ is of type $E_6$ and $\ell=3$, the two cuspidal pairs over $\Qlb$ have the same modular reduction $(E_6(a_3),\varepsilon)$, which has trivial central character. Similarly, when $G$ is of type $E_7$ and $\ell=2$, the modular reduction of the cuspidal pair over $\Qlb$ is $(E_7(a_5),\ubk)$, which has trivial central character. Thus, in these cases the modular reduction descends to a cuspidal pair for the adjoint group, although the cuspidal pair over $\Qlb$ does not. A similar behavior was crucial for our treatment of the group $\GL(n)$ in~\cite{genspring1}.
\end{rmk}

When $\ell$ is a good prime, this modular reduction procedure accounts for all the cuspidal pairs over $\bk$:
\begin{prop}
\label{prop:0-cusp-not-A}
Let $G$ be a simply connected quasi-simple group of exceptional type, and suppose that $\ell$ is good for $G$ and $\bk$ is big enough for $G$. Then the only cuspidal pair(s) for $G$ over $\bk$ is/are the modular reduction(s) of the cuspidal pair(s) for $G$ over $\Qlb$. In particular, there is at most one cuspidal pair of each central character.
\end{prop}

\begin{proof}
This follows immediately from the count of cuspidal pairs in the cases of a good prime $\ell$, given in the relevant tables of Appendix~\ref{sec:tables}. But we can give the following more uniform explanation.

In the case when $\ell$ is easy for $G$ (meaning that $\ell\nmid|W|$), we can see \textit{a priori} that there cannot be any difference between the cuspidal data for $G$ over $\bk$ and those for $G$ over $\Qlb$, determined by Lusztig in~\cite{lusztig}; in particular, there is no difference between the cuspidal pairs. Since $\ell$ is easy for $G$, it is also rather good for $G$ (see Lemma~\ref{lem:implications}) and thus rather good for every Levi subgroup $L$ of $G$ (see Lemma~\ref{lem:good-inherited}), so we have an identification of $\fN_{L,\bk}$ with the analogous set $\fN_{L,\Qlb}$ for every Levi subgroup $L$ of $G$. Having made these identifications, we have inclusions $\fN_{L,\Qlb}^\cusp\subset\fN_{L,\bk}^\cusp$ by~\cite[Proposition 2.22]{genspring1}. Because $\ell$ does not divide the order of any of the groups $N_G(L)/L$, and because Lemma~\ref{lem:splitting}\eqref{it:cuspidal-splitting-field} holds, we also have an identification of $\Irr(\bk[N_G(L)/L])$ with the analogous set $\Irr(\Qlb[N_G(L)/L])$ for every relevant Levi subgroup $L$ of $G$. So all the inclusions $\fN_{L,\Qlb}^\cusp\subset\fN_{L,\bk}^\cusp$ must be equalities in order for~\eqref{eqn:mgsc-variant} to hold both for $\Qlb$ and for $\bk$.

Now suppose that $\ell$ is good but not easy for $G$, i.e.\ $\ell=5$ in type $E_6$, $\ell\in\{5,7\}$ in type $E_7$, or $\ell=7$ in type $E_8$. In each of these cases the prime $\ell$ divides $|W|$ exactly once, and moreover (as remarked in \S\ref{ss:linear-prime}) the $\ell$-Sylow class is of type $A_{\ell-1}$. Among all the cuspidal data $(L,\cO_L,\cE_L)$ over $\Qlb$, the only one for which $\ell$ divides $|N_G(L)/L|$ is the principal cuspidal datum $(T,\{0\},\bk)$.

It is still the case that $\ell$ is rather good for $G$ (see Lemma~\ref{lem:implications}), so we still have an identification of $\fN_{L,\bk}$ with the analogous set $\fN_{L,\Qlb}$ for every Levi subgroup $L$ of $G$. The only failure in the previous argument is that the number of pairs $|\Irr(\bk[W])|$ in the principal induction series is less than the corresponding number of pairs $|\Irr(\Qlb[W])|$ in Lusztig's setting, the defect being the number of $\ell$-singular conjugacy classes of $W$. Lemma~\ref{lem:linear-prime} shows that this defect is exactly compensated for by the new induction series associated to the $\ell$-Sylow class of Levi subgroups (see Theorem~\ref{thm:regular-series}). So the count of cuspidal pairs comes out the same as in Lusztig's $\Qlb$ setting.
\end{proof}

Proposition~\ref{prop:0-cusp-not-A} completes the proof of Theorem~\ref{thm:cuspidal-facts}.

For the remainder of the subsection, we assume that $\ell$ is a bad prime for $G$. Recall that Theorem~\ref{thm:regular-series} gave a uniform criterion for cuspidality of the pair $(\cO_\reg,\ubk)$. Consulting Table~\ref{tab:vertex}, we see that in the exceptional types, $(\cO_\reg,\ubk)$ is cuspidal for all bad primes $\ell$ except when $\ell=2$ in type $E_6$ and when $\ell=3$ in type $E_7$. (In these cases, the count reveals that there are in fact no cuspidal pairs of trivial central character).

When $G$ is of type $G_2$, this information completes the classification of cuspidal pairs. When $\ell=2$, the two cuspidal pairs must be $(G_2,\ubk)$ and $(G_2(a_1),\ubk)$, the latter being the modular reduction of Lusztig's cuspidal pair $(G_2(a_1),\varepsilon)$. Similarly, when $\ell=3$, the two cuspidal pairs must be $(G_2,\ubk)$ and $(G_2(a_1),\varepsilon)$.

When $G$ is of type $E_6$ and $\ell=2$, the count reveals that there are two cuspidal pairs of each nontrivial central character $\chi$, and we know that $(E_6(a_3),1\times\chi)$ is one of these. The other cuspidal pair must be either $(E_6,\chi)$ or $(E_6(a_1),\chi)$, but we lack a nontrivial central character analogue of Theorem~\ref{thm:regular-series} to decide which.

When $G$ is of type $E_6$ and $\ell=3$, the count reveals that there are three cuspidal pairs, and the above arguments show that $(E_6,\ubk)$ and $(E_6(a_3),\varepsilon)$ are two of these. We will see in \S\ref{ss:E6-mgsc} that the pair $(E_6(a_3),\ubk)$ belongs to the induction series attached to the cuspidal datum $(2A_2,[3]^2,\ubk)$, so the third cuspidal pair must be the only other pair supported on a distinguished orbit, namely $(E_6(a_1),\ubk)$.

It seems reasonable to guess that the following weaker form of Theorem~\ref{thm:cuspidal-facts} holds in arbitrary characteristic; it is true for types $A$--$D$ by~\cite{genspring2}, and for types $G_2$ and $E_6$ as we have just seen.

\begin{conj} \label{conj:distinct}
Let $G$ be any connected reductive group and suppose that $\bk$ is big enough for $G$.
For a fixed nilpotent orbit $\cO$, the cuspidal pairs $(\cO,\cE)\in\fN_{G,\bk}^\cusp$ involving that orbit have distinct central characters.  In other words, each nilpotent orbit supports at most one cuspidal pair with a given central character.
\end{conj}

Simply assuming Conjecture~\ref{conj:distinct} is enough to complete the classification of cuspidal pairs in the following additional cases:
\begin{itemize}
\item In type $F_4$ when $\ell=2$, the 4 cuspidal pairs would have to be the four distinguished orbits with the trivial local systems, since $(F_4(a_3),\ubk)$ is the modular reduction of Lusztig's cuspidal pair, and Conjecture~\ref{conj:distinct} rules out the other pair supported on $F_4(a_3)$.
\item In type $E_7$ when $\ell=2$, the 6 cuspidal pairs would have to be the six distinguished orbits with the trivial local systems, since $(E_7(a_5),\ubk)$ is the modular reduction of Lusztig's cuspidal pair, and Conjecture~\ref{conj:distinct} rules out the other pair supported on $E_7(a_5)$.
\end{itemize}

\begin{rmk}
We see a curious partial pattern here for the $\ell=2$ case. In types $B$/$C$/$D$, the cuspidal pairs when $\ell=2$ are exactly the distinguished orbits with trivial local systems~\cite{genspring2}; this definitely also holds for $G_2$, and subject to Conjecture~\ref{conj:distinct} it holds for $F_4$ and $E_7$. However, it does not hold in type $A$, and the count shows that it cannot hold for $E_6$ or $E_8$ either.
\end{rmk}

%%%%%%%%%%%%%%%%%%%%%%%%%%%%%%%%%%%%%%%%%%%%%%%%%%%%%%%%%%%%%%%%%%%%%%%
\section{Determining induction series}
\label{sec:series}
%%%%%%%%%%%%%%%%%%%%%%%%%%%%%%%%%%%%%%%%%%%%%%%%%%%%%%%%%%%%%%%%%%%%%%%

Continue to let $G$ denote a simply connected quasi-simple group of exceptional type. The classification of cuspidal pairs considered in \S\ref{ss:determining-cuspidal} is part of the bigger problem of determining the modular generalized Springer correspondence~\eqref{eqn:mgsc}. Each row in one of the tables of Appendix~\ref{sec:tables} corresponds to an induction series $\fN_{G,\bk}^{(L,\cO_L,\cE_L)}$, whose size is given in the final column: to determine the correspondence in full, one needs to know which pairs $(\cO,\cE)\in\fN_{G,\bk}$ belong to each induction series, and an explicit description of the bijection~\eqref{eqn:bijection} between those pairs and the irreducible $\bk$-representations of $N_G(L)/L$. 

The third author solved this problem for the principal induction series $\fN_{G,\bk}^{(T,\{0\},\ubk)}$ in~\cite{juteau}, explicitly describing the modular Springer correspondence in each exceptional type in~\cite[Section 9]{juteau}. In Section~\ref{sec:basic-sets} we have explained how to solve the problem for the induction series labelled by the minimal cuspidal datum for a nontrivial central character. Thus, the answers are known for the top row of each table in Appendix~\ref{sec:tables} (in particular, see \S\ref{sss:E6ell2} for the top row of Table~\ref{tab:E6ell2chinot1} and \S\ref{sss:E7ell3} for the top row of Table~\ref{tab:E7ell3chinot1}). For the bottom row of each table, the problem is the determination of cuspidal pairs, already discussed in \S\ref{ss:determining-cuspidal}. As the reader can see, this still leaves many rows, especially for small values of $\ell$.

This problem is considerably more challenging than the characteristic-zero version solved (almost completely) by Spaltenstein~\cite{spaltenstein}. This is largely because there is no modular analogue of Lusztig's result~\cite[1.5(V)]{spaltenstein} locating the pairs in an induction series $\fN_{G,\bk}^{(L,\cO_L,\cE_L)}$ corresponding to the trivial and sign representations of $N_G(L)/L$, which Spaltenstein called the `starting point to apply [the restriction theorem] in a nontrivial way'. 
Therefore, in this section we restrict ourselves to some special cases where we can make progress.

%------------------------------------------------------------------------------
\subsection{The easy case}
\label{ss:easy-mgsc}
%------------------------------------------------------------------------------

We first show that, if $\ell$ is easy for $G$ (i.e.\ $\ell$ does not divide $|W|$), the modular generalized Springer correspondence is essentially the same as Lusztig's generalized Springer correspondence, determined in~\cite{lusztig,lus-spalt} for classical groups and~\cite{spaltenstein} for exceptional groups. (As in Section~\ref{sec:basic-sets}, one must twist Lusztig's correspondence by the sign character.)
In fact, we prove a slightly more general result, taking into account nontrivial central characters. 

\begin{prop} \label{prop:easy-mgsc}
Let $\chi:Z(G)/Z(G)^\circ\to\bk^\times$ be any central character, and assume that $\ell$ is rather good for $G$ and does not divide $|W(\chi)|$. Then the part of the modular generalized Springer correspondence~\eqref{eqn:mgsc} corresponding to $\fN^\chi_{G,\bk}$ is the same as in Lusztig's setting.
\end{prop}

\begin{proof}
By Lemma~\ref{lem:splitting}\eqref{it:indep-of-k}, we can assume that $\bk$ is part of a triple $(\K,\O,\bk)$ as defined in \S\ref{ss:decomposition-numbers}. Since $\ell \nmid |Z(G)/Z(G)^\circ|$ (see Lemma~\ref{lem:good}), $\chi$ is the modular reduction of a unique $\K$-character of $Z(G)/Z(G)^\circ$, which we also denote by $\chi$ for simplicity.
For any cuspidal datum $(L,\cO_L,\cE_L^\K)$ over $\K$ such that $\cE_L^\K$ has central character $\chi$, the group $N_G(L)/L$ is a subquotient of $W(\chi)$, as explained in \S\ref{ss:min-cusp-datum}. So, by our assumption on $\ell$, there exists a bijection $\imath_{L} : \Irr(\K [N_G(L)/L]) \simto \Irr(\bk [N_G(L)/L])$ that makes the decomposition matrix $(d^{N_G(L)/L}_{E,\imath_L(E')})_{E,E'}$ the identity matrix. Similarly, since $\ell$ is rather good for $G$, for any nilpotent orbit $\cO \subset \cN_G$, if we choose $x_\cO \in \cO$, then there exists a bijection $\imath_{\cO} : \Irr(\K [A_G(x_\cO)]) \simto \Irr(\bk [A_G(x_\cO)])$ that makes the decomposition matrix $(d^{A_G(x_\cO)}_{E,\imath_\cO(E')})_{E,E'}$ the identity matrix. Using these bijections we deduce a canonical bijection $\theta : \fN_{G,\K} \simto \fN_{G,\bk}$, which restricts to a bijection $\theta^\chi : \fN_{G,\K}^\chi \simto \fN_{G,\bk}^\chi$.

By the same arguments as in the proof of Proposition~\ref{prop:0-cusp-not-A}, 
the cuspidal data over $\bk$ with central character $\chi$ are exactly the modular reductions of the corresponding cuspidal data over $\K$. If $(L,\cO_L,\cE_L^\K)$ is such a cuspidal datum over $\K$, and if $\cE_L^\bk$ denotes the modular reduction of $\cE_L^\K$, then~\eqref{eqn:assumption-dec-numbers} holds by Remark~\ref{rmk:modular-reduction-of-cuspidal} and Lemma~\ref{lem:good}, so that,
by Proposition~\ref{prop:equality-decomp-number}, for any $E\in\Irr(\K [N_G(L)/L])$ we have
\begin{equation}
\label{eqn:equality-dec-easy}
d^{\cN_G}_{\Psi^{(L,\cO_L,\cE_L^\K)}_\K(E), \Psi^{(L,\cO_L,\cE_L^\bk)}_\bk(\imath_L(E))}=1.
\end{equation}
Using an induction argument on nilpotent orbits (ordered by inclusion of closures), equality~\eqref{eqn:equality-dec-easy} and the fact that the decomposition matrix for $A_G(x_\cO)$ is the identity force the following equality for any cuspidal datum $(L,\cO_L,\cE_L^\K)$ such that $\cE_L^\K$ has central character $\chi$ and any $E \in \Irr(\K [N_G(L)/L])$:
\begin{equation}
\Psi^{(L,\cO_L,\cE_L^\bk)}_\bk(\imath_L(E)) = \theta^\chi \bigl( \Psi^{(L,\cO_L,\cE_L^\K)}_\K(E) \bigr).
\end{equation}
This proves the claim.
\end{proof}

\begin{rmk}
Under the assumptions of Proposition~\ref{prop:easy-mgsc}, we have moreover that for any cuspidal datum $(L,\cO_L,\cE_L)\in\fM_{G,\bk}^\chi$, the induced perverse sheaf $\Ind_{L\subset P}^G(\IC(\cO_L,\cE_L))$ is semisimple, so that the simple perverse sheaves in the corresponding induction series are its direct summands, as in Lusztig's setting. Indeed, $\Ind_{L\subset P}^G(\IC(\cO_L,\cE_L))$ is semisimple whenever $\ell$ does not divide $|N_G(L)/L|$, as follows immediately from~\cite[Corollary 2.18]{genspring1} and~\cite[Theorem 3.1(3)--(4)]{genspring2}. 
\end{rmk}

Proposition~\ref{prop:easy-mgsc} means that our problem is already solved for the final table listed for each $G$ and $\chi$ in Appendix~\ref{sec:tables}; that is, for Tables~\ref{tab:G2ell>3},~\ref{tab:F4ell>3},~\ref{tab:E6ell>3chinot1},~\ref{tab:E6ell>5chi1},~\ref{tab:E7ell>3chinot1},~\ref{tab:E7ell>7chi1}, and~\ref{tab:E8ell>7}.

%------------------------------------------------------------------------------
\subsection{The good-but-not-easy case}
\label{ss:good-mgsc}
%------------------------------------------------------------------------------

Now suppose that $\ell$ is a good prime for $G$ but not easy, i.e.\ $\ell=5$ in type $E_6$, $\ell\in\{5,7\}$ in type $E_7$, or $\ell=7$ in type $E_8$. If $G$ is of type $E_6$ or $E_7$, then the case of nontrivial central character $\chi$ is covered by Proposition~\ref{prop:easy-mgsc}, since $\ell$ does not divide $|W(\chi)|$. So we can restrict attention to the case of the trivial central character. 

As 
shown in Tables~\ref{tab:E6ell5chi1},~\ref{tab:E7ell5chi1},~\ref{tab:E7ell7chi1} and~\ref{tab:E8ell7}, the only non-principal cuspidal datum in these cases (apart from the cuspidal pairs, determined in Proposition~\ref{prop:0-cusp-not-A}) is the cuspidal datum $(A_{\ell-1},[\ell],\ubk)$ associated to the $\ell$-Sylow class of Levi subgroups (see Theorem~\ref{thm:regular-series}). Since we know the elements of the principal induction series from~\cite{juteau}, we can immediately determine the elements of the induction series $\fN_{G,\bk}^{(A_{\ell-1},[\ell],\ubk)}$ in each case. In accordance with Theorem~\ref{thm:regular-series}, the pair $(\cO_\reg,\ubk)$ is always an element of this series; the complete lists are as follows.
\begin{itemize}
\item $G=E_6$, $\chi=1$, $\ell=5$: $\fN_{G,\bk}^{(A_4,[5],\ubk)}$ consists of $(E_6(a_1),\ubk)$ and $(E_6,\ubk)$.
\item $G=E_7$, $\chi=1$, $\ell=5$: $\fN_{G,\bk}^{(A_4,[5],\ubk)}$ consists of $(E_6,\ubk)$, $(E_7(a_3),\ubk)$, $(E_7(a_3),\varepsilon)$, $(E_7(a_2),\ubk)$, $(E_7(a_1),\ubk)$, and $(E_7,\ubk)$.
\item $G=E_7$, $\chi=1$, $\ell=7$: $\fN_{G,\bk}^{(A_6,[7],\ubk)}$ consists of $(E_7(a_4),\varepsilon)$ and $(E_7,\ubk)$.
\item $G=E_8$, $\ell=7$: $\fN_{G,\bk}^{(A_6,[7],\ubk)}$ consists of $(D_7,\ubk)$, $(E_8(b_4),\varepsilon)$, $(E_8(a_1),\ubk)$, and $(E_8,\ubk)$.
\end{itemize}
We were not able to determine the bijection~\eqref{eqn:bijection} in these cases.

%------------------------------------------------------------------------------
\subsection{Type $G_2$}
\label{ss:G2-mgsc}
%------------------------------------------------------------------------------

In this subsection we determine the modular generalized Springer correspondence when $G$ is of type $G_2$. By Proposition~\ref{prop:easy-mgsc}, we need only consider the cases $\ell=2$ and $\ell=3$. 

In the $\ell=3$ case, as Table~\ref{tab:G2ell3} shows, we have only the principal induction series and the two cuspidal pairs. So the problem is just to determine the modular Springer correspondence, which was done in~\cite[\S 9.1.2]{juteau}.

In the $\ell=2$ case, the modular Springer correspondence is described in~\cite[\S 9.1.1]{juteau}: the induction series associated to $(T,\{0\},\ubk)$ consists of the pairs $(\{0\},\ubk)$ and $(\tilde A_1,\ubk)$, corresponding to the trivial and nontrivial irreducible $2$-modular representations of $W(G_2)$ respectively. The only information left to determine is which of the two other non-cuspidal pairs, $(A_1,\ubk)$ and $(G_2(a_1),\varphi_{21})$, belongs to which of the two remaining (singleton) induction series. (Here, $\varphi_{21}$ denotes the nontrivial irreducible $2$-modular representation of $\fS_3$, or rather the corresponding $\bk$-local system on the subregular orbit $G_2(a_1)$.) By Proposition~\ref{prop:levi-rule} we cannot have $(A_1,\ubk)\in\fN_{G,\bk}^{(\tilde A_1,[2],\ubk)}$, so it must be that $\fN_{G,\bk}^{(A_1,[2],\ubk)}=\{(A_1,\ubk)\}$ and $\fN_{G,\bk}^{(\tilde A_1,[2],\ubk)}=\{(G_2(a_1),\varphi_{21})\}$.

\subsection{Type $E_6$ with $\ell=3$}
\label{ss:E6-mgsc}

In this subsection we assume that $G$ is a simply connected quasi-simple group of type $E_6$ and that $\ell=3$ (see Table~\ref{tab:E6ell3}). We will determine the series $\fN_{G,\bk}^{(2A_2,[3]^2,\ubk)}$ and the bijection~\eqref{eqn:bijection} for this series.

By Lemma~\ref{lem:splitting}\eqref{it:indep-of-k},
we can assume that $\bk$ is part of a triple $(\K,\O,\bk)$ as in \S\ref{ss:decomposition-numbers}, and we will use the notation introduced in this subsection. We take $(L,\cO_L,\cE_L^\K)$ to be the cuspidal datum $(2A_2,[3]^2,\cE_\chi^\K)$, where the central character $\chi:Z(G)/Z(G)^\circ\to\K^\times$ is one of the two nontrivial characters. The induced character $Z(G)/Z(G)^\circ\to\bk^\times$ is trivial, and $\cE_L^\bk$ is the constant sheaf $\ubk$. (Note that we have made a change to the notation of \S\ref{ss:decomposition-numbers}, in that $\tilde\chi$ has become $\chi$ and $\chi$ has become $1$.) The assumption~\eqref{eqn:assumption-dec-numbers} holds in this case, see Remark~\ref{rmk:modular-reduction-of-cuspidal}.

In Table~\ref{tab:dec-matric-E6ell3} we display the decomposition matrix for the group $N_G(L)/L\cong W(G_2)$, in the same style as Table~\ref{tab:dec-matric-E6ell2chinot1}. We let $\phi_1,\phi_2,\phi_3,\phi_4$ denote the $3$-modular irreducible representations (which are all one-dimensional); their order is clearly determined by the decomposition matrix. We will determine the corresponding pairs $(\cO_i,\cE_i):=\Psi_\bk^{(2A_2,[3]^2,\ubk)}(\phi_i)\in\fN_{G,\bk}$ using
Proposition~\ref{prop:levi-rule} and Proposition~\ref{prop:equality-decomp-number}.
See~\cite[\S13.4]{carter} for the closure order on nilpotent orbits, and~\cite[\S8.4]{cm} for the groups $A_G(x)$.

\begin{table}
\[
\begin{array}{|c|c|cccc|}
\hline
E & (\cO,\cE^\K)&\phi_1&\phi_2&\phi_3&\phi_4 \\
\hline\hline
\chi_{1,0}   & (2A_2, \chi)       &1&.&.&.\\
\chi_{1,3}'' & (2A_2{+}A_1, \chi)  &.&1&.&.\\
\chi_{2,1}  & (A_5, \chi)         &.&1&1&.\\
\chi_{2,2}  & (E_6(a_3), 1\times\chi) &1&.&.&1\\
\chi_{1,3}' & (E_6(a_1), \chi)    &.&.&1&.\\
\chi_{1,6}  & (E_6, \chi)         &.&.&.&1\\
\hline
\end{array}
\]
\caption{Decomposition matrix for the $2A_2$ series in $E_6$ when $\ell=3$}
\label{tab:dec-matric-E6ell3}
\end{table}

By Proposition~\ref{prop:levi-rule}, each orbit $\cO_i$ must have a Bala--Carter Levi subgroup that contains a Levi subgroup of type $2A_2$. Thus, the only possibilities for the pairs $(\cO_i,\cE_i)$, after ruling out the known cuspidal pairs $(E_6,\ubk)$ and $(E_6(a_3),\varepsilon)$, are:
\[
(2A_2,\ubk),\ (2A_2+A_1,\ubk),\ (A_5,\ubk),\ (E_6(a_3),\ubk),\ (E_6(a_1),\ubk).
\]
From Proposition~\ref{prop:equality-decomp-number} we know that $d^{\cN_G}_{\IC(2A_2,\chi),\IC(\cO_1,\cE_1)}=1$, so $\cO_1\subset\overline{2A_2}$. Similarly $\cO_2\subset\overline{2A_2+A_1}$, $\cO_3\subset\overline{A_5}$ and $\cO_4\subset\overline{E_6(a_3)}$. Considering these facts in succession, we deduce that
\[
(\cO_1,\cE_1)=(2A_2,\ubk),\ (\cO_2,\cE_2)=(2A_2+A_1,\ubk),\ (\cO_3,\cE_3)=(A_5,\ubk),\ (\cO_4,\cE_4)=(E_6(a_3),\ubk).
\]

As mentioned in \S\ref{ss:determining-cuspidal}, the occurrence of $(E_6(a_3),\ubk)$ here means that the third cuspidal pair must be $(E_6(a_1),\ubk)$. Since the principal series in this case was determined in~\cite[\S9.3.2]{juteau}, the elements of the remaining induction series $\fN_{G,\bk}^{(A_2,[3],\ubk)}$ must be
\[
(D_4(a_1),\varepsilon),\ (D_4,\ubk),\ (A_4+A_1,\ubk),\ (D_5(a_1),\ubk),\ (D_5,\ubk).
\]
We were not able to determine the bijection~\eqref{eqn:bijection} for the latter series.

\subsection{Type $E_7$ with $\ell=2$}
\label{ss:E7-mgsc}

Similarly to the previous subsection, for $G$ simply connected quasi-simple of type $E_7$ and $\ell = 2$, 
one can determine the series $\fN_{G,\bk}^{((3A_1)'',[2]^3,\ubk)}$ using Proposition~\ref{prop:equality-decomp-number}
and the known characteristic zero series $\fN_{G,\K}^{((3A_1)'',[2]^3,\cE^\K_\chi)}$, where $\chi$ is the nontrivial central character and $\cE^\K_\chi$ is the corresponding cuspidal local system. The corresponding modular pairs are
\[
((3A_1)'', \bk),\ (4A_1, \bk),\ ((A_3{+}A_1)'',\bk),\ (A_3{+}2A_1,\bk),
\]
as can be seen from Table \ref{tab:dec-matric-E7ell2} showing the first lines of the $2$-modular decomposition matrix of $W(F_4)$, using a similar reasoning.

\begin{table}[hb]
\[
\begin{array}{|c|c|cccc|}
\hline
E & (\cO,\cE^\K)&\phi_1&\phi_2&\phi_3&\phi_4 \\
\hline\hline
\chi_{1,0}&((3A_1)'',\chi)&1&.&.&.\\
\chi_{2,4}''&(4A_1,\chi)&.&1&.&.\\
\chi_{1,12}''&(A_2{+}3A_1,\chi)&1&.&.&.\\
\chi_{4,1}&((A_3{+}A_1)'',\chi)&.&1&1&.\\
\chi_{8,3}''&(A_3{+}2A_1,\chi)&2&1&.&1\\
%\chi_{2,4}'&(D_4(a_1){+}A_1,2,\chi)&.&.&1&.\\
%\chi_{9,2}&(D_4(a_1){+}A_1,\varepsilon\times\chi)&1&1&1&1\\
%\chi_{4,7}''&(A_3{+}A_2{+}A_1,\chi)&.&1&1&.\\
%\chi_{9,6}''&(D_4{+}A_1,\chi)&1&1&1&1\\
%\chi_{8,3}'&(A_5'',\chi)&2&.&1&1\\
%\chi_{6,6}''&(A_5{+}A_1,\chi)&2&.&.&1\\
%\chi_{4,8}&(D_5(a_1){+}A_1,\chi)&.&.&.&1\\
%\chi_{16,5}&(D_6(a_2),\chi)&.&2&2&2\\
%\chi_{12,4}&(E_7(a_5),3,\chi)&.&2&2&1\\
%\chi_{6,6}'&(E_7(a_5),21,\chi)&2&.&.&1\\
%\chi_{2,16}''&(D_5{+}A_1,\chi)&.&.&1&.\\
%\chi_{9,6}'&(D_6(a_1),\chi)&1&1&1&1\\
%\chi_{8,9}''&(E_7(a_4),1\times\chi)&2&.&1&1\\
%\chi_{4,7}'&(E_7(a_4),\varepsilon\times\chi)&.&1&1&.\\
%\chi_{9,10}&(D_6,\chi)&1&1&1&1\\
%\chi_{1,12}'&(E_7(a_3),1\times\chi)&1&.&.&.\\
%\chi_{8,9}'&(E_7(a_3),\varepsilon\times\chi)&2&1&.&1\\
%\chi_{4,13}&(E_7(a_2),\chi)&.&1&1&.\\
%\chi_{2,16}'&(E_7(a_1),\chi)&.&1&.&.\\
%\chi_{1,24}&(E_7,\chi)&1&.&.&.\\
\hline
\end{array}
\]
\caption{Part of the decomposition matrix for the $(3A_1)''$ series in $E_7$ when $\ell=2$}
\label{tab:dec-matric-E7ell2}
\end{table}

\vfill
\pagebreak
\clearpage
\appendix

%%%%%%%%%%%%%%%%%%%%%%%%%%%%%%%%%%%%%%%%%%%%%%%%%%%%%%%%%%%%%%%%%%%%%%%
\section{Tables of cuspidal data and sizes of induction series}
\label{sec:tables}
%%%%%%%%%%%%%%%%%%%%%%%%%%%%%%%%%%%%%%%%%%%%%%%%%%%%%%%%%%%%%%%%%%%%%%%

This appendix collects tables of cuspidal data and cuspidal pairs for all simply-connected quasi-simple exceptional groups $G$ (with some indeterminacies), displaying the size of the corresponding induction series in each case. See \S\ref{ss:strategy} for the explanation of how to read these tables and \S\S\ref{ss:computing}--\ref{ss:determining-cuspidal} for the explanation of how they were calculated.

%-------------------------------------------------------------
\subsection{Type $G_2$}
\label{ss:G2}
%%-------------------------------------------------------------

The distinguished nilpotent orbits and the corresponding groups $A_G(x)$ are:
\[
\begin{array}{|c||c|c|c|c|c|c|c|c|c|c|c|}
\hline
\cO & 
G_2  & 
G_2(a_1) 
\\
\hline
A_G(x) &
1&
\fS_3 
\\
\hline
\end{array}
\]
\noindent
The tables for $G_2$ are as follows.

\begin{table}[ht]
\[
\begin{array}{|c|c|c|}
\hline
(L, \cO_L, \cE_L)\in\fM_{G,\bk} & N_G(L)/L & |\fN_{G,\bk}^{(L,\cO_L,\cE_L)}| \\
\hline\hline
(T, \{0\}, \ubk) & W(G_2) & 2\\
\hline
(A_1, [2], \ubk) & \fS_2 & 1\\
\hline
(\tilde A_1, [2], \ubk) & \fS_2 & 1\\
\hline\hline
\text{2 cuspidal pairs: } & 1& 2 \times 1\\
(G_2,\ubk), (G_2(a_1),\ubk) &&\\
\hline\hline
&& |\fN_{G,\bk}|=6 \\
\hline
\end{array}
\]
\caption{Induction series for $G_2$ when $\ell=2$}\label{tab:G2ell2}
\end{table}

\begin{table}[ht]
\[
\begin{array}{|c|c|c|}
\hline
(L, \cO_L, \cE_L)\in\fM_{G,\bk} & N_G(L)/L & |\fN_{G,\bk}^{(L,\cO_L,\cE_L)}| \\
\hline\hline
(T, \{0\}, \ubk) & W(G_2) & 4\\
\hline\hline
\text{2 cuspidal pairs: } & 1& 2 \times 1\\
(G_2,\ubk), (G_2(a_1),\varepsilon)  &&\\
\hline\hline
&& |\fN_{G,\bk}|=6 \\
\hline
\end{array}
\]
\caption{Induction series for $G_2$ when $\ell=3$}\label{tab:G2ell3}
\end{table}

\begin{table}[ht]
\[
\begin{array}{|c|c|c|}
\hline
(L, \cO_L, \cE_L)\in\fM_{G,\bk} & N_G(L)/L & |\fN_{G,\bk}^{(L,\cO_L,\cE_L)}| \\
\hline\hline
(T, \{0\}, \ubk) & W(G_2) & 6\\
\hline \hline
\text{1 cuspidal pair: } (G_2(a_1),\varepsilon) & 1& 1\\
\hline\hline
&& |\fN_{G,\bk}|=7 \\
\hline
\end{array}
\]
\caption{Induction series for $G_2$ when $\ell>3$}\label{tab:G2ell>3}
\end{table}

\clearpage

%-------------------------------------------------------------
\subsection{Type $F_4$}
\label{ss:F4}
%-------------------------------------------------------------

The distinguished nilpotent orbits and the corresponding groups $A_G(x)$ are:
\[
\begin{array}{|c||c|c|c|c|}
\hline
\cO & 
F_4  & 
F_4(a_1) & 
F_4(a_2) & 
F_4(a_3)
\\
\hline
A_G(x) &
1&
\fS_2 &
\fS_2 &
\fS_4
\\
\hline
\end{array}
\]
\noindent
The tables for $F_4$ are as follows.

\begin{table}[ht]
\[
\begin{array}{|c|c|c|}
\hline
(L, \cO_L, \cE_L)\in\fM_{G,\bk} & N_G(L)/L & |\fN_{G,\bk}^{(L,\cO_L,\cE_L)}|\\
\hline\hline
(T, \{0\}, \ubk) & W(F_4) & 4\\
\hline
(A_1, [2], \ubk) & W(B_3) & 2\\
\hline
(\tilde A_1, [2], \ubk) & W(B_3) & 2\\
\hline
(A_1 + \tilde A_1, [2] \times [2], \ubk) & \fS_2 \times \fS_2 & 1\\
\hline
(B_2, [5], \ubk) & W(B_2) & 1\\
\hline
(B_3, [7], \ubk) & \fS_2 & 1\\
\hline
(C_3, [6], \ubk) & \fS_2 & 1\\
(C_3, [4,2], \ubk) &  & 1\\
\hline \hline
\text{4 cuspidal pairs, incl.: } & 1& 4\times 1\\
(F_4,\ubk), (F_4(a_3),\ubk) && \\
\hline\hline
&& |\fN_{G,\bk}|=17 \\
\hline
\end{array}
\]
\caption{Induction series for $F_4$ when $\ell=2$}\label{tab:F4ell2}
\end{table}

\begin{table}[ht]
\[
\begin{array}{|c|c|c|}
\hline
(L, \cO_L, \cE_L)\in\fM_{G,\bk} & N_G(L)/L & |\fN_{G,\bk}^{(L,\cO_L,\cE_L)}|\\
\hline\hline
(T, \{0\}, \ubk) & W(F_4) & 14\\
\hline
(A_2, [3], \ubk) & W(G_2) & 4\\
\hline
(\tilde A_2, [3], \ubk) & W(G_2) & 4\\
\hline \hline
\text{3 cuspidal pairs, incl.: } & 1& 3\times 1\\
(F_4,\ubk), (F_4(a_3),\varepsilon) && \\
\hline\hline
&& |\fN_{G,\bk}|=25 \\
\hline
\end{array}
\]
\caption{Induction series for $F_4$ when $\ell=3$}\label{tab:F4ell3}
\end{table}

\begin{table}[ht]
\[
\begin{array}{|c|c|c|}
\hline
(L, \cO_L, \cE_L)\in\fM_{G,\bk} & N_G(L)/L & |\fN_{G,\bk}^{(L,\cO_L,\cE_L)}|\\
\hline\hline
(T, \{0\}, \ubk) & W(F_4) & 25\\
\hline \hline
\text{1 cuspidal pair: } (F_4(a_3),\varepsilon) & 1& 1\\
\hline\hline
&& |\fN_{G,\bk}|=26 \\
\hline
\end{array}
\]
\caption{Induction series for $F_4$ when $\ell>3$}\label{tab:F4ell>3}
\end{table}

\clearpage

%-------------------------------------------------------------
\subsection{Type $E_6$}
\label{ss:E6}
%-------------------------------------------------------------

The distinguished nilpotent orbits and the corresponding groups $A_G(x)$ are:
\[
\begin{array}{|c||c|c|c|c|c|c|c|c|c|c|c|}
\hline
\cO & 
E_6  & 
E_6(a_1) & 
E_6(a_3) \\
\hline
A_G(x) &
Z(G)&
Z(G)&
\fS_2\times Z(G) \\
\hline
\end{array}
\]
\noindent
The tables for $E_6$ are as follows.

\begin{table}[ht]
\[
\begin{array}{|c|c|c|}
\hline
(L, \cO_L, \cE_L)\in\fM_{G,\bk}^1 & N_G(L)/L & |\fN_{G,\bk}^{(L,\cO_L,\cE_L)}|\\
\hline\hline
(T, \{0\}, \ubk) & W(E_6) & 6\\
\hline
(A_1, [2], \ubk) & \fS_6 & 4\\
\hline
(2A_1, [2]^2, \ubk) & W(B_3) & 2\\
\hline
(3A_1, [2]^3, \ubk) & \fS_2 \times \fS_3 & 2\\
\hline
(A_3, [4], \ubk) & W(B_2) & 1\\
\hline
(A_3 + A_1, [4] \times [2], \ubk) & \fS_2 & 1\\
\hline
(D_4, [7,1], \ubk) & \fS_3 & 2\\
(D_4, [5,3], \ubk) &  & 2\\
\hline
(D_5, [9,1], \ubk) & 1 & 1\\
(D_5, [7,3], \ubk) & & 1\\
\hline\hline
\text{No cuspidal pairs } & & \\
\hline\hline
&& |\fN_{G,\bk}^1|=22 \\
\hline
\end{array}
\]
\caption{Induction series for $E_6$ when $\ell=2$ and $\chi=1$}\label{tab:E6ell2chi1}
\end{table}

\begin{table}[ht]
\[
\begin{array}{|c|c|c|}
\hline
(L, \cO_L, \cE_L)\in\fM_{G,\bk}^\chi & N_G(L)/L & |\fN_{G,\bk}^{(L,\cO_L,\cE_L)}| \\
\hline\hline
(2 A_2, [3]^2, \cE_\chi) & W(G_2) & 2\\
\hline
(2 A_2 + A_1, [3]^2\times[2], \cE_\chi\boxtimes\ubk) & \fS_2 & 1\\
\hline
(A_5, [6], \cE_\chi) & \fS_2 & 1\\
\hline\hline
\text{2 cuspidal pairs, incl.\ } (E_6(a_3),1\times\chi) & 1& 2\times 1\\
\hline\hline
&& |\fN_{G,\bk}^\chi|=6 \\
\hline
\end{array}
\]
\caption{Induction series for $E_6$ when $\ell=2$ and $\chi\neq1$}\label{tab:E6ell2chinot1}
\end{table}

\begin{table}[ht]
\[
\begin{array}{|c|c|c|c|}
\hline
(L, \cO_L, \cE_L)\in\fM_{G,\bk} & N_G(L)/L & |\fN_{G,\bk}^{(L,\cO_L,\cE_L)}|\\
\hline\hline
(T, \{0\}, \ubk) & W(E_6) & 12\\
\hline
(A_2, [3], \ubk) & \fS_3^2 \rtimes \fS_2 & 5\\
\hline
(2 A_2, [3]^2, \ubk) & W(G_2) & 4\\
\hline\hline
\text{3 cuspidal pairs: } & 1 & 3\times 1\\
(E_6,\ubk), (E_6(a_1),\ubk), (E_6(a_3),\varepsilon) &&\\
\hline\hline
&& |\fN_{G,\bk}|=24 \\
\hline
\end{array}
\]
\caption{Induction series for $E_6$ when $\ell=3$}\label{tab:E6ell3}
\end{table}

\begin{table}[ht]
\[
\begin{array}{|c|c|c|}
\hline
(L, \cO_L, \cE_L)\in\fM_{G,\bk}^1 & N_G(L)/L & |\fN_{G,\bk}^{(L,\cO_L,\cE_L)}|\\
\hline\hline
(T, \{0\}, \ubk) & W(E_6) & 23\\
\hline
(A_4, [5], \ubk) & \fS_2 & 2\\
\hline\hline
\text{No cuspidal pairs } & & \\
\hline\hline
&& |\fN_{G,\bk}^1|=25 \\
\hline
\end{array}
\]
\caption{Induction series for $E_6$ when $\ell=5$ and $\chi=1$}\label{tab:E6ell5chi1}
\end{table}

\begin{table}[ht]
\[
\begin{array}{|c|c|c|}
\hline
(L, \cO_L, \cE_L)\in\fM_{G,\bk}^\chi & N_G(L)/L & |\fN_{G,\bk}^{(L,\cO_L,\cE_L)}| \\
\hline\hline
(2 A_2, [3]^2, \cE_\chi) & W(G_2) & 6\\
\hline\hline
\text{1 cuspidal pair: } (E_6(a_3), \varepsilon \times \chi) & 1 & 1 \\
\hline\hline
&& |\fN_{G,\bk}^\chi|=7 \\
\hline
\end{array}
\]
\caption{Induction series for $E_6$ when $\ell>3$ and $\chi\neq1$}\label{tab:E6ell>3chinot1}
\end{table}

\begin{table}[ht]
\[
\begin{array}{|c|c|c|}
\hline
(L, \cO_L, \cE_L)\in\fM_{G,\bk}^1 & N_G(L)/L & |\fN_{G,\bk}^{(L,\cO_L,\cE_L)}|\\
\hline\hline
(T, \{0\}, \ubk) & W(E_6) & 25\\
\hline\hline
\text{No cuspidal pairs } & & \\
\hline\hline
&& |\fN_{G,\bk}^1|=25 \\
\hline
\end{array}
\]
\caption{Induction series for $E_6$ when $\ell>5$ and $\chi=1$}\label{tab:E6ell>5chi1}
\end{table}

\clearpage

%-------------------------------------------------------------
\subsection{Type $E_7$}
\label{ss:E7}
%-------------------------------------------------------------

The distinguished nilpotent orbits and the corresponding groups $A_G(x)$ are:
\[
\begin{array}{|c||c|c|c|c|c|c|c|c|c|c|c|}
\hline
\cO & 
E_7  & 
E_7(a_1) & 
E_7(a_2) & 
E_7(a_3) & 
E_7(a_4) & 
E_7(a_5)\\
\hline
A_G(x) &
Z(G)&
Z(G)&
Z(G)&
\fS_2\times Z(G) &
\fS_2\times Z(G) &
\fS_3\times Z(G) \\
\hline
\end{array}
\]
\noindent
The tables for $E_7$ are as follows.

\begin{table}[ht]
\[
\begin{array}{|c|c|c|}
\hline
(L, \cO_L, \cE_L)\in\fM_{G,\bk} & N_G(L)/L & |\fN_{G,\bk}^{(L,\cO_L,\cE_L)}| \\
\hline\hline
(T,\{0\}, \ubk) & W(E_7) & 8\\
\hline
(A_1, [2], \ubk) & W(D_6) & 4\\
\hline
(2A_1, [2]^2, \ubk) & W(B_4) \times \fS_2 & 2\\
\hline
((3A_1)', [2]^3, \ubk) & W(C_3) \times \fS_2 & 2\\
\hline
((3A_1)'', [2]^3, \ubk) & W(F_4) & 4\\
\hline
(A_3, [4], \ubk) & W(B_3) \times \fS_2 & 2\\
\hline
(4A_1, [2]^4, \ubk) & W(B_3) & 2\\
\hline
((A_3+A_1)', [4]\times[2], \ubk) & \fS_2\times\fS_2\times\fS_2 & 1\\
\hline
((A_3+A_1)'', [4]\times[2], \ubk) & W(B_3) & 2\\
\hline
(D_4, [7,1], \ubk) & W(C_3) & 2\\
(D_4, [5,3], \ubk) & & 2\\
\hline
(A_3 + 2A_1, [4]\times[2]^2, \ubk) & \fS_2\times\fS_2 & 1\\
\hline
(D_4 + A_1, [7,1] \times [2], \ubk) & W(B_2) & 1\\
(D_4 + A_1, [5,3] \times [2], \ubk) & & 1\\
\hline
(D_5, [9,1], \ubk) & \fS_2\times\fS_2 & 1\\
(D_5, [7,3], \ubk) & & 1\\
\hline
(D_5 + A_1, [9,1]\times[2], \ubk) & \fS_2 & 1\\
(D_5 + A_1, [7,3]\times[2], \ubk) & & 1\\
\hline
(D_6, [11,1], \ubk) & \fS_2 & 1\\
(D_6, [9,3], \ubk) & & 1\\
(D_6, [7,5], \ubk) & & 1\\
\hline\hline
\text{6 cuspidal pairs, incl.: } & 1 & 6\times 1\\
(E_7,\ubk), (E_7(a_5),\ubk) &&\\
\hline\hline
&& |\fN_{G,\bk}|=47 \\
\hline
\end{array}
\]
\caption{Induction series for $E_7$ when $\ell=2$}\label{tab:E7ell2}
\end{table}

\begin{table}[ht]
\[
\begin{array}{|c|c|c|}
\hline
(L, \cO_L, \cE_L)\in\fM_{G,\bk}^1 & N_G(L)/L & |\fN_{G,\bk}^{(L,\cO_L,\cE_L)}| \\
\hline\hline
(T, \{0\}, \ubk) & W(E_7) & 30\\
\hline
(A_2, [3], \ubk) & \fS_6 \times \fS_2 & 14\\
\hline
(2A_2, [3]^2, \ubk) & W(G_2) \times \fS_2 & 8\\
\hline
(E_6, E_6, \ubk) & \fS_2 & 2\\
(E_6, E_6(a_1), \ubk) & & 2\\
(E_6, E_6(a_3), \varepsilon) & & 2\\
\hline\hline
\text{No cuspidal pairs } & & \\
\hline\hline
&& |\fN_{G,\bk}^1|=58 \\
\hline
\end{array}
\]
\caption{Induction series for $E_7$ when $\ell=3$ and $\chi=1$}\label{tab:E7ell3chi1}
\end{table}

\begin{table}[ht]
\[
\begin{array}{|c|c|c|}
\hline
(L, \cO_L, \cE_L)\in\fM_{G,\bk}^\chi & N_G(L)/L & |\fN_{G,\bk}^{(L,\cO_L,\cE_L)}| \\
\hline\hline
((3A_1)'', [2]^3, \cE_\chi) & W(F_4) & 14\\
\hline
(A_2+3A_1, [3]\times[2]^2, \ubk\boxtimes\cE_\chi) & W(G_2) & 4\\
\hline
((A_5)'', [6], \cE_\chi) & W(G_2) & 4\\ 
\hline\hline
\text{3 cuspidal pairs, incl.\ } (E_7(a_5),\varepsilon\times\chi) & 1& 3\times 1\\
\hline\hline
&& |\fN_{G,\bk}^\chi|=25 \\
\hline
\end{array}
\]
\caption{Induction series for $E_7$ when $\ell=3$ and $\chi\neq1$}\label{tab:E7ell3chinot1}
\end{table}

\begin{table}[ht]
\[
\begin{array}{|c|c|c|}
\hline
(L, \cO_L, \cE_L)\in\fM_{G,\bk}^1 & N_G(L)/L & |\fN_{G,\bk}^{(L,\cO_L,\cE_L)}| \\
\hline\hline
(T, \{0\}, \ubk) & W(E_7) & 54\\
\hline
(A_4, [5], \ubk) & \fS_3 \times \fS_2 & 6\\
\hline\hline
\text{No cuspidal pairs } & & \\
\hline\hline
&& |\fN_{G,\bk}^1|=60 \\
\hline
\end{array}
\]
\caption{Induction series for $E_7$ when $\ell=5$ and $\chi=1$}\label{tab:E7ell5chi1}
\end{table}

\begin{table}[ht]
\[
\begin{array}{|c|c|c|}
\hline
(L, \cO_L, \cE_L)\in\fM_{G,\bk}^\chi & N_G(L)/L & |\fN_{G,\bk}^{(L,\cO_L,\cE_L)}| \\
\hline\hline
((3A_1)'', [2]^3, \cE_\chi) & W(F_4) & 25\\
\hline\hline
\text{1 cuspidal pair: } (E_7(a_5), \varepsilon \times \chi) & 1 & 1 \\
\hline\hline
&& |\fN_{G,\bk}^\chi|=26 \\
\hline
\end{array}
\]
\caption{Induction series for $E_7$ when $\ell>3$ and $\chi\neq1$}\label{tab:E7ell>3chinot1}
\end{table}

\begin{table}[ht]
\[
\begin{array}{|c|c|c|}
\hline
(L, \cO_L, \cE_L)\in\fM_{G,\bk}^1 & N_G(L)/L & |\fN_{G,\bk}^{(L,\cO_L,\cE_L)}| \\
\hline\hline
(T, \{0\}, \ubk) & W(E_7) & 58\\
\hline
(A_6, [7], \ubk) & \fS_2 & 2\\
\hline\hline
\text{No cuspidal pairs } & & \\
\hline\hline
&& |\fN_{G,\bk}^1|=60 \\
\hline
\end{array}
\]
\caption{Induction series for $E_7$ when $\ell=7$ and $\chi=1$}\label{tab:E7ell7chi1}
\end{table}

\begin{table}[ht]
\[
\begin{array}{|c|c|c|}
\hline
(L, \cO_L, \cE_L)\in\fM_{G,\bk}^1 & N_G(L)/L & |\fN_{G,\bk}^{(L,\cO_L,\cE_L)}| \\
\hline\hline
(T, \{0\}, \ubk) & W(E_7) & 60\\
\hline\hline
\text{No cuspidal pairs } & & \\
\hline\hline
&& |\fN_{G,\bk}^1|=60 \\
\hline
\end{array}
\]
\caption{Induction series for $E_7$ when $\ell>7$ and $\chi=1$}\label{tab:E7ell>7chi1}
\end{table}

\clearpage

%-------------------------------------------------------------
\subsection{Type $E_8$}
\label{ss:E8}
%-------------------------------------------------------------

The distinguished nilpotent orbits and the corresponding groups $A_G(x)$ are:
\[
\begin{array}{|c||c|c|c|c|c|c|c|c|c|c|c|}
\hline
\cO & 
E_8  & 
E_8(a_1) & 
E_8(a_2) & 
E_8(a_3) & 
E_8(a_4) & 
E_8(b_4) & 
E_8(a_5) & 
E_8(b_5) & 
E_8(a_6) & 
E_8(b_6) & 
E_8(a_7) 
\\
\hline
A_G(x) &
1&
1&
1&
\fS_2 &
\fS_2 &
\fS_2 &
\fS_2 &
\fS_3 &
\fS_3 &
\fS_3 &
\fS_5
\\
\hline
\end{array}
\]
\noindent
The tables for $E_8$ are as follows.

\begin{table}[ht]
\[
\begin{array}{|c|c|c|}
\hline
(L, \cO_L, \cE_L)\in\fM_{G,\bk} & N_G(L)/L & |\fN_{G,\bk}^{(L,\cO_L,\cE_L)}|\\
\hline\hline
(T, \{0\}, \ubk) & W(E_8) & 12\\
\hline
(A_1, [2], \ubk) & W(E_7) & 8\\
\hline
(2A_1, [2], \ubk) & W(B_6) & 4\\
\hline
(3A_1, [2]^3, \ubk) & W(F_4) \times \fS_2 & 4\\
\hline
(A_3, [4], \ubk) & W(B_5) & 3\\
\hline
(4A_1, [2]^4, \ubk) & W(B_4) & 2\\
\hline
(A_3 + A_1, [4] \times [2], \ubk) & W(B_3) \times \fS_2 & 2\\
\hline
(D_4, [7,1], \ubk) & W(F_4) & 4\\
(D_4, [5,3], \ubk) & & 4\\
\hline
(A_3 + 2A_1, [4] \times [2]^2, \ubk) & W(B_2) \times \fS_2 & 1\\
\hline
(D_4 + A_1, [7,1] \times [2], \ubk) & W(B_3) & 2\\
(D_4 + A_1, [5,3] \times [2], \ubk) & & 2\\
\hline
(D_5, [9,1], \ubk) & W(B_3) & 2\\
(D_5, [7,3], \ubk) & & 2\\
\hline
(2A_3, [4]^2, \ubk) & W(B_2) & 1\\
\hline
(D_5 + A_1, [9,1] \times [2], \ubk) & \fS_2\times\fS_2 & 1\\
(D_5 + A_1, [7,3] \times [2], \ubk) & & 1\\
\hline
(D_6, [11,1], \ubk) & W(B_2) & 1\\
(D_6, [9,3], \ubk) & & 1\\
(D_6, [7,5], \ubk) & & 1\\
\hline
(A_7, [8], \ubk) & \fS_2 & 1\\
\hline
(D_7, [13,1], \ubk) & \fS_2 & 1\\
(D_7, [11,3], \ubk) & & 1\\
(D_7, [9,5], \ubk) & & 1\\
\hline
(E_7,E_7,\ubk) & \fS_2 & 1\\
(E_7,E_7(a_5),\ubk) & & 1\\
(E_7, \text{4 other cuspidals}) & & 4 \times 1\\
\hline\hline
\text{10 cuspidal pairs, incl.:} & 1 & 10\times 1\\
(E_8,\ubk), (E_8(a_7),\ubk) &  & \\
\hline
\hline
&& |\fN_{G,\bk}|=78 \\
\hline
\end{array}
\]
\caption{Induction series for $E_8$ when $\ell=2$}\label{tab:E8ell2}
\end{table}

\begin{table}[ht]
\[
\begin{array}{|c|c|c|}
\hline
(L, \cO_L, \cE_L)\in\fM_{G,\bk} & N_G(L)/L & |\fN_{G,\bk}^{(L,\cO_L,\cE_L)}|\\
\hline\hline
(T, \{0\}, \ubk) & W(E_8) & 47\\
\hline
(A_2, [3], \ubk) & W(E_6) \times \fS_2 & 24\\
\hline
(2 A_2, [3]^2, \ubk) & W(G_2)^2 \rtimes \fS_2 & 14\\
\hline
(E_6, E_6, \ubk) & W(G_2) & 4\\
(E_6, E_6(a_1), \ubk) & & 4\\
(E_6, E_6(a_3),\varepsilon) & & 4\\
\hline\hline
\text{8 cuspidal pairs, incl.:} & 1 & 8\times 1\\
(E_8,\ubk), (E_8(a_7),\varepsilon) &  & \\
\hline\hline
&&|\fN_{G,\bk}|=105\\
\hline
\end{array}
\]
\caption{Induction series for $E_8$ when $\ell=3$}\label{tab:E8ell3}
\end{table}

\begin{table}[ht]
\[
\begin{array}{|c|c|c|}
\hline
(L, \cO_L, \cE_L)\in\fM_{G,\bk} & N_G(L)/L & |\fN_{G,\bk}^{(L,\cO_L,\cE_L)}|\\
\hline\hline
(T, \{0\}, \ubk) & W(E_8) & 95\\
\hline
(A_4, [5], \ubk) & \fS_5\times\fS_2 & 12\\
\hline\hline
\text{5 cuspidal pairs, incl.:} & 1 & 5\times 1\\
(E_8,\ubk), (E_8(a_7),\varepsilon) &  & \\
\hline\hline
&&|\fN_{G,\bk}|=112\\
\hline
\end{array}
\]
\caption{Induction series for $E_8$ when $\ell=5$}\label{tab:E8ell5}
\end{table}

\begin{table}[ht]
\[
\begin{array}{|c|c|c|}
\hline
(L, \cO_L, \cE_L)\in\fM_{G,\bk} & N_G(L)/L & |\fN_{G,\bk}^{(L,\cO_L,\cE_L)}|\\
\hline\hline
(T, \{0\}, \ubk) & W(E_8) & 108\\
\hline
(A_6, [7], \ubk) & \fS_2\times\fS_2 & 4\\
\hline\hline
\text{1 cuspidal pair: } (E_8(a_7), \varepsilon) & 1 & 1\\
\hline\hline
&&|\fN_{G,\bk}|=113\\
\hline
\end{array}
\]
\caption{Induction series for $E_8$ when $\ell=7$}\label{tab:E8ell7}
\end{table}

\begin{table}[ht]
\[
\begin{array}{|c|c|c|}
\hline
(L, \cO_L, \cE_L)\in\fM_{G,\bk} & N_G(L)/L & |\fN_{G,\bk}^{(L,\cO_L,\cE_L)}|\\
\hline\hline
(T, \{0\}, \ubk) & W(E_8) & 112\\
\hline\hline
\text{1 cuspidal pair: } (E_8(a_7), \varepsilon) & 1 & 1\\
\hline\hline
&&|\fN_{G,\bk}|=113\\
\hline
\end{array}
\]
\caption{Induction series for $E_8$ when $\ell>7$}\label{tab:E8ell>7}
\end{table}

\clearpage

%%%%%%%%%%%%%%%%%%%%%%%%%%%%%%%%%%%%%%%%%%%%%%%%%%%%%%%%%%%%%%%%%%%%%%%%%%%


\begin{thebibliography}{AHJR2}

\bibitem[AHJR1]{genspring1}
P.~Achar, A.~Henderson, D.~Juteau, and S.~Riche, {\em Modular generalized
  Springer correspondence~I: the general linear group}, J. Eur. Math. Soc.
  (JEMS) {\bf 18} (2016), 1405--1436.

\bibitem[AHJR2]{genspring2}
P.~Achar, A.~Henderson, D.~Juteau and S.~Riche, \emph{Modular generalized Springer correspondence II: classical groups}, preprint  \href{http://arxiv.org/abs/1404.1096}{arXiv:1404.1096}, to appear in J.\ Eur.\ Math.\ Soc.

\bibitem[AM]{am}
P.~Achar and C.~Mautner, {\em Sheaves on nilpotent cones, Fourier transform,
  and a geometric Ringel duality}, Mosc. Math. J. {\bf 15} (2015), 407--423.
  
\bibitem[Bo1]{bonnafe1}
C.~Bonnaf\'e, \emph{Actions of relative Weyl groups. I}, J.\ Group Theory \textbf{7} (2004), no.~1, 1--37.  

\bibitem[Bo2]{bonnafe-elements}
C.~Bonnaf\'e, \emph{\'El\'ements unipotents r\'eguliers des sous-groupes de Levi}, Canad.\ J.\ Math.\ \textbf{56} (2004), no.~2, 246--276.

\bibitem[Ca]{carter}
R.~W.~Carter, {\em Finite groups of Lie type: conjugacy classes and complex
  characters}, John Wiley \& Sons, Inc., New York, 1985.

\bibitem[CM]{cm}
D.~H. Collingwood and W.~M. McGovern, {\em Nilpotent orbits in semisimple {L}ie
  algebras}, Van Nostrand Reinhold Co., New York, 1993.

\bibitem[DM]{dignemichel}
F.~Digne and J.~Michel, \emph{Representations of finite groups of Lie type}, London Mathematical Society Student Texts {\bf 21}, Cambridge University Press, Cambridge, 1991.

\bibitem[EM]{em}
S.~Evens and I.~Mirkovi\'{c}, {\em Fourier transform and the Iwahori--Matsumoto involution}, Duke Math. J. {\bf 86} (1997), 435--464.

\bibitem[GHM1]{ghm}
M.~Geck, G.~Hiss and G.~Malle, \emph{Cuspidal unipotent Brauer characters}, J.\ Algebra \textbf{168} (1994), no.~1, 182--220.

\bibitem[GHM2]{ghm2}
M.~Geck, G.~Hiss, and G.~Malle, {\em Towards a classification of the
  irreducible representations in non-describing characteristic of a finite
  group of Lie type}, Math. Z. {\bf 221} (1996), 353--386.

\bibitem[Gr]{green}
J.~A.~Green, \emph{On the indecomposable representations of a finite group}, Math.\ Z.\ \textbf{70} (1958/59), 430--445.

\bibitem[Ho]{howlett}
R.~B.~Howlett, {\em Normalizers of parabolic subgroups of reflection groups}, J.\ Lond.\ Math.\ Soc.\ (2) {\bf 21} (1980), no.~1, 62--80.

\bibitem[Ju]{juteau}
D.~Juteau, \emph{Modular Springer correspondence, decomposition matrices and basic sets}, preprint  \href{http://arxiv.org/abs/1410.1471}{arXiv:1410.1471}, to appear in Bull. Soc. Math. France.

\bibitem[JLS]{jls}
D.~Juteau, C.~Lecouvey and K.~Sorlin, {\em Springer basic sets and modular Springer correspondence for classical types}, preprint \href{http://arxiv.org/abs/1410.1477}{arXiv:1410.1477}.

\bibitem[JMW]{jmw}
D.~Juteau, C.~Mautner and G.~Williamson, \emph{Parity sheaves}, J. Amer. Math. Soc.\ \textbf{27} (2014), no.~4, 1169--1212.

\bibitem[KW]{kw}
R.~Kiehl and R.~Weissauer, \emph{Weil conjectures, perverse sheaves and $l$'adic Fourier transform}, Ergebnisse der Mathematik und ihrer Grenzgebiete, 3rd Series, vol.~42, Springer--Verlag, Berlin, 2001.

\bibitem[Lu1]{lusztig}
G.~Lusztig, \emph{Intersection cohomology complexes on a reductive group},
Invent.~Math.~\textbf{75} (1984), 205--272.

\bibitem[Lu2]{charsh1}
G.~Lusztig, \emph{Character sheaves I}, Adv.\ in Math.\ {\bf 56} (1985), no.~3, 193--237.

\bibitem[Lu3]{charsh3}
G.~Lusztig, \emph{Character sheaves III}, Adv.\ in Math.\ {\bf 57} (1985), no.~3, 266--315.

\bibitem[Lu4]{charsh5}
G.~Lusztig, \emph{Character sheaves V}, Adv.\ in Math.\ {\bf 61} (1986), no.~2, 103--155.

\bibitem[Lu5]{lusztig-cusp2}
G.~Lusztig, \emph{Cuspidal local systems and graded Hecke algebras. II}, Representations of groups (Banff, AB, 1994), CMS Conf.~Proc., vol.~16, Amer.~Math.~Soc., Providence, RI, 1995, 217--275.

\bibitem[Lu6]{lusztig-csdg5}
G.~Lusztig, \emph{Character sheaves on disconnected groups. V}, Represent.~Theory~\textbf{8} (2004), 346--376.

\bibitem[Lu7]{lusztig-gsc}
G.~Lusztig, \emph{On the generalized Springer correspondence}, preprint \href{http://arxiv.org/abs/1608.02223}{arXiv:1608.02223}.

\bibitem[LS]{lus-spalt}
G.~Lusztig and N.~Spaltenstein, {\em On the generalized Springer correspondence
  for classical groups}, Algebraic groups and related topics (Kyoto/Nagoya,
  1983), Adv. Stud. Pure Math., vol.~6, North-Holland Publishing Co.,
  Amsterdam, 1985, 289--316.

\bibitem[MS]{marsspringer}
J.~G.~M.~Mars and T.~A.~Springer, \emph{Character sheaves}, Ast\'erisque~\textbf{173--174} (1989), 111--198.

\bibitem[Mi]{chevie}
J.~Michel, \emph{The development version of the CHEVIE package of GAP3}, preprint \href{http://arxiv.org/abs/1310.7905}{arXiv:1310.7905}.

\bibitem[Spa]{spaltenstein}
N.~Spaltenstein, {\em On the generalized Springer correspondence for exceptional groups},
Algebraic groups and related topics (Kyoto/Nagoya,
  1983), Adv. Stud. Pure Math., vol.~6, North-Holland Publishing Co.,
  Amsterdam, 1985, 317--338.

\end{thebibliography}
\end{document}